\documentclass[a4paper,11pt,oneside,reqno]{amsart}

\usepackage[english]{babel}
\usepackage[utf8x]{inputenc}    
\usepackage[dvipsnames]{xcolor}            
\usepackage{xspace}
\usepackage[pdftex]{graphicx}
\usepackage{epstopdf}
\usepackage{mathrsfs}
\usepackage{stmaryrd}
\usepackage{enumitem}	
\usepackage{verbatim}
\usepackage{amssymb}		
\usepackage{csquotes}	
\usepackage[a4paper,scale={0.72,0.74},marginratio={1:1},footskip=7mm,headsep=10mm]{geometry}	

\usepackage{subcaption}	

\makeatletter		
\newcommand{\leqnos}{\tagsleft@true\let\veqno\@@leqno}
\newcommand{\reqnos}{\tagsleft@false\let\veqno\@@eqno}
\reqnos
\makeatother

\numberwithin{equation}{section}	

\usepackage{caption}           
\usepackage{url}

\usepackage{hyperref}
\hypersetup{					
	breaklinks=true,			
	colorlinks,
    citecolor=blue,
    filecolor=black,
    linkcolor=blue,
    urlcolor=black
}

\newcommand{\ind}{{\sf 1}}

\newcommand{\bP}{\mathbf{P}}
\newcommand{\Pb}{\mathbf{P}}

\newcommand{\bE}{\mathbf{E}}
\newcommand{\Eb}{\mathbf{E}}

\newcommand{\bbP}{\mathbb{P}}
\newcommand{\bbE}{\mathbb{E}}
\newcommand{\Pbb}{\Pb_{\gb}}			
\newcommand{\Ebb}{\Eb_{\gb}}			
\newcommand{\Pbbx}{\Pb_{\gb,x}}			
\newcommand{\Ebbx}{\Eb_{\gb,x}}			
\newcommand{\Pbbzero}{\Pb_{\gb,0}}			
\newcommand{\Ebbzero}{\Eb_{\gb,0}}			
\newcommand{\Pbby}{\Pb_{\gb,y}}			
\newcommand{\Ebby}{\Eb_{\gb,y}}			
\newcommand{\Pbbz}{\Pb_{\gb,z}}			

\newcommand{\R}{\mathbb{R}}
\newcommand{\bbN}{\mathbb{N}}
\newcommand{\N}{\mathbb{N}}
\newcommand{\bbZ}{\mathbb{Z}}

\newcommand{\cA}{{\ensuremath{\mathcal A}}}
\newcommand{\cB}{{\ensuremath{\mathcal B}}}

\newcommand{\cC}{{\ensuremath{\mathcal C}}}

\newcommand{\cD}{{\ensuremath{\mathcal D}}}
 
\newcommand{\cE}{{\ensuremath{\mathcal E}}}

\newcommand{\cL}{{\ensuremath{\mathcal L}}}
\newcommand{\cM}{{\ensuremath{\mathcal M}}}

\newcommand{\cO}{{\ensuremath{\mathcal O}}}

\newcommand{\cQ}{{\ensuremath{\mathcal Q}}}

\renewcommand{\phi}{\varphi}

\newcommand{\ga}{\alpha}
\newcommand{\gb}{\beta}
\newcommand{\gga}{\gamma} 			

\newcommand{\gGa}{\Gamma}
\newcommand{\gd}{\delta}

\newcommand{\gep}{\varepsilon}		
\newcommand{\eps}{\varepsilon}

\newcommand{\gz}{\zeta}
\newcommand{\gh}{\eta}

\newcommand{\gk}{\kappa}
\newcommand{\gl}{\lambda}
\newcommand{\gL}{\Lambda}

\newcommand{\gr}{\rho}
\newcommand{\gs}{\sigma}

\newcommand{\gO}{\Omega}

\newcommand{\bh}{\textbf{\textit{h}}}

\newcommand{\pt}{\hspace{1pt}}

\newcounter{cst}[section]		
\newcounter{svf}[section]		
\newcommand{\cntc}{{\stepcounter{cst}\arabic{cst}}}		

\newtheorem{theorem}{Theorem}[section]

\newtheorem{proposition}[theorem]{Proposition}
\newtheorem{prop}[theorem]{Proposition}

\newtheorem{corol}[theorem]{Corollary}
\newtheorem{lemma}[theorem]{Lemma}
\newtheorem{lem}[theorem]{Lemma}
\newtheorem{claim}[theorem]{Claim}

\theoremstyle{definition}

\newtheorem{remark}[theorem]{Remark}

\newtheorem{rmq}[theorem]{Remark}

\numberwithin{equation}{section}			

\newcommand{\dd}{d}		

\newcommand{\sumtwo}[2]{\sum_{\substack{#1 \\ #2}}}

\renewcommand{\preceq}{\preccurlyeq}		

\renewcommand{\hat}{\widehat}
\renewcommand{\tilde}{\widetilde}
\newcommand{\ol}{\overline}
\newcommand{\ul}{\underline}
\newcommand{\uN}{{u_N}}
\newcommand{\twg}{\,\tilde\wedge\,}

\newcommand{\wet}{\mathrm{wet}}

\title[Surface transition in the collapsed phase of the IPDSAW along a hard wall]{Surface transition in the collapsed phase of 
a self-interacting walk  adsorbed along a hard wall}

\author{Alexandre Legrand}
\address{Universit\'e de Nantes, Laboratoire Jean Leray, 2, rue de la Houssini\`ere,44322 Nantes cedex 3, France}
\email{alexandre.legrand@univ-nantes.fr}

\author{Nicolas P\'etr\'elis}

\address{Universit\'e de Nantes, Laboratoire Jean Leray, 2, rue de la Houssini\`ere,44322 Nantes cedex 3, France}

\email{nicolas.petrelis@univ-nantes.fr}

\subjclass[]{Primary 60K35; Secondary 82B41}

\keywords{Polymer collapse, wetting, surface transition, large deviations, Brownian meander}

\definecolor{verde}{RGB}{0,146,70}

\begin{document}

\begin{abstract}

The present paper is dedicated to the 2-dimensional Interacting Partially Directed Self Avoiding Walk constrained to remain in the upper-half plane and interacting with the horizontal axis.
The model has been 
introduced in \cite{F90} to investigate the 
behavior of a homopolymer dipped in a poor solvent and adsorbed along a horizontal hard wall. It is known to undergo a \emph{collapse} transition between an \emph{extended} phase, 
inside which typical configurations of the polymer have a large horizontal extension (comparable to their total size), and a \emph{collapsed} phase inside which the polymer looks like a globule.

In the present paper, we establish rigorously that inside the collapsed phase, a \emph{surface} transition occurs between an \emph{adsorbed-collapsed} regime where the bottommost layer of the globule is pinned at the hard wall, and a \emph{desorbed-collapsed} regime where the globule wanders away from the wall. 
To prove the existence of this surface transition and exhibit its associated critical curve,  we display some sharp asymptotics of the partition function for a slightly simplified version of the model.

\end{abstract}

\maketitle

\let\thefootnote\relax\footnote{{\it Acknowledgements.}  The authors thanks the Centre Henri Lebesgue ANR-11-LABX-0020-01 for creating an attractive mathematical environment.

The authors are grateful to Stuart Whittington and Quentin Berger for fruitful discussions.}

\section*{Notation}
Let $\N$ be the set of positive integers, and $\N_0:=\N\cup\{0\}$.
Let $(a_L)_{L\leq 1}$ and $(b_L)_{L\leq 1}$ be two sequences of positive numbers. We will write that 
\begin{equation}\label{defequiv}
a_L\sim_{L\to \infty} b_L \quad \text{if} \quad \lim_{L\to \infty} a_L/b_L=1,
\end{equation}
and also that 
\begin{equation}\label{defcomp}
a_L\asymp b_L \quad \text{if} \quad  c_1\,  b_L\leq a_L\leq c_2\,  b_L\quad  \forall L\geq 1
\end{equation}
with $c_1, c_2$ two positive constants.



\section{Introduction}

In the present paper, we investigate a model for a $1+1$ dimensional polymer dipped in a poor solvent and simultaneously 
adsorbed along a horizontal hard wall. Although the model has attracted a continuous attention in the physics literature starting in the 90's (see e.g. \cite{JJ91}, \cite{F90},  \cite{F91}) 
until more recently (see e.g. \cite{MGSY02}, \cite{RGSY02} or \cite{PMB17}), it had, up to our knowledge, not been considered so far in the mathematical literature.  This model  interpolates between two
families of polymer models that have been entirely solved in the last 20 years, i.e., the wetting of a $1+1$-dimensional random walk adsorbed along a hard wall 
(see e.g \cite{dH09}, \cite{Giac07} and \cite{Giac11}) and the collapse transition of the $2$-dimensional Interacting Partially-Directed Self-avoiding Walk (IPDSAW) (see \cite{CNPT18} for a review).


The coupling parameters of the model are $\beta\in [0,\infty)$ the repulsion intensity between the monomers and the solvent around them ---or, equivalently, the attraction intensity in between monomers--- and $\delta\in [0,\infty)$ the interaction intensity between the monomers and the hard wall.    We will discuss in detail the phase diagram of the model in Section \ref{PHD} below, but let us mention already that
the phase diagram is divided into two main phases:
\begin{itemize}
\item $\cE$: an Extended phase inside which a typical trajectory has a finite vertical width and a macroscopic horizontal extension,
\item $\cC$: a Collapsed phase inside which the vertical width and the horizontal extension of a typical configuration are comparable.
\end{itemize}
It turns out that $\cE$ can be divided into two sub-phases. A critical curve is indeed conjectured to partition
$\cE$ into a Desorbed-Extended phase ($\cD\cE$) inside which the polymer wanders away from the hard wall 
and an Adsorbed-Extended  phase ($\cA\cE$) inside which the polymer is localized along the wall (see e.g. \cite[Figure 2]{F90}). The situation is more subtle
in the Collapsed phase where typical configurations look roughly like a globule (see Fig. \ref{Fig:phase-diag} (B)). The number of contact between this globule and the hard-wall changes drastically inside~$\cC$ along some other critical curve which triggers what physicists call a surface transition, that is a loss of analyticity of the second order term of the exponential development of the partition function, whereas the leading order term (i.e., the free energy) remains linear. 


%
%

The aim of our paper is to investigate the collapsed phase and in particular the surface transition mentioned above. To that aim, we will introduce in Section \ref{sec2} a simplified version of our model called the \emph{one-bead model}. In a few words (see Section \ref{onebt} for more details) every trajectory considered in our model can be decomposed into a family of sub-trajectories called beads. Those beads are typically of finite size in $\cE$ but are much larger inside $\cC$. We can even safely conjecture that inside $\cC$, a typical trajectory is made of a unique macroscopic bead (this is proven e.g. for the 2-dimensional IPDSAW in \cite{LP20}). For this reason, we will restrict the set of allowed paths to those forming only one bead. This restricted version of the model turns out to be more tractable and should share many features with its non-restricted counterpart.


Let us give a short outline of the paper. In Section \ref{pdia} below, we begin with a rigorous definition of the model and then we provide a qualitative description of its phase diagram. With Theorem \ref{Phase-diag} we identify rigorously the Collapsed phase~($\cC$) and the Extended phase~($\cE$).
 Section \ref{sec2} is dedicated to the definition of the single-bead version of the model.  Theorem~\ref{th:partfunbead}, which is the most important result of the paper, is stated in Section \ref{sec3}  and allows us to characterize 
 the surface transition with the help of sharp asymptotic developments of the partition function inside $\cC$. 
 We prove Theorem~\ref{th:partfunbead} (along with Corollary~\ref{corol:partfunbead}), in Section~\ref{sec:prth22}. We delay the proofs of Theorem~\ref{Phase-diag} and Proposition~\ref{Phase-diag:bead} to Section~\ref{sec:prth11} for they are quite standard (apart from the random walk representation introduced in Section~\ref{sec:prth22}). We then collect the proofs of technical estimates in Section~\ref{sec:tech}. Appendix~\ref{app:wetting} provides well-known results on the wetting model, and Appendix~\ref{app:FKG} displays a (conditional) {\rm FKG} inequality on random walks with distribution $\Pbb$ (defined in \eqref{def:Pgb} below).

 \section{Description of the model and phase diagram}\label{pdia}
 

\subsection{The model}
For a polymer of length $L\in \N$, the set of its allowed configurations is denoted by $\gO_L^{+}$ and consists of those trajectories of a $1+1$-dimensional self-avoiding random walk on $\mathbb{Z}^2$ taking unitary steps up, down and to the right and constrained to remain above the horizontal axis $y=0$. An alternative representation of such trajectories can be given by decomposing them according to their number of horizontal/rightward steps, and the length and orientation of the vertical stretches in between, i.e.,

\begin{equation}\label{def:gO+}
\gO_L^{+}:=\bigcup_{N\geq1} \cL_{N,L}^{+}:=\bigcup_{N\geq1}\bigg\{ (\ell_i)_{i=1}^{N}\in\bbZ^{N}\; ; \sum_{i=1}^{N}|\ell_i|=L-N\;, \sum_{i=1}^k\ell_i\geq0, \,\forall\pt k\leq N
\bigg\}.
\end{equation}

Henceforth, we will only use this latter representation and we note that  each vertical stretch is followed by a horizontal step ---in particular we assume that all trajectories end with a horizontal step. 
For every $\ell\in \Omega_L^+$, we  denote by $N_\ell$ its horizontal extension (i.e.,  its number of horizontal step) so that  $\ell\in \cL^+_{N_\ell,L}$.

\begin{figure}[h]
 \centering
 \includegraphics[width=0.7\linewidth]{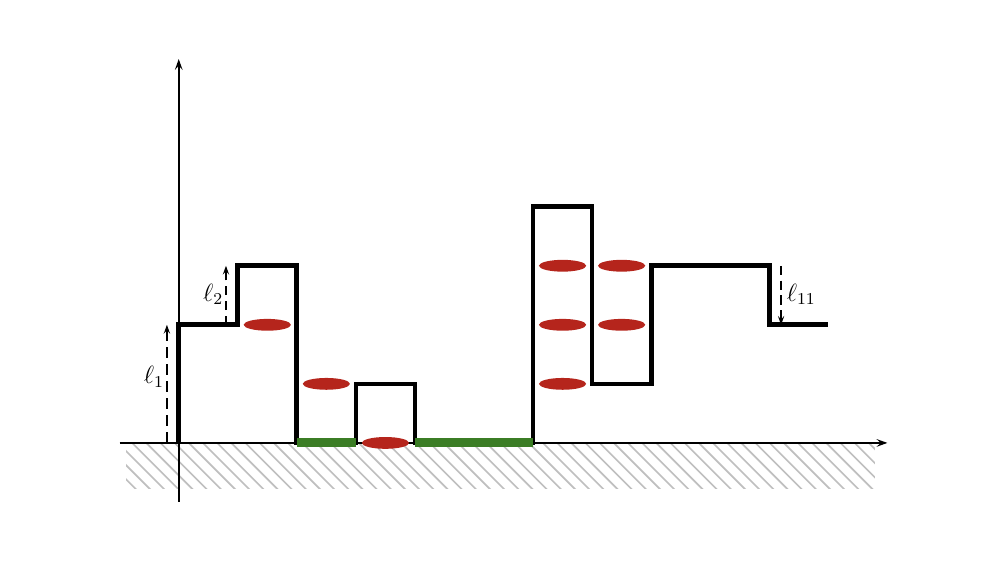}
 \caption{\footnotesize Representation of a polymer of length $L=29$, with horizontal extension $N=11$. Each self-touching (in red) is rewarded with an energy $\gb$, and each contact with the wall (in green) is rewarded with $\delta$.}
 \label{fig:polymer}
\end{figure}

With each configuration we associate a Hamiltonian, which takes into account that monomers are both attracting each other  and 
 adsorbed along the $x$-axis. To be more specific, a given $\ell\in \gO_L^+$  is assigned an energetic reward $\beta\geq 0$ for every {\emph{self-touching}} (i.e., a pair of neighboring sites visited non-consecutively by $\ell$) and 
 an energetic reward $\delta\geq 0$ for every contact with the $x$-axis (see Fig.~\ref{fig:polymer}). Thus, define
\begin{equation}\label{def:Ham}
H(\ell)\;:=\;\gb \sum_{i=0}^{N_\ell} \ell_i\twg \ell_{i+1}\,+\,\gd\sum_{k=1}^{N_\ell}\ind_{\{\sum_{i=1}^k\ell_i=0\}}
\end{equation} 
where the operator $\twg$ is defined for any $x,y\in\bbZ$ by $x\twg y:=\min\{|x|,|y|\} \ind_{\{xy\leq 0\}}$, and where we set $\ell_0=\ell_{N_\ell+1}=0$ for notational convenience. 
At this stage we introduce the polymer measure $\bP_{L}^{\beta,\delta}$, a probability on $\Omega_L^+$ defined as 
\begin{equation}\label{def:polmes}
\bP_{L}^{\beta,\delta}(\ell)=\frac{e^{{H(\ell)}}}{Z_{L,\gb,\gd}^{+}}, \quad \ell \in \Omega_L^+,
\end{equation}
where $Z_{L,\gb,\gd}^{+}$ is a normalisation term called \emph{partition function} of the system. The \emph{free energy} provides the exponential growth rate of $Z_{L,\gb,\gd}^{+}$ in $L$. It is defined as  
\begin{equation}\label{def:freeen}
f(\beta,\delta)=\lim_{L\to\infty} \frac{1}{L} \log Z_{L,\gb,\gd}^{+}\;,
\end{equation}
(we will show that $f$ is well-defined in the proof of Theorem~\ref{Phase-diag}).

\begin{remark}\label{compasm}
The present model may be seen as an advanced version of IPDSAW, that was introduced in \cite{ZL68}.  For the latter model, there is no hard-wall preventing the 
polymer to enter the lower half-plane and also no wetting interaction with the hard-wall. As a consequence, the allowed configurations of IPDSAW are obtained by relaxing the 
constraint $\sum_{i=1}^k\ell_i\geq 0$ in \eqref{def:gO+} and the Hamiltonian by removing the term $\gd\sum_{k=1}^{N}\ind_{\{\sum_{i=1}^k\ell_i=0\}}$  in \eqref{def:Ham}.
\end{remark}


\subsection{Phase diagram}\label{PHD}
Since the coupling parameters $\beta$ and $\delta$ are both non-negative, the phase diagram is drawn on the first quadrant $\cQ:=[0,\infty)^2$. Similarly to what is observed for IPDSAW (see e.g. \cite{NGP13} or \cite{CNPT18}), the phase diagram
 can be divided into a collapsed phase $\cC$ inside which typical trajectories undergo a self-touching saturation and an extended phase~$\cE$ inside which the horizontal extension of a typical trajectory is comparable to its total size. To be more specific, inside $\cC$, we expect that  
a typical trajectory of length $L$ (i.e., $\ell$ sampled from $\bP_L^{\beta,\delta}$) satisfies $H(\ell)=\beta L+o(L)$. For this reason 
its horizontal extension must be small (i.e., $N_\ell=o(L)$) and its vertical stretches should be long with alternating signs. In~$\cE$, in turn, a typical trajectory is expected to be composed of $O(L)$ vertical stretches of finite length (which is also what is expected at $\beta=0$). 

Let us now briefly explain (with three simple observations)  why the free energy $f(\beta,\delta)$ is equal to $\beta$ in $\cC$. First, $f(\beta,\delta)\geq \beta$ for every $(\delta,\beta)\in \cQ$. This inequality derives from restricting the computation of $Z^+_{L,\beta,\delta}$ to a unique trajectory $\tilde \ell \in \cL_{\sqrt{L},L}$ given by 
\begin{equation}\label{deflti}
\tilde \ell_i=(-1)^{i-1} (\sqrt{L}-1)\quad \text{ for every} \quad i\in \{1,\dots,\sqrt{L}\}
\end{equation}
(we assume $\sqrt L\in \N$ for conciseness) so that 
$H(\tilde\ell)=\beta (\sqrt{L}-1)^2+ \delta (\sqrt{L}-1)$. Second, those trajectories in $\Omega_L$ that are performing a self-touching saturation are not many and therefore they do not carry any entropy.  Third, we mentioned above that such saturated trajectories are made of $o(L)$ stretches and a trajectory may touch the hard-wall at most once per vertical stretch (recall Fig.~\ref{fig:polymer}), hence their interactions with the wall cannot be numerous enough to contribute to the free energy. These three points are sufficient to understand why the free energy equals $\beta$ in $\cC$ and thus, it is natural to define the \emph{excess} free energy of the system as
\begin{equation}\label{excessfreeen}
\tilde f(\beta,\delta)=f(\beta,\delta)-\beta,
\end{equation}
which allows us to define the 
 \emph{extended} and the \emph{collapsed}  phases as  
\begin{align}
\notag \cC&:=\{(\beta,\delta)\in \cQ\colon \tilde f(\beta,\delta)=0\},\\\label{def:cCcE}
\cE&:=\cQ\setminus\cC=\{(\beta,\delta)\in \cQ\colon \tilde f(\beta,\delta)>0\}.
\end{align}
Before stating Theorem \ref{Phase-diag} below, 
we need to settle some notations. We define for any $\gb>0$ the following probability distribution on $\bbZ$:
\begin{equation}\label{def:Pgb}
\Pb_\gb(\,\cdot=k)=\frac{e^{-\frac{\gb}{2}|k|}}{c_\gb}\quad,\qquad c_\gb:=\sum_{k\in\bbZ} e^{-\frac{\gb}{2}|k|} = \frac{1+e^{-\gb/2}}{1-e^{-\gb/2}}.
\end{equation}
We consider a one-dimensional random walk $X:=(X_i)_{i\geq 0}$ starting from the origin and such that
 $(X_{i+1}-X_i)_{i\geq 0}$ is an  i.i.d. sequence of random variables with law $\Pb_\gb$. Then, we let $h_\beta(\delta)$ be the free energy of 
 the wetting model that consists of the random walk $X$ constrained to remain non-negative and to finish on the $x$-axis, and pinned at the origin by an
 energetic factor $\delta$, i.e., 
 \begin{equation}\label{wetmod}
 h_\beta(\delta)=\lim_{N\to \infty} \frac{1}{N} \log \Ebbzero\Big[e^{\delta \sum_{i=1}^N 1_{\{X_i=0\}}} 1_{\{X\in B_N^{0,+}\}}\Big], 
 \end{equation} 
where $B^{0,+}_N$ is the set of non-negative trajectories of length $N$ ending at $0$ ---more generally, define for all $y\geq0$,
\begin{equation}\label{def:By+}
B^{y,+}_n\;:=\; \big\{(X_k)_{k=1}^n \in\bbZ^{n}\,;\, X_n=y\,, X_k\geq 0 \;\forall\,1\leq k\leq n\big\}.
\end{equation}
An explicit formula for $h_\gb(\gd)$ is given in Appendix~\ref{app:wetting} (see \eqref{eq:explicit:h}). 
We also define $\gGa_\gb:=c_\gb e^{-\gb}$ which is decreasing in $\gb$, and $\gb_c>0$ the unique solution of the equation $\gGa_\gb=1$.

\begin{theorem}\label{Phase-diag}
The boundary between the collapsed and the extended phase can be characterized explicitly, i.e., 
\begin{align}
\nonumber \cC&=\{(\beta,\delta)\in \cQ\colon \beta\geq \beta_c, \delta \leq \delta_c(\beta)\},
\end{align}
where, for every $\beta\geq \beta_c$, the quantity $\delta_c(\beta)$ is the unique solution in $\gd$ of 
\begin{equation}\label{eq:cc}
\log \Gamma_\beta+h_\beta(\delta)=0,
\end{equation}
which yields the following analytic expression,
\begin{equation}\label{eq:cc2}
\delta_c(\beta) = \log \bigg(  \frac{\sinh(\beta) + \sqrt{\sinh(\beta)^2 +1 - e^{\beta} }}{1-e^{-\beta}}\bigg)\,.
\end{equation}

\end{theorem}

\begin{remark}
Note that our formula for the critical curve in \eqref{eq:cc2} was already conjectured in  \cite[equation 19]{F91} or \cite[equation 21]{IF91} (both expressions 
coincide provided we set $\kappa=e^{\delta_c(\beta)}$ and $\tau=e^{\beta}$). In \cite{F91} and \cite{IF91}, the heuristics supporting this formula 
are based on some additional assumption and on a computation of the grand canonical with the help of transfer matrix. 
\end{remark}

\subsubsection{Discussion} 
Let us further explain the phenomenon behind the existence of a surface transition inside $\cC$ that physicists have conjectured  (see e.g.  \cite[Fig. 2]{MGSY02}).
As mentioned above, a typical trajectory in the collapsed phase looks like a globule delimited by a lower envelope and an upper envelope (see their rigorous definition in \eqref{defenv}).
For $L\in \N$, $\beta>\beta_c$ and $\delta <\delta_c(\beta)$, we will prove in Section \ref{S2} that the lower envelope of a trajectory sampled from $\bP_{L}^{\beta,\delta}$ behaves roughly as a random walk of length $O(\sqrt{L})$, constrained to remain non-negative,
and pinned at the $x$-axis (hard wall) with intensity $\delta$. This leaves us with a wetting model whose critical point $\tilde \delta_c(\beta)$ can be explicitly computed, see \eqref{def:deltatilde} (and it satisfies $\tilde \delta_c(\beta)< \delta_c(\beta)$). Thus, when $\delta\leq\tilde \delta_c(\beta)$ the lower-envelope touches the hard-wall only $o(\sqrt{L})$ times, whereas when $\delta>\tilde \delta_c(\beta)$ it remains localized along the hard wall and touches it $O(\sqrt L)$ times. 
As a consequence, this wetting transition of  the lower envelope is not encoded in the excess free energy $\tilde f(\beta,\delta)$ (which remains equal to $0$ in $\cC$) simply because the number of contact between the polymer (or equivalently its  lower envelope) and the hard-wall is at most  $O(\sqrt L)$. To be more specific,  we will see with  Theorem \ref{th:partfunbead} below that, in the exponential 
growth rate of $Z^+_{L,\beta,\delta}$, $\delta\in(\tilde \gd_c(\gb),\gd_c(\gb))$ only contributes to the second order term. We will further discuss the behavior of a typical lower envelope in~$\cC$ after stating Theorem~\ref{th:partfunbead}.

In order to display a qualitative picture of the phase diagram (see Fig.~\ref{Fig:phase-diag}), let us end this discussion with a few words about the extended 
regime $\cE$, where a typical trajectory of length $L$   is expected to have an horizontal extension of order $L$.  Physicists (see \cite[Fig.2]{F91} or \cite{IF91}) have conjectured that another critical curve $\hat\delta_c: [0,\beta_c]\to \R^+$ divides $\cE$ into a Desorbed-Extended phase denoted by $\cD\cE$ and an  Adsorbed-Extended phase denoted by 
$\cA\cE$ but there is so far no guess for what the value of $\hat \delta_c (\beta)$ could be.

\begin{figure}[h]
 \centering
  \begin{subfigure}[b]{0.4\linewidth}
    \includegraphics[width=\linewidth]{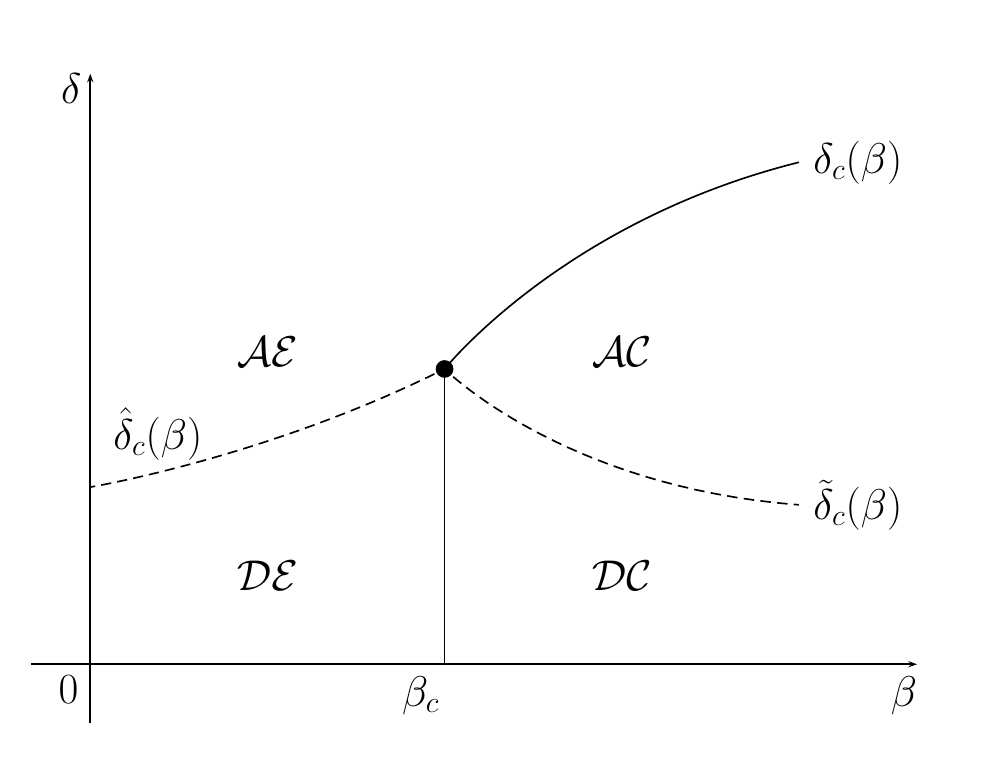}
    \caption{\footnotesize Phase diagram.}
  \end{subfigure}
  \begin{subfigure}[b]{0.45\linewidth}
    \includegraphics[width=\linewidth]{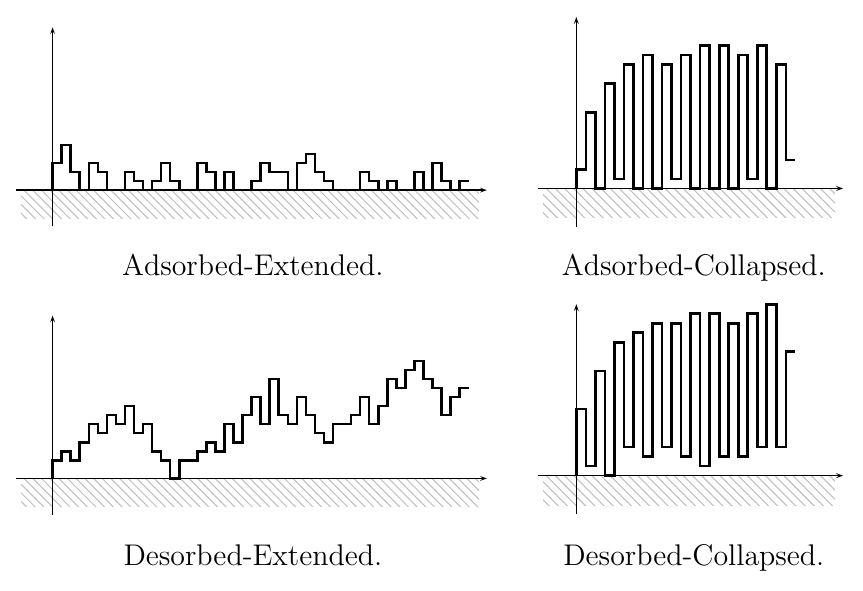}
    \caption{\footnotesize Geometry of the polymer.}
  \end{subfigure}
  \caption{\small Qualitative picture of the phase diagram and the geometry of the polymer in each phase.
 In this paper, we rigorously determine the critical line $\delta_c(\beta)$ between the extended and collapsed phases, and we characterize the surface transition critical line $\tilde  \delta_c(\beta)$.
No analytic expression has been conjectured for the critical line $\hat \delta_c(\beta)$ yet. 
  }
  \label{Fig:phase-diag}
\end{figure}

\begin{figure}
\end{figure}

\section{Inside the collapsed phase: restriction to the single-bead model. Asymptotics of the partition functions}\label{sec2}

\subsection{Bead decomposition of a trajectory}\label{onebt}

A trajectory of $\gO_L^{+}$ can be decomposed into a collection of sub-trajectories called beads. A bead is a succession of non-zero vertical stretches with alternating signs, which ends when two consecutive stretches have the same orientation, or when a stretch has length zero. To be more specific, we consider $\ell\in \Omega_L^+$ and we recall that $N_\ell$ is its horizontal extension and that by convention $\ell_0=\ell_{N_\ell+1}=0$. 
Then, we set  $x_0=0$ and for $j\in \N$ such that $x_{j-1}<N_\ell$ we set 
\[x_j=\inf\{i\geq x_{j-1}+1\colon\, l_i\;\tilde{\wedge}\;l_{i+1}=0\}\]
so that $x_j$ is the index of the last vertical stretch composing the $j$-th bead of $\ell$. Finally, we let $n(\ell)$ be the number of beads in $\ell$, in particular it satisfies $x_{n(\ell)}=N_\ell$. 
Thus, we can decompose any trajectory $\ell \in \Omega_L$ into a succession of beads denoted by $\cB_j$ with $j\in \{1,\dots,n(\ell)\}$, as follows
 \begin{equation}\label{beads}
\ell=\cup_{j=1}^{n(\ell)} \cB_j:=\cup_{j=1}^{n(\ell)} \{\ell_{x_{j-1}+1},\dots,\ell_{x_j}\}.
 \end{equation}

A key issue concerning the collapsed phase of our model consists in showing that a typical trajectory contains a unique macroscopic bead outside which very few monomers are laying. Such result was derived for IPDSAW in \cite{CNP16}
 and recently improved in \cite{LP20}. Its proof  
requires some sharp asymptotics of the partition function restricted to single-bead trajectories (i.e., trajectories consisting of one bead only, see Section \ref{defbead} below for a definition). Such result is also very useful since it tells us that, inside its collapsed phase, the model should share many features with its single-bead counterpart. 
In particular, we expect that the geometric description of a typical path under the single-bead version of the model remains valid under the unrestricted model. This should simplify substantially the investigation of $\cC$. 
 
\begin{figure}[h]
 \centering
  \includegraphics[width=0.58\linewidth]{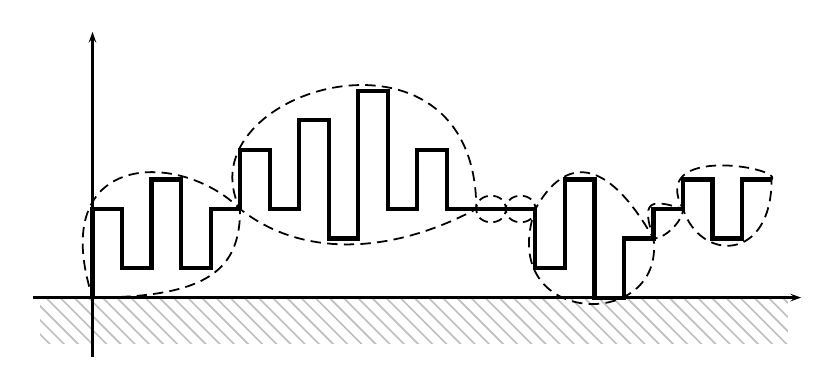}
  \caption{\footnotesize Decomposition of a trajectory into a succession of seven beads. For all $\ell\in\gO_L^{+}$ and $1\leq i\leq N_\ell$, a new bead starts at the $i$-th vertical stretch if and only if $\ell_{i-1}\ell_i\geq0$ (where we recall $\ell_0=\ell_{N_\ell+1}=0$).}
  \label{fig:beaddecompo}
\end{figure}

\subsection{Single-bead restriction of the model}\label{defbead} Let $L\in\N$, and define $\gO_L^{\,\circ,+}$ the subset of $\gO_L^+$ gathering those trajectories $\ell$ constrained to form only one ``bead'' ---all its stretches are of non-zero length and alternate orientations--- and to come back to the wall with its last stretch (in particular its horizontal extension $N_\ell$ must be even). That is,
\begin{equation}\label{def:gOxy+}
\gO_L^{\,\circ,+}:=\bigcup_{N\geq1} \cL_{N,L}^{\,\circ,+}:=\bigcup_{N\geq1}\left\{  \begin{aligned}&(\ell_i)_{i=1}^{2N}\in\bbZ^{2N}\; ; \sum_{i=1}^{2N}|\ell_i|=L-2N\;,\; \sum_{i=1}^{2N}\ell_i=0\;, \\
&\sum_{i=1}^k\ell_i\geq0,\, \forall k\leq 2N\;,\; \ell_i\ell_{i+1} <0, \forall\,1\leq i< 2N
\end{aligned}\right\}.\end{equation}
The partition function restricted to such trajectories becomes:
\begin{equation}\label{def:zxy+}
Z_{L,\gb,\gd}^{\,\circ,+}:= \sum_{N=1}^{L/2}\sumtwo{\ell\in\cL_{N,L}^{\,\circ,+}}{\ell_0=\ell_{2N+1}=0} e^{H(\ell)},
\end{equation}
(recall that $\ell_0=\ell_{2N+1}=0$  for notational convenience).

\subsection{Surface transition: asymptotics of the single-bead partition function
inside the collapsed phase} \label{sec3}

The single-bead model undergoes the same phase transition as the full model, albeit its critical curve differs slightly. Let us define $f^{\pt\circ}$ and $\tilde f^{\pt\circ}$ respectively the free energy and excess free energy of the single-bead model,
\[
f^{\pt\circ}(\gb,\gd):=\lim_{L\to\infty} \log Z^{\,\circ,+}_{L,\gb,\gd} \qquad\text{ and }\qquad \tilde f^{\pt\circ}(\gb,\gd) := f^{\pt\circ}(\gb,\gd) -\gb\;.
\]
Since $\tilde \ell$ forms a single bead (recall \eqref{deflti}), it follows that $\tilde f^{\pt\circ}(\gb,\gd)\geq0$ for all $(\gb,\gd)\in\cQ$.

\begin{proposition}\label{Phase-diag:bead}
For the single-bead model, we have 
\begin{align}
\cC_{\mathrm{bead}}\;:=\;\{(\beta,\delta)\in \cQ\colon \tilde f^{\pt\circ}(\beta,\delta)=0\}&\;=\;\{(\beta,\delta)\in \cQ\colon \beta\geq \beta_c, \delta \leq \delta^{\pt\circ}_c(\beta)\}\;,
\end{align}
where, for every $\beta\geq \beta_c$, the quantity $\delta_c^{\pt\circ}(\beta)$ is the unique solution in $\gd$ of 
\begin{equation}\label{eq:cc:bead}
2 \log \Gamma_\beta+h_\beta(\delta)=0.
\end{equation}
\end{proposition}
Notice that an analytic expression of $\delta_c^{\pt\circ}(\beta)$ can be derived from \eqref{eq:cc:bead} and \eqref{eq:explicit:h}, similarly to \eqref{eq:cc2} in Theorem~\ref{Phase-diag}. Let us now focus on the collapsed phase of the single-bead model. The surface transition occurs along a curve denoted by $\tilde\delta_c:[\beta_c,\infty) \mapsto \R^+$ where 
$\tilde \delta_c(\beta)$ turns out to be the critical point of the wetting model introduced in \eqref{wetmod}, that is for every 
$\beta\geq \beta_c$,
\begin{equation}\label{def:deltatilde}
\tilde \delta_c(\beta)\,:=\,\inf\{\delta\geq 0 \colon\,  h_\beta(\delta)>0\}\,=\,-\log(1-e^{-\beta/2})\;.
\end{equation}
The second identity in \eqref{def:deltatilde} is proven in Proposition~\ref{prop:wetpoint}. Definitions \eqref{eq:cc:bead} and \eqref{def:deltatilde} ensure us that $\tilde \delta_c(\beta)\leq \delta_c(\beta)\leq \delta_c^{\pt\circ}(\beta)$ for $\beta\geq \beta_c$: thus, the curve $\tilde\delta_c:[\beta_c,\infty) \mapsto \R^+$ lies in both $\cC$ and $\cC_{\mathrm{bead}}$. We will see in Theorem \ref{th:partfunbead} below that the second order term in the exponential development of $Z_{L,\gb,\gd}^{\,\circ,+}$ loses its analyticity along that curve.

At this stage, we divide the collapsed phase $\cC_{\mathrm{bead}}$ into a \emph{desorbed collapsed} phase $\cD\cC$ and an  \emph{adsorbed collapsed}  phase $\cA\cC$ defined as
\begin{equation}\begin{aligned}
\cD\cC&:=\{(\beta,\delta)\in \cQ\colon \beta\geq \beta_c, \delta \leq \tilde\delta_c(\beta)\}\;,\\
\cA\cC&:=\{(\beta,\delta)\in \cQ\colon \beta\geq \beta_c, \tilde\delta_c(\beta)\leq \delta \leq \delta_c^{\pt\circ}(\beta) \}\;,
\end{aligned}\end{equation}
where we dropped the subscript ``bead'' to lighten notations. To fully state Theorem~\ref{th:partfunbead}, we need to introduce some definitions. Let $\cL$ be the logarithmic moment generating function of the distribution $\Pbb$ (recall \eqref{def:Pgb}), that is for any $|h|<\gb/2$,
\begin{equation}\label{def:cL}
\cL(h) \;:=\; \log \Ebb [e^{hX_1}],
\end{equation}
and for every $\bh=(h_0,h_1)\in\cD_\gb:=\{(h_0,h_1)\in\R^2, |h_1|<\gb/2, |h_0+h_1|<\gb/2\}$, define
\begin{equation}\label{def:cLgL}
\cL_\gL (\bh) \;:=\; \int_0^1 \cL(h_0 x+h_1) \dd x\;,
\end{equation}
which is convex on $\cD_\gb$. In \cite[Lemma 5.3]{CNP16}, it is proven that $\bh\in\cD_\gb \mapsto \nabla\cL_\gL(\bh)=(\partial_{h_0}\cL_\gL, \partial_{h_1}\cL_\gL)(\bh)$ is a $\cC^1$-diffeomorphism from $\cD_\gb$ to $\R^2$, so let $\tilde \bh:\R^2\to\cD_\gb$ be its inverse. With those notations at hand, define
\begin{equation}
\phi_{(\gb,\gd)}(a)\,:=\,a \Big(2 \log \Gamma_\beta+  h_\beta(\delta)-\tfrac{1}{2a^2} \tilde h_0\big(\tfrac{1}{2a^2},0\big)+  \mathcal{L}_{\Lambda} (\tilde{\bf h}\big(\tfrac{1}{2a^2},0)\big)\Big)\;,\quad a\in(0,\infty)\,.
\end{equation}
We will prove that $\phi_{(\gb,\gd)}$ is negative and strictly concave on $(0,\infty)$ and reaches its maximum at some $\tilde a = \tilde a(\gb,\gd)\in(0,\infty)$. Finally, set $\Phi(\gb,\gd):=\phi_{(\gb,\gd)}(\tilde a(\gb,\gd))$, let $\gs_\gb^2$ be the variance of $X_1$ under $\Pbb$, and recall \eqref{defcomp} for the definition of $\asymp$. Recall also \eqref{eq:cc:bead} and \eqref{def:deltatilde} for the definitions of $\delta^{\pt\circ}_c(\beta)$ and $\tilde \delta_c(\beta)$ respectively.

\begin{theorem}\label{th:partfunbead}
Let  $\beta>\beta_c$.
\begin{enumerate}[label=(\roman*)]
\item For $\delta\in (\tilde \delta_c(\beta), \delta^{\pt\circ}_c(\beta))$, then
\begin{equation}
Z_{L,\beta,\delta}^{\,\circ,+}\:\asymp\: \frac{1}{L^{3/4} } e^{\, \beta L+\Phi(\beta,\delta) \sqrt{L}}\,,
\end{equation}
\item for $\delta\in (0, \tilde \delta_c(\beta)]$ and $\eps>0$, there exist $C>0$ and $L_0\in\N$ such that for $L\geq L_0$,
\begin{equation}
e^{\, \beta L+\Phi(\beta,0) \sqrt{L}+(\Psi(\gb)-\eps) L^{1/6}} \:\leq\: Z_{L,\beta,\delta}^{\,\circ,+}\:\leq\: \frac{C}{L^{3/4}} \pt e^{\, \beta L+\Phi(\beta,0) \sqrt{L}}\,,
\end{equation}
\item for $\delta=0$, then
\begin{equation}
Z_{L,\beta,\delta}^{\,\circ,+}\:=\: e^{\, \beta L+\Phi(\beta,0) \sqrt{L}+\Psi(\beta) L^{1/6} (1+o(1))}\;,\qquad\text{as}\quad L\to\infty\,,
\end{equation}
\end{enumerate}
where
\begin{equation*}
\Psi(\gb)\;:=\;-\,|\mathrm{a}_1|\pt \Bigg[\frac{\tilde a(\gb,0)\, \sigma_\gb^2}2 \; \tilde h_0\Big(\tfrac{1}{2\pt \tilde a_{(\gb,0)}^2},0\Big)^2 \Bigg]^{1/3}  
\end{equation*}
and $\mathrm{a}_1$ denotes the first zero (in absolute value) of the Airy function.
\end{theorem}

These estimate allow us to derive some properties of typical trajectories under the polymer measure, most notably regarding the number of contacts with the hard wall in the collapsed phase.

\begin{samepage}
\begin{corol}\label{corol:partfunbead}
Let $\gb>\gb_c$.
\begin{enumerate}[label=(\roman*)]
\item The function $\gd\mapsto\Phi(\gb,\gd)$ is $\cC^1$ on $(\tilde \delta_c(\beta), \delta^{\pt\circ}_c(\beta))$.  For $\delta\in (\tilde \delta_c(\beta), \delta^{\pt\circ}_c(\beta))$ and for any $\eps>0$ we have
\begin{equation}
\lim_{L\to\infty} \bP_{L,\beta,\delta}^{\,\circ,+} \Bigg( \sum_{k=1}^{N_\ell}\ind_{\{\sum_{i=1}^k\ell_i=0\}}  \in \Big[\partial_{\gd}\Phi(\gb,\gd)-\eps,\partial_{\gd}\Phi(\gb,\gd)+\eps\Big] \sqrt{L} \Bigg) = 1\:.
\end{equation}
\item For $\delta\in [0,\tilde \delta_c(\beta))$, there exist some $K>0$ (which only depends on $\gb$) such that
\begin{equation}
\lim_{L\to\infty} \bP_{L,\beta,\delta}^{\,\circ,+} \bigg( \sum_{k=1}^{N_\ell}\ind_{\{\sum_{i=1}^k\ell_i=0\}}  \leq K L^{1/6} \bigg) =1 \:.
\end{equation}
\end{enumerate}
\end{corol}
\end{samepage}

\begin{rmq} $\bullet$ The surface transition occurring along the curve $\{(\gb,\tilde\gd_c(\gb)), \gb>\gb_c\}$ is proven by Theorem~\ref{th:partfunbead} $(i)$-$(ii)$ (and confirmed by Corollary~\ref{corol:partfunbead}). Indeed, $\gd>\tilde\gd_c(\gb)$ implies that $h_\gb(\gd)>h_\gb(0)=0$, hence $\phi_{(\gb,\gd)}(a)>\phi_{(\gb,0)}(a)$ for $a\in(0,\infty)$, and $\Phi(\gb,\gd)>\Phi(\gb,0)$ (whereas $h_\gb(\gd)=0$ for all $\gd\leq\tilde\gd_c(\gb)$).\\
$\bullet$ In Theorem \ref{th:partfunbead} $(ii)$, we conjecture that the upper bound is not optimal, and that $(iii)$ should apply at least for every $\gd<\tilde\gd_c(\gb)$.
Similarly, for Theorem \ref{corol:partfunbead} $(ii)$, we expect  the typical number of contacts with the hard wall to be of smaller order than $L^{1/6}$.
\\
$\bullet$ In Theorem \ref{th:partfunbead} $(iii)$,  obtaining $\Psi(\gb)$ requires to compute the Laplace transform of the area enclosed by a Brownian meander of length $T$. The first zero of the Airy function $\mathrm{a}_1$ appears in the leading order of such Laplace transform as $T\to\infty$ (see Section~\ref{sec:step4}).
\end{rmq}

%

\subsubsection{Discussion} 
Let us give some insights into Theorem~\ref{th:partfunbead}. For any trajectory $\ell \in \gO_L^{\,\circ,+}$ forming a single bead, we define its \emph{lower envelope} $I:=(I_i)_{i=0}^{N_\ell/2}$ and \emph{upper envelope} $S:=(S_i)_{i=0}^{N_\ell/2+1}$ as follow:
\begin{align}\label{defenv}
S_k&=\sum_{i=1}^{2k-1} \ell_i, \quad k\in \{1,\dots, \tfrac{N_\ell}2+1\}\\
\nonumber I_k&=\sum_{i=1}^{2k} \ell_i, \quad k\in \{1,\dots, \tfrac{N_\ell}2\}.
\end{align}
The single-bead constraint ensures that $S$ (resp. $I$) describes the topmost (resp. bottommost) layer of the polymer (see Fig.~\ref{fig:envelopes}).
\begin{figure}[h]
 \centering
    \includegraphics[width=0.55\linewidth]{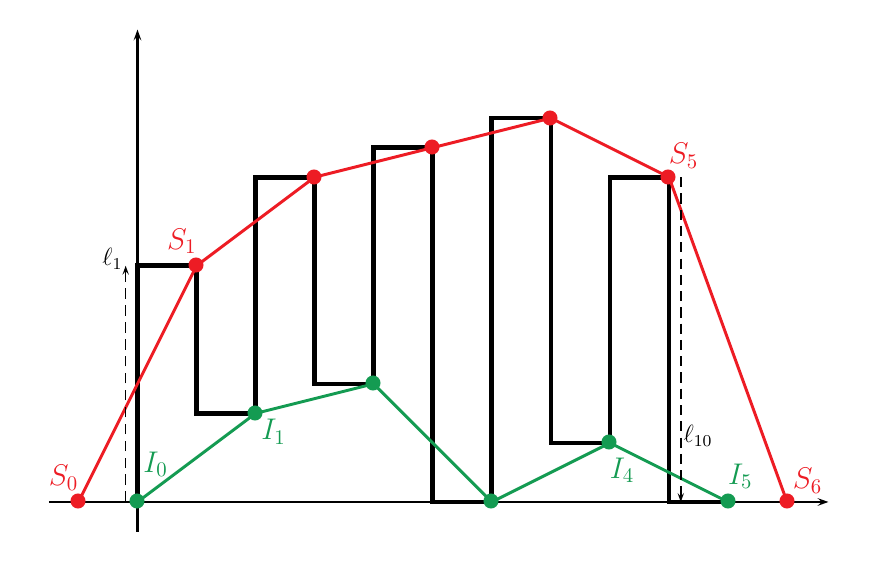}
  \caption{\footnotesize Representation of a single-bead trajectory, with its upper envelope (in red) and its lower envelope (in green) respectively described by the sequences $S$ and $I$.}
  \label{fig:envelopes}
\end{figure}

In Section \ref{rwrep} below, we will show that, under the polymer measure, the envelopes $I$ and $S$ of a given single-bead configuration may be sampled as trajectories 
of non-negative random walk bridges that are coupled via geometric constraints. 
In Section \ref{S2}, we will break that  geometric coupling and integrate over $S$ so that $I$ can be investigated on its own (see \eqref{def:Enq}). However, this comes with a cost (see Proposition~\ref{prop:estimD} below), the law of $I$ being perturbed by 
\begin{enumerate}
\item a wetting term  $\delta \sum_{k=1}^N 1_{\{I_k=0\}}$ which comes from the fact that the polymer (or equivalently its lower envelope) is adsorbed along the hard wall,\smallskip
\item a pre-wetting  term proportional to $-A_N(I)/N$ where $A_N(I)$ is the area below $I$ (see \eqref{def:AN}). The latter penalization 
comes from the fact that the upper envelope $S$ must sweep an abnormally large  area, i.e.,  $A_N(S)=A_N(I)+qN^2$ with $q>0$. Therefore,  $S$, which is already in a large deviation regime,
pushes down the lower envelope $I$ so as to keep $A_N(I)$ small. 
\end{enumerate}
\smallskip

The influence of a pre-wetting term on a $1+1$-dimensional random walk constrained to remain positive has already been studied in \cite{HV04} and \cite{ISV15} (see also \cite{IV18} for a review). 
Among physical motivations is e.g. the study of a liquide gaz interface when a thermodynamically stable 
gas is in contact with a substrate (hard-wall) that has a strong preference for the liquid 
phase, with the temperature decreasing to the liquid/gaz critical point. From this point of view, the present paper displays a new example of a physical object (i.e., the lower envelope of a collapsed  homopolymer interacting with a hard wall) associated with a model for which pre-wetting appears naturally. 

As stated in Theorem \ref{th:partfunbead}, the pre-wetting term  does not have an influence on the lower-envelope inside $\cA\cC$  since the pinning term is strong enough to keep the lower envelope at finite distance from the hard-wall. Inside $\cD\cC$, in turn  the pre-wetting term should dominate and we expect that \cite[Theorem~1.2]{HV04} also apply here, 
implying that $I$ has fluctuations of order $L^{1/6}$ (its length being $\sqrt{L}$).

%

\subsubsection{Open problems}\label{openp}
From a mathematical point of view, there are still many issues that remain to be settled concerning the present model.
\begin{enumerate}
\item Consider  
a  random walk $(X_i)_{i=0}^N$ (recall \eqref{def:Pgb}) constrained to remain non-negative and perturbed by both a wetting and a pre-wetting terms of parameter $\delta$ and $\gamma$ respectively, i.e.,
\begin{equation}\label{partfunprew}
\Eb_\gb\Big[e^{\delta \sum_{i=1}^N \ind_{\{I_k=0\}}} \ind_{\{I\in B^{0,+}_{N}\} }\, e^{-\frac{\gamma}{N} \, A_N(X)} \Big]. 
\end{equation}
Display sharp asymptotics for the partition function in \eqref{partfunprew} when $\delta$ is not larger than the critical point of the pure wetting model (i.e., $\delta\leq\tilde \delta_c(\beta)$).
Such estimates would be the key to improve Theorem~\ref{th:partfunbead} and Corollary~\ref{corol:partfunbead} in $\cD\cC$.
%
%
%
\item Consider the unrestricted model (defined in \eqref{def:gO+}--\eqref{def:freeen}) and prove that inside its collapsed phase $\cC$ a typical trajectory is made of a unique macroscopic bead.
Give a bound on the number of monomers that may lay outside this bead. 
\item For the unrestricted model again, prove the existence of the surface transition and provide an expression for its associated critical curve.
As stated in Proposition~\ref{Phase-diag:bead} the collapsed phase $\cC_{\mathrm{bead}}$ of the single-bead model contains its unrestricted counterpart $\cC$. 
However, and this is closely related to the former open issue, we conjecture that a surface transition takes place for the unrestricted model along the very same critical curve $\beta\mapsto  \tilde \delta_c(\beta)$.  
\item Provide a characterization of the critical curve dividing $\cE$ into an Adsorbed-Extended phase and a Desorbed-Extended phase.
\end{enumerate}


\section{Proof of Theorem \ref{th:partfunbead}}\label{sec:prth22}

We divide the proof of Theorem~\ref{th:partfunbead} into 6 steps. First, we adapt the random-walk representation of IPDSAW initially introduced in \cite{NGP13} to the present one-bead model. We derive a probabilistic representation of the partition function by rewriting, for every $N\leq L/2$, the contribution to the partition function of those trajectories made of $2N$ stretches, in terms of two auxiliary random walks $S$ and $I$. One particularity comes from the fact that $I$ and $S$ are coupled since the area enclosed in-between $S$ and $I$ is imposed by the length of the polymer, and another one comes from the one-bead constraint which implies that they cannot cross trajectories: hence $S$ (resp. $I$) will play the role of the \emph{upper envelope} (resp. \emph{lower envelope}) of the polymer. 
The second step consists in breaking the geometric coupling between $I$ and $S$. This is achieved by integrating over $S$, and transforming the area constraint into an exponential perturbation of the law of $I$. 
In the third and fourth steps we estimate the partition function of the lower envelope $I$, in $\cA\cC$ and $\cD\cC$ respectively. 
Finally the fifth step proves that inside the collapsed phase, the horizontal extension of a typical trajectory is of order $\sqrt{L}$, and the sixth step collects all those estimates to prove Theorem~\ref{th:partfunbead}.

\subsection{Step 1: random-walk representation.}\label{rwrep}

The understanding of IPDSAW (recall Remark \ref{compasm}) has recently been improved (see  \cite{CNPT18} for a review). The key tool was a new probabilistic representation of the partition function 
based on an auxiliary random walk conditioned to enclose a prescribed area. It turns out that this representation is of substantial help 
for the present model as well but under a different form. This is the object of the present step.

Let us first provide a probabilistic description of the one-bead partition function $Z_{L,\gb,\gd}^{\,\circ,+}$. We recall  \eqref{def:Ham} and \eqref{def:zxy+} and we observe that 
\begin{equation}\label{def:twg}
x\twg y\;=\;\frac{|x|+|y|-|x+y|}{2}\,, \quad x,y\in \mathbb Z\,. 
\end{equation}
Hence the one-bead partition function can be written as
\begin{equation}\label{def:zxy+:bis}
Z_{L,\gb,\gd}^{\,\circ,+}= e^{\gb L}\sum_{N=1}^{L/2} e^{-2\gb N}\sumtwo{\ell\in\cL_{N,L}^{\,\circ,+}}{\ell_0=\ell_{2N+1}=0} e^{-\frac{\gb}{2} \sum_{i=0}^{2N} |\ell_i + \ell_{i+1}|}\,e^{\gd\sum_{k=1}^{2N}\ind_{\{\sum_{i=1}^{k}\ell_i=0\}}}\,.
\end{equation}
At this stage, recall the definition of $\bP_\gb$ in \eqref{def:Pgb}. We consider two independent random walks $S:=(S_i)_{i\geq 0}$ and $I:=(I_i)_{i\geq 0}$ starting from 0 and such that $(S_{i+1}-S_i)_{i\geq 0}$ and $(I_{i+1}-I_i)_{i\geq 0}$  are i.i.d. sequences of random variables of law $\Pb_\gb$. We notice that for every $\ell\in\cL_{N,L}^{\,\circ,+}$ (with $\ell_0=\ell_{2N+1}=0$) the first factor in the second sum in \eqref{def:zxy+:bis} satisfies
\begin{equation}\label{def:zxy+:ter} 
e^{-\frac{\gb}{2} \sum_{i=0}^{2N} |\ell_i + \ell_{i+1}|}\,=\,c_\beta^{2N+1} \pt \bP_\gb\left(S_k=\sum_{i=0}^{2k-1}\ell_i, \forall k\leq N+1\right) \pt \bP_\gb\left( I_k=\sum_{i=0}^{2k}\ell_i, \forall k\leq N\right)\,.
\end{equation}

Let us introduce a one-to-one correspondence between trajectories $(\ell)_{i=0}^{2N+1}\in\bbZ^{2N+2}$ with $\ell_0=\ell_{2N+1}=0$, and $(S_i)_{i=0}^{N+1}\in\bbZ^{N+2}$, $(I_i)_{i=0}^N\in\bbZ^{N+1}$ with $S_0=I_0=0$ and $S_{N+1}=I_N$, by letting $S_k=\sum_{i=0}^{2k-1}\ell_i, \forall 1\leq k\leq N+1$ and $I_k=\sum_{i=0}^{2k}\ell_i, \forall 1\leq k\leq N$. 
Then, constraints on $\ell\in\cL^{\,\circ,+}_{N,L}$ can be transcribed to $S$ and $I$ (recall \eqref{def:gOxy+}). Indeed, $\sum_{i=1}^{2N}\ell_i=0$ is equivalent to $S_{N+1}=I_N=0$, and $\sum_{i=1}^{k}\ell_i\geq0$, $\forall k\leq 2N$ is equivalent to $S_i\geq0$, $\forall i\leq N+1$ and $I_i\geq0$, $\forall i\leq N$. Besides, we can write
\begin{equation}\label{def:G_N}
\sum_{i=1}^{2N}|\ell_i| = \sum_{k=1}^{N} | I_k - S_k | + \sum_{k=1}^{N+1} |S_k-I_{k-1}|\,=:\, G(S,I)\;,
\end{equation}
which is the ``geometric area'' between $S$ and $I$, and it is constrained to be equal to $L-2N$ in $\cL^{\,\circ,+}_{N,L}$. Finally, $\ell_i\ell_{i+1} <0, \forall\,1\leq i< 2N$ is equivalent to $S \succ I$ where we define
\begin{equation}\label{def:succN}
\big\{ S \succ I \big\}\;:=\;
\left\{\begin{aligned}
&S_k>I_k, \;\forall\, 1\leq k< N+1 \;, \\
&S_k>I_{k-1}, \;\forall \,1\leq k < N
\end{aligned}\right\}\,.
\end{equation}
In particular this means that in the one-bead model, $S$ remains above $I$ (thereby we respectively call $S$ and $I$ the \emph{upper} and \emph{lower envelopes} of the polymer ---recall Fig.~\ref{fig:envelopes}), and it implies that we can rewrite the geometric area $G(S,I)$: defining the \emph{signed area} as
\begin{equation}\label{def:AN}
A_n(X) \;:=\; \sum_{i=0}^{n} X_n\,,
\end{equation}
for any $n\in\N$ and $(X_k)_{k=0}^n \in\bbZ^{n+1}$, we then have that $\{S\succ I\}$ and $S_{N+1}=I_N=0$ imply $G(S,I)=2(A_{N+1}(S) - A_{N}(I))$.

Going back to \eqref{def:zxy+:bis} and plugging \eqref{def:zxy+:ter} in, those observations prove that we can rewrite the partition function as follows. Let $\Pbbx$ denote the law of a random walk on $\bbZ$ starting from $x$ with increments distributed as $\Pbb$ (henceforward we will omit the $x$-dependence in $\Pbbx$ when $x$ is clear from context). Recall that $\gGa_\gb=c_\gb e^{-\gb}$, and $B^{0,+}_{N}$ is defined in \eqref{def:By+}.
\begin{proposition}\label{prop:PEZ}
For $\gb>0$, $\gd\geq0$ and $L\geq1$,
\begin{equation}\label{PEZ}
\widetilde{Z}_{L,\gb,\gd}^{\,\circ,+} \;:=\; \frac{1}{c_\gb e^{\gb L}} Z_{L,\gb,\gd}^{\,\circ,+} \;=\; \sum_{N=1}^{L/2} (\gGa_\gb)^{2N} D^{\,\circ}_{N,q_N^L}\;,
\end{equation}
with $q_{N}^L:=(L-2N)/2N^2$, and for $q\in (0,\infty)\cap \frac{\N}{2N^2}$,
\begin{equation}\label{def:DNL}
D^{\,\circ}_{N,q}:= \Ebbzero\left[e^{\gd\sum_{k=1}^{N}\ind_{\{I_k=0\}}} 
\ind_{\{S\in B^{0,+}_{N+1},\,I\in B^{0,+}_{N}\}}
\ind_{\{ S \succ I\}}
\ind_{\left\{A_{N+1}(S)=A_{N}(I)+q N^2 \right\}}\right].
\end{equation}
where $S,I$ are two independent random walks distributed as $\Pbbzero$.
\end{proposition}

Our aim is to provide sharp estimates on the partition function $Z_{L,\gb,\gd}^{\,\circ,+}$. To that purpose, we need to estimate $D^{\,\circ}_{N,q}$ uniformly in $q\in[q_1,q_2]\cap \frac1{2N^2}\bbN$ for any $0<q_1<q_2$ fixed.

\subsection{Step 2: integrating over $S$ in $D^{\,\circ}_{N,q}$}\label{S2}
With Proposition \eqref{prop:estimD} below, we show that the geometric constraint imposed to $I$ and $S$ can be relaxed provided we introduce an exponential perturbation to the law of $I$, thus breaking the coupling between $I$ and $S$.

Define $g:\R^2\to\R$ by
\begin{equation}\label{def:gqp}
g(q,p) \;:=\; \tilde \bh^{q,p} \cdot (q,p) - \cL_\gL(\tilde \bh^{q,p})\;,
\end{equation}
where $\cD_\gb\subset\R^2$, $\tilde \bh^{\cdot,\cdot}:\R^2\to\cD_\gb$ and $\cL_\gL:\cD_\gb\to\R$ are all defined in \eqref{def:cLgL} ---from now on we will note $\tilde \bh^{\cdot,\cdot}:=\tilde \bh{(\cdot,\cdot)}$ to tighten notations. We will see that $g$ is the rate function of a large deviation principle associated to the couple $\gL_N:=(A_N(X)/N,X_N)$ under $\Pbb$ (see the proof of Proposition~\ref{prop:DNLSI}). Using \cite[Lemma V.4]{dH00}, it follows that $g$ is convex. Moreover $\tilde \bh$ and $\cL_\gL$ are $\cC^1$ functions and
\begin{equation}\label{eq:gradg}
\nabla g(q,p) \,:=\, (\partial_q g, \partial_p g)(q,p) \,=\, \tilde \bh^{q,p}\;,
\end{equation}
in particular $g$ is $\cC^2$. Finally $\nabla g(0,0)=\tilde \bh^{0,0}=(0,0)$, so $g(q,p)\geq g(0,0)=0$ for all $(q,p)\in\R^2$.

 \begin{proposition}\label{prop:estimD}
Fix some $q_2>q_1>0$. Then 
 \begin{equation}\label{eq:lbDNL:3}
D^{\,\circ}_{N,q} \;\asymp\; \frac{e^{-Ng(q,0)}}{N^2}  E^{\,\circ}_{N,q} \;,
\end{equation}
uniformly in $q\in [q_1,q_2]\cap \frac{1}{2N^2}\N$, where
\begin{equation}\label{def:Enq}
E^{\,\circ}_{N,q}\;:=\;\Ebbzero \left[e^{\gd\sum_{k=1}^{N}\ind_{\{I_k=0\}}} \ind_{\{I\in B^{0,+}_{N}\}} e^{-\,\partial_qg(q,0) \frac{A_{N}(I)}{N} } \right].
\end{equation}
\end{proposition}
 \medskip

\subsubsection{Outline of the proof of Proposition \ref{prop:estimD}} Obtaining an upper bound on $D^{\,\circ}_{N,q}$ is rather straightforward: we will simply remove the events
$\{S \succ I\}$ and $\{S_i\geq 0, \forall \,  0\leq i\leq N+1\}$ from the definition of 
$D^{\,\circ}_{N,q}$ and then apply large deviation estimate for the signed area $A_{N+1}(S)$ of the upper envelope. The proof of the lower bound is more involved and will consist in 
\begin{enumerate}
\item getting rid of the event $\{S \succ I\}$ by first introducing a stronger constraint, i.e., 
$S$  (resp. $I$) remains above (resp. below) some deterministic curves. 
\smallskip

\item conditioning on $I$ that remains below the deterministic curve introduced in (1) and integrating over $S$. Then, with the help of large deviation estimates for random walks enclosing very large area, proving that the events $\{S\in B_{N+1}^{0,+}\}$ and $\{A_{N+1}(S)=qN^2+A_N(I)\}$ have a probability $\frac{1}{N^2} e^{-Ng(q,0)} e^{-\partial_qg(q,0) \frac{A_N(I)}{N}}$ and that the cost for $S$ to remain above the deterministic curve introduced in (1) is bounded from below by a constant, 
\smallskip
  
\item getting rid of the additional constraint on $I$ introduced in (1) with an {\rm FKG} inequality, by observing that the random walk $I$ whose law is penalized exponentially by $-\frac{1}{N} A_N(I)$ typically remains below its unpenalized counterpart.  
\end{enumerate}

\subsubsection{Proof of the lower bound in \eqref{eq:lbDNL:3}} \label{sec:LB}

\subsubsection{Integrating on the upper envelope $S$}
To begin with, we constrain the upper and lower envelopes to remain respectively above and below some fixed curves. In particular this allows us to handle the condition $S \succ I$. Define $f(t):=\gga (t\wedge (1-t)), t\in[0,1]$, and
\begin{equation}\label{def:tildefSI}\begin{array}{ll}
\tilde f_S (k) \;:=\; (N+1) f\big(\tfrac k{N+1}\big) + K_S\;, & \quad\forall\pt 1\leq k\leq N\;,\\
\tilde f_I (k) \;:=\; N f\big(\tfrac kN\big) + K_I\;, & \quad \forall \pt 1\leq k\leq N-1\;,
\end{array}\end{equation}
where $\gga, K_S, K_I>0$ are constants. If we constrain $S_k$, $1\leq k\leq N$ (resp. $I_k$, $1\leq k\leq N-1$) to remain above $\tilde f_S(k)$ (resp. below $\tilde f_I(k)$), and provided that $K_S-K_I\geq1+\gga$, then we have
\begin{equation}\label{def:tildefSI:bis}
\{S \succ I\}\;\supset \; \left\{ \begin{aligned} &S_k\geq \tilde f_S(k),\: \forall\, 1\leq k \leq N\,; \\ &\;I_k\leq \tilde f_I(k),\: \forall\, 1\leq k \leq N-1\end{aligned} \right\} \,.
\end{equation}
Geometrically, 
$\tilde f_\cdot(k)$ is a piecewise linear curve above the wall of order $N$ when $k=N/2$, and constant at its ends. The constants $K_S, K_I$ are only there for technical purposes and are mostly irrelevant when $N$ is large (see Figure~\ref{fig:courbef}). More precisely, we will fix $\gga$ such that Proposition~\ref{prop:DNLSI} below applies (for any $K_S>0$), then we fix $K_I$ when applying Lemma~\ref{lem:FKG}, and we finally fix $K_S>K_I$ such that \eqref{def:tildefSI:bis} holds. 




\begin{figure}[h]
 \centering
 \includegraphics[width=0.7\linewidth]{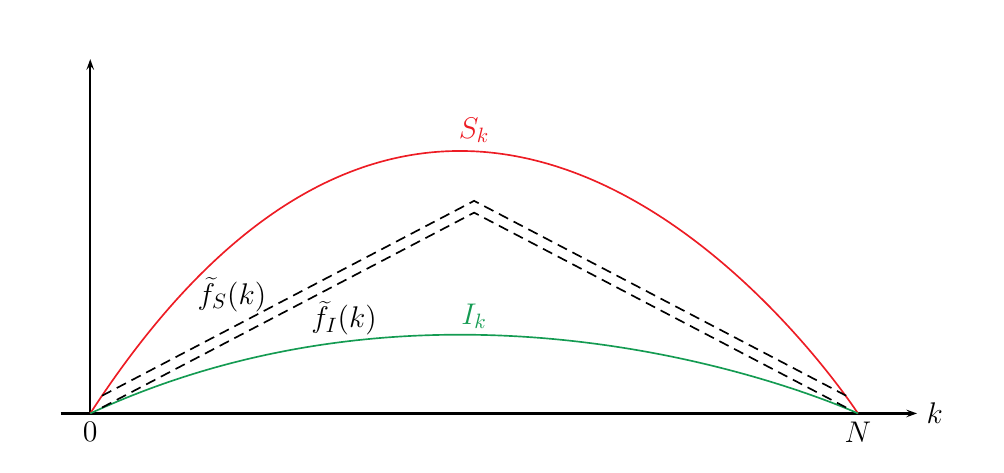}
 \caption{\footnotesize Representation of the constraints on $S$, $I$, for $N$ large. The piecewise linear curves $\tilde f_S$ and $\tilde f_I$ are above the wall, reaching a height of order $N$ in the middle and bounded at both ends. $S$ is constrained to remain above $\tilde f_S$, and $I$ below $\tilde f_I$.
 }
 \label{fig:courbef}
\end{figure}

We also add the constraint $\{A_N(I)\leq C_A N^{3/2}\}$, where $C_A>0$ is some constant which will also be fixed when applying Lemma~\ref{lem:FKG} below. Let us condition $D^{\,\circ}_{N,q}$ over the trajectory $I$ (recall \eqref{def:DNL}), so that we separate the constraints on $I$ and $S$ and we obtain the lower bound
\begin{equation}\label{eq:lbDNL:1}
D^{\,\circ}_{N,q}\;\geq\; \Ebbzero \left[e^{\gd\sum_{k=1}^{N}\ind_{\{I_k=0\}}} \ind_{\{I\in B^{0,+}_N\}} \ind_{\{A_N(I)\leq C_A N^{3/2}\}} \ind_{\{I_k\leq \tilde f_I(k),\: \forall\, 1\leq k < N\}}\, \widehat D^{\,\circ}_{N,q}(I) \right]\,,
\end{equation}
with:
\begin{equation}\label{def:DNLSI}
\widehat D^{\,\circ}_{N,q}(I)\;:=\;\Ebbzero \left[ \ind_{\{S_{N+1}=0\}} \ind_{\{A_{N+1}(S)=A_N(I)+qN^2\}} \ind_{\{S_k\geq \tilde f_S(k),\: \forall\, 1\leq k < N+1\}} \middle| \, I\right]\,.
\end{equation}

Let us fix $I=(I_k)_{k=1}^N$ which satisfies all constraints in~\eqref{eq:lbDNL:1}. Proposition~\ref{prop:DNLSI} below allows us to estimate $\widehat D^{\,\circ}_{N,q}(I)$ up to uniform constant factors. We postpone the proof to Section~\ref{proof:DNLSI}.

\begin{proposition}\label{prop:DNLSI} Let $\eps>0$ and $q_1<q_2,p_1<p_2\in\R$. There exists $\gga>0$ such that for all $K_S>0$, there exist $C_2>C_1>0$ and $N_0\in \N$ such that for every $N\geq N_0$, one has
\begin{equation}\begin{aligned}
\frac{C_{\cntc}}{N^2} e^{-Ng(q,p)}&\;\leq\; \Ebbzero\Big[ \ind_{\{X_{N}=pN,\, A_N(X)=q N^2\}} \ind_{\{X_k\,\geq\, \tilde f(k),\,\forall\, 1\leq k < N\}} \Big]\\
&\;\leq\; \Ebbzero\Big[ \ind_{\{X_{N}=pN,\, A_N(X)=q N^2\}} \Big] \;\leq\; \frac{C_{\cntc}}{N^2} e^{-Ng(q,p)}\;,
\end{aligned}\end{equation}
where we define $\tilde f(k):=N f(\frac kN) + k\pt p + K_S$, and these bounds hold uniformly in $q\in [q_1,q_2]\cap \frac{\bbZ}{2N^2}$, $p\in[p_1,p_2]\cap \frac{\bbZ}{N}$ satisfying $p\leq 2q-\eps$.
\end{proposition}

\begin{remark} $(i)$ Although we will only use Proposition~\ref{prop:DNLSI} with $p=0$, we prove it in the general case ($p\neq0$) because such estimates will be useful when studying the model without the single-bead restriction.\\
$(ii)$ The assumption $q\geq\frac{p+\eps}2$ ensures that the trajectory has to enclose a large area, so it is prompted to draw a concave shape above the straight line from $(0,0)$ to $(N, pN)$; hence it remains above the curve $\tilde f$ provided that $\gga$ is sufficiently small. Notice that the assumption $q=\frac p2$ would be satisfied by a linear trajectory from $(0,0)$ to $(N, pN)$, and a lower value of $q$ would force the trajectory to be convex (below the linear trajectory). \\
$(iii)$ To prove Proposition~\ref{prop:DNLSI}, we first recall from \cite{CNP16} an estimate of the expectation without the constraint $\{X_k\,\geq\, \tilde f(k),\,\forall\, 1\leq k < N\}$, then we prove that this constraint only costs up to a constant factor. In particular this strongly reinforces \cite[Prop.~2.5]{CNP16}, since it yields that 
there exists $C_\cntc>0$ such that uniformly in $N\geq N_0$ and in  $q \in [q_1,q_2]\cap \frac{\N}{2N^2}$ with $q_1>0$,
\begin{equation*}
\Pbbzero(X_i>0, 0<i<N\, |\, A_N=qN^2, X_N=0)\geq C_{\arabic{cst}}\;.
\end{equation*}
\end{remark}

Recall \eqref{def:DNLSI}, and let $p'=0$, $q':=(qN^2+A_N(I))/(N+1)^2$. Under our assumptions, there is a compact subset $[q'_1,q'_2]\subset(0,\infty)$ and $N_0\in\N$ such that $q'\in[q'_1,q'_2]$ for all $N\geq N_0$, $q\in[q_1,q_2]$ and $I$ satisfying all constraints from \eqref{eq:lbDNL:1} (recall that $A_N(I)=O(N^{3/2})$). Moreover we have $0=p'<q'_1/2$, hence we can apply Proposition~\ref{prop:DNLSI} to $\widehat D^{\,\circ}_{N,q}(I)$, and we obtain the uniform bound for $N\geq N_0$,
\begin{equation}\label{eq:lbDNLSI:1}
\widehat D^{\,\circ}_{N,q}(I)\;\geq\; \frac{C_{3}}{N^2} e^{-Ng(q', 0)}\;.
\end{equation}
Now recall that $g$ is $\cC^2$ and convex, so there exists $C_\cntc>0$ such that
\begin{equation*}
g(q',0)\;\leq\; g(q,0) + \partial_q g(q,0) (q'-q) + C_{\arabic{cst}} (q'-q)^2 \;,
\end{equation*}
uniformly in $q\in[q_1,q_2]$ and $q'\in[q'_1,q'_2]$. Recall that $A_N(I)=O(N^{3/2})$, and notice $q'-q=\frac{A_N(I)}{N^2}+O(1/N)$. 
Thereby \eqref{eq:lbDNLSI:1} becomes
\begin{equation}\label{eq:lbDNLSI:2}
\widehat D^{\,\circ}_{N,q}(I)\;\geq\; \frac{C_\cntc}{N^2} e^{-Ng(q, 0) -\partial_q g(q,0)\frac{A_N(I)}N}\:,
\end{equation}
where $C_{\arabic{cst}}$ is uniform in $q\in[q_1,q_2]\cap \frac{\N}{2N^2}$ and $I$. Plugging \eqref{eq:lbDNLSI:2}  into \eqref{eq:lbDNL:1} we obtain
\begin{equation}\label{eq:lbDNL:2}
\begin{aligned}
D^{\,\circ}_{N,q}\,\geq\, &\frac{C_{\arabic{cst}}}{N^2} e^{-Ng(q,0)} \, \Ebbx \bigg[e^{\gd\sum_{k=1}^{N}\ind_{\{I_k=0\}}} \ind_{\{I\in B^{0,+}_N\}} \ind_{\{A_N(I)\leq C_A N^{3/2}\}} \\
&\qquad \qquad \qquad \quad \times \ind_{\{I_k\leq \tilde f_I(k),\: \forall\, 1\leq k < N\}}\,  e^{-\partial_q g(q,0)\frac{A_N(I)}N}\bigg]\:.
\end{aligned}
\end{equation}
Recall that $g$ is non-negative, and let us point out that $\partial_q g(q,0)=\tilde h_0^{q,0}>0$ for all $q>0$ (recall \eqref{eq:gradg}, and see \cite[Rem.~5.5]{CNP16} or Section~\ref{proof:DNLSI}).

\subsubsection{Relaxing some constraints on the lower envelope $I$} Our goal now is to drop the events  $\{I_k\leq \tilde f_I(k),\: \forall\, 1\leq k < N\}$ and $\{A_N(I)\leq C_A N^{3/2}\}$ in the r.h.s. in \eqref{eq:lbDNL:2} by paying up to a constant factor uniform in $N\geq N_0$ and $q\in [q_1,q_2]$. To that aim we use an {\rm FKG} inequality. Indeed, notice that the functions\begin{align*}
 x=(x_k)_{k=1}^N \;&\longmapsto\; e^{\gd\sum_{k=1}^{N}\ind_{\{x_k=0\}}} e^{-\partial_q g(q,0)\frac{A_N(x)}N}\:,\\
 \text{and} \qquad\qquad\qquad x \;&\longmapsto\; \ind_{\{x_k\leq \tilde f_I(k),\: \forall\, 1\leq k < N\}} \ind_{\{A_N(x)\leq C_A N^{3/2}\}}\;,
\end{align*}
are both bounded and \emph{non-increasing} on $B_N^{0,+}$, where we say that $f:B_N^{0,+}\to\R$ is non-increasing if for all $x,y\in B_N^{0,+}$ such that $x_k\leq y_k,\forall\,1\leq k\leq N$, one has $f(x)\geq f(y)$. Thereby the {\rm FKG} inequality claimed in Proposition~\ref{prop:FKG} yields that
\begin{equation}\label{eq:FKG:2}\begin{aligned}
&\Ebbzero\bigg[  e^{\gd\sum_{k=1}^{N}\ind_{\{X_k=0\}}}\ind_{\{X_k\leq \tilde f_I(k),\: \forall\, 1\leq k < N;A_N(x)\leq C_A N^{3/2}\}}
e^{-\partial_q g(q,0)\frac{A_N(X)}N} \Big| X\in B^{0,+}_N \bigg]\\
&\quad\geq \Pbbzero\big(X_k\leq \tilde f_I(k),\: \forall\, 1\leq k < N \,;\, A_N(X)\leq C_A N^{3/2}\,\big|\, X\in B^{0,+}_N\big)\\
&\qquad \times \Ebbzero\Big[ e^{\gd\sum_{k=1}^{N}\ind_{\{X_k=0\}}}  e^{-\partial_q g(q,0)\frac{A_N(X)}N} \,\Big|\, X\in B^{0,+}_N \Big].
\end{aligned}\end{equation}
Finally, the first factor of the r.h.s. is bounded from below by some constant (close to 1) by the following lemma, which is proven in Section~\ref{proof:FKG}.
\begin{lemma}\label{lem:FKG} Let $\eps,\gga>0$. There exist $C_A, K_I>0$ and $N_0\in\N$ such that
\begin{equation}\label{bdcurve}
\Pbbzero\big(X_k\leq \tilde f_I(k),\,\forall\, 1\leq k < N\,; A_N(X)\leq C_A N^{3/2} \,\big|\, X\in B^{0,+}_N\big) \:\geq\:1-\eps\;,
\end{equation}
for all $N\geq N_0$.
\end{lemma}
By recollecting \eqref{eq:lbDNL:2} and \eqref{eq:FKG:2}, and provided that $C_A, K_I$ are sufficiently large to apply Lemma~\ref{lem:FKG}, we conclude that 
for every $N\geq N_0$ and every $q\in [q_1,q_2]\cap \frac{\N}{2N^2}$ we have
\begin{equation}\label{eq:lbDNL:4}
D^{\,\circ}_{N,q} \;\geq\; \frac{C_\cntc}{N^2} e^{-Ng(q,0)} \,\Ebbzero \left[e^{\gd\sum_{k=1}^{N}\ind_{\{I_k=0\}}} \ind_{\{I\in B^{0,+}_N\}} e^{-\partial_q g(q,0)\frac{A_N(I)}N} \right],
\end{equation}
which completes the proof of the lower bound.

\subsubsection{Proof of the upper bound in \eqref{eq:lbDNL:3}}\label{sec:UB}
We recall \eqref{def:DNL} and we bound $D^{\,\circ}_{N,q}$ from above by relaxing partially the  constraints on $S$ -more precisely the constraints $\{S \succ I\}$ and $\{S_i\geq0,\pt \forall \pt 1\leq i\leq N\}$. Therefore, 
\begin{equation}\label{def:DNLalt}\begin{aligned}
D^{\,\circ}_{N,q}\,&\leq\, \Ebbzero\left[e^{\gd\sum_{k=1}^{N}\ind_{\{I_k=0\}}} 
\ind_{\{I\in B^{0,+}_N\}}\ind_{\{S_{N+1}=0\}}
\ind_{\{A_{N+1}(S)=A_N(I)+q N^2\}}\right],\\
&=\,  \Ebbzero\left[e^{\gd\sum_{k=1}^{N}\ind_{\{I_k=0\}}} 
\ind_{\{I\in B^{0,+}_N\}}\, \tilde D^{\,\circ}_{N,q}(I) \right]. 
\end{aligned}\end{equation}
with
\[
\tilde D^{\,\circ}_{N,q}(I):=\Pbbzero\Big[S_{N+1}=0, A_{N+1}(S)=A_N(I)+qN^2 \,\Big| \pt I\pt \Big]\;.
\]
In order to get uniform bounds on $\tilde D^{\,\circ}_{N,q}(I)$ with Proposition~\ref{prop:DNLSI}, we need to drop those $I$ sweeping a too large area. To that aim, for any $c>0$
we rewrite \eqref{def:DNLalt} as $\tilde D^{\,\circ}_{N,q}\leq R^1_{N,q}(c)+R^2_{N,q}(c)$ with
\begin{align*}
R^1_{N,q}(c)&:= \Ebbzero\left[e^{\gd\sum_{k=1}^{N}\ind_{\{I_k=0\}}} 
\ind_{\{I\in B^{0,+}_N,\,  A_N(I)\leq cN^2\}}\, \tilde D^{\,\circ}_{N,q}(I) \right]\;,\\
R^2_{N,q}(c)&:= \Ebbzero\left[e^{\gd\sum_{k=1}^{N}\ind_{\{I_k=0\}}} 
\ind_{\{I\in B^{0,+}_N,\, A_N(I)> c N^2 \}}\, \tilde D^{\,\circ}_{N,q}(I) \right]\;,
\end{align*} 
and we write the very crude bound 
\begin{align}
\nonumber R^2_{N,q}(c)&\leq  e^{\delta N} \Pbbzero\left[A_N(I)> c N^2 \right]\leq e^{\delta N} \Pbbzero\left[\max_{1\leq i\leq N} I_i> c N \right]\\
&\leq e^{\delta N} \Pbbzero\left[\textstyle \sum_{i=1}^N |I_i-I_{i-1}|> c N \right]\;.
\end{align}
Then, we set $M:=1+\max\{g(q,0), q\in [q_1,q_2]\}$ and we use the fact that $(I_i-I_{i-1})_{i\geq 1}$ are i.i.d. with finite small exponential moments to conclude (via a Markov exponential inequality) that there exists a $c_M>0$ such that  for every $N\geq N_0$ and $q\in [q_1,q_2]\cap\frac{\N}{2N^2}$, 
\begin{equation}\label{boundM}
R^2_{N,q}(c_M)\leq e^{-M N}.
\end{equation}

Let us now consider $R^{1}_{N,q}(c_M)$  and use Proposition~\ref{prop:DNLSI} to assert that there exists $C_\cntc>0$ and $N_0\in \N$ (depending on $M$)  such that  for every $N\geq N_0$, $q\in [q_1,q_2]\cap \frac{\N}{2N^2}$ and  $I$ satisfying  $A_N(I)\leq  c_M N^2$ we have 
\[
\tilde D^{\,\circ}_{N,q}(I)\leq   \frac{C_{\arabic{cst}}}{N^2} e^{-N g(q', 0)}\;,
\]
with $q'=(qN^2+A_N(I))/(N+1)^2$. Recall that $g$ is convex, therefore we can bound $g(q',0)$ from below with
\begin{equation*}
g(q',0) \;\geq\; g(q,0) + \partial_{q} g(q,0) (q'-q)\;.
\end{equation*}
Moreover $q'-q=\frac{A_N(I)}{N^2}+O(1/N)$, thus,
\begin{align}
\nonumber R^1_{N,q}(c_M)&\leq \frac{C_\cntc}{N^2} e^{-g(q,0) N} \Ebbzero\left[e^{\gd\sum_{k=1}^{N}\ind_{\{I_k=0\}}} 
\ind_{\{I\in B^{0,+}_N,\pt A_N(I)\leq c_M N^2\}}\, e^{-\partial_q g(q,0) \frac{A_N(I)}{N}}\right]\\
& \label{eq:boundR1} \leq \frac{C_{\arabic{cst}}}{N^2} e^{-g(q,0) N} E^{\,\circ}_{N,q}\;,
\end{align}
where we recall \eqref{def:Enq}.

Our proof will be complete once we show that for $N$ large enough and for every $q\in [q_1,q_2]\cap \frac{\N}{2N^2}$,
the r.h.s. in \eqref{boundM} (i.e., $e^{-MN}$) is not larger than the r.h.s. in \eqref{eq:boundR1}. To that aim, we write the lower bound 
\begin{equation}\label{eq:DNLSI:lbENq}
E^{\,\circ}_{N,q} \,\geq\, \Pbbzero(A_N(I)\leq C_AN^{3/2}\mid  I\in B_N^{0,+})\,\Pbbzero(I\in B_N^{0,+}) \, e^{-\max\{\partial_qg(q,0), q\in [q_1,q_2]\}  \sqrt{N}}\,,
\end{equation}
for any $C_A>0$, where we constrained the walk to have area at most $C_AN^{3/2}$. A direct consequence of Lemma~\ref{lem:FKG} is that there exist $C_A,C_\cntc>0$ such that for $N$ large enough, 
\begin{equation*}
\Pbbzero\big(A_N(I)\leq C_AN^{3/2}\,\big|\, I\in B_N^{0,+}\big) \geq  C_{\arabic{cst}}\;,
\end{equation*} 
(more generally this follows from an invariance principle on non-negative random walk bridges, see \eqref{eq:invarexcursion}). The second factor in \eqref{eq:DNLSI:lbENq} is bounded from below by $C_\cntc N^{-3/2}$ for some $C_{\arabic{cst}}$ (see \eqref{eq:asymp:PosBB}). 
Recalling that $q\mapsto\partial_qg(q,0)$ is continuous on $[a_1,a_2]$, we conclude that there exists $C_\cntc, c_1>0$ such that for $N$ large enough and 
every $q\in [q_1,q_2]\cap \frac{\N}{2N^2}$, we have
\begin{equation}\label{eq:lobdint}
E^{\,\circ}_{N,q} \geq \frac{C_{\arabic{cst}}}{N^{3/2}} \, e^{-c_1 \sqrt{N}}\;.
\end{equation}
Recalling that $M=1+\max\{g(q), q\in [q_1,q_2]\}>0$, it suffices to combine \eqref{boundM} with \eqref{eq:boundR1} and \eqref{eq:lobdint} to complete the proof of the upper bound in \eqref{eq:lbDNL:3}.

\subsection{Step 3 : area-penalized wetting model for $\gd> \tilde \gd_c(\gb)$}
In this section we give estimates on $E^{\,\circ}_{N,q}$ when $\gd>\tilde \gd_c(\gb)$ ---in particular in the phase $\cA\cC$. 
Notice that $E^{\,\circ}_{N,q}$ is the partition function of a wetting model with an additional pre-wetting term. Let $Z_{\wet,N}^{\gb,\gd}$ be the standart wetting partition function (without pre-wetting, see \eqref{def:Zwetting}). Such wetting models have already been studied extensively by mathematicians ---we provide well-known results on them in Proposition~\ref{prop:wetpoint}. Recall that $\tilde \gd_c(\gb)=\inf\{\gd>0,h_\gb(\gd)>0\}$ is the critical point of the wetting model $Z_{\wet,N}^{\gb,\gd}$, $h_\gb(\gd)$ is its free energy, and $\asymp$ is defined in \eqref{defcomp}.
\begin{prop}\label{prop:ENq:AC}
Let $q_2>q_1>0$, and assume $\gd>\tilde \gd_c(\gb)$. Then $E^{\,\circ}_{N,q}\asymp e^{h_\gb(\gd)N}$ uniformly in $q\in[q_1,q_2]$.
\end{prop}
\begin{proof} The upper bound is a straightforward consequence of the inequality $E^{\,\circ}_{N,q} \leq Z_{\wet, N}^{\gb,\gd}$ and Proposition~\ref{prop:wetpoint}~(ii). For the lower bound, we first bound $\partial_q g(q,0)$ by some $c>0$ for all $q\in[q_1,q_2]$ (recall that it is a continuous and positive function), then decompose trajectories in $E^{\,\circ}_{N,q}$ into excursions:
\begin{equation}\label{eq:lbENq:1}
E^{\,\circ}_{N,q} \;\geq\; \sum_{r=1}^N \sumtwo{1\leq t_1,\ldots,t_r\leq N,}{t_1+\cdots+t_r=N} \prod_{i=1}^r K_\gb(t_i)\,e^{\gd}\, \Ebbzero\left[ e^{-c\frac{A_{t_i}(I)}{N}}\,\middle|\, \tau=t_i, I_{t_i}=0\right]\;,
\end{equation}
where we define $\tau=\inf\{t\geq1;I_t\leq0\}$ and $K_\gb(t)=\Pbbzero(\tau=t,I_t=0)$, $t\geq1$ as in Appendix~\ref{app:wetting}.
Using Jensen's inequality, we can claim that 
\begin{equation}\label{jensen}
\Ebbzero\left[ e^{-c\frac{A_{t}(I)}{N}}\,\middle|\, \tau=t, I_{t}=0\right]\geq e^{-c\frac{1}{N} \Ebbzero\left[ A_{t}(I)\,\middle|\, \tau=t, I_{t}=0\right]}, \quad t\in \N.
\end{equation}
Moreover, by using the inequality $\Ebbzero( |A_{t}(I)|) \leq \Ebbzero(|I_1|) \sum_{j=1}^t j$  combined with \eqref{equivexc} we can state that  there exists a $\bar c>0$ such that 
\begin{equation}\label{jensencont}
 \Ebbzero\left[ A_{t}(I)\,\middle|\, \tau=t, I_{t}=0\right]\leq  \Ebbzero( |A_{t}(I)|) \,  K_\gb(t)^{-1}\leq \bar c \, t^{7/2}, \quad t\in \N.
\end{equation}
Using (\ref{eq:lbENq:1}--\ref{jensencont}) we obtain that there exists a $c_1>0$ such that,
\begin{equation}\label{eq:lbENq:2}
E^{\,\circ}_{N,q} \;\geq\; \sum_{r=1}^N \sumtwo{1\leq t_1,\ldots,t_r\leq N,}{t_1+\cdots+t_r=N} \prod_{i=1}^r K_\gb(t_i)\,e^{\gd}\, e^{-c_1\frac{t_i^{7/2}}{N} }\;.
\end{equation}
Since $\delta>\tilde \gd_{c}(\gb)$, equation \eqref{eq:prop:wetpoint} guarantees that
\begin{equation}\label{def:Qinfty}
Q_\infty(t)\;:=\; K_{\gb}(t)e^\gd e^{-h_\gb(\gd) t} \qquad t\in\N\;,
\end{equation}
is a probability distribution on $\N$. We also define for all $N\in\N$,
\begin{equation}\label{def:QN}
Q_N(t)\;:=\; K_{\gb}(t)e^\gd e^{-h_Nt} e^{-\frac {c_1}N t^{7/2}}\qquad t\in\N\;,
\end{equation}
where $h_N$ is the unique solution of
\begin{equation}\label{def:hN}
\sum_{t\geq1} K_{\gb}(t)e^\gd e^{-h_Nt} e^{-\frac {c_1}N t^{7/2}} \; = \; 1 \;.
\end{equation}
Notice that $h_N\leq h_\gb(\gd)$ and $h_N\to h_\gb(\gd)$ as $N\to\infty$. Plugging these notations into \eqref{eq:lbENq:2}, we have
\begin{equation}\label{eq:lbENq:3}
E^{\,\circ}_{N,q} \; \geq \; e^{h_NN} Q_N(N\in\tau) \;,
\end{equation}
where $\tau$ is a renewal process in $\N$ whose inter-arrival distribution is given by $Q_N$.

\begin{lem}\label{lem:tdhhN:conv}
There is $C_\cntc>0$ such that for all $N\in\N$,
\begin{equation}
0\;\leq\; h_\gb(\gd) -h_N\;\leq\; \frac{C_{\arabic{cst}}}N\;.
\end{equation}
\end{lem}
\begin{proof}
We already stated that $h_N\leq h_\gb(\gd)$ for all $N\in\N$. Using that $Q_\infty$ and $Q_N$ are probability distributions on $\N$, 
we write
\begin{equation*}
\sum_{t\geq1} K_{\gb}(t)e^\gd e^{-h_\gb(\gd) t} \big(e^{(h_\gb(\gd)-h_N)t}-1\big) \; = \; \sum_{t\geq1} K_{\gb}(t)e^\gd e^{-h_N t} \big(1-e^{-\frac{c_1}N t^{7/2}}\big) \;\geq\; 0 \;.
\end{equation*}
We bound the left hand side from below using $e^{(h_\gb(\gd)-h_N)t}-1\geq (h_\gb(\gd)-h_N)t\geq h_\gb(\gd)-h_N$ for all $t\geq1$ and \eqref{def:Qinfty}, so
\begin{equation*}
h_\gb(\gd)-h_N\;\leq\; \sum_{t\geq1} K_{\gb}(t)e^\gd e^{-h_N t} \big(1-e^{-\frac{c_1}N t^{^{7/2}}}\big)\;.
\end{equation*}
Since $\lim_{N\to \infty} h_N= h_\gb(\gd)$, we have some $N_0\in\N$ such that $h_N\geq h_\gb(\gd)/2$ for all $N\geq N_0$. Moreover $1-e^{-\frac cN t^{7/2}}\leq \frac{c_1}N t^{7/2}$, so for all $N\geq N_0$,
\begin{equation*}
h_\gb(\gd)-h_N\;\leq\; \frac{c_1}N \sum_{t\geq1} K_{\gb}(t)e^\gd e^{-(h_\gb(\gd)/2)\, t} \,t^{7/2} \;,
\end{equation*}
and that sum is finite, which concludes the proof.
\end{proof}
Applying this lemma to \eqref{eq:lbENq:3}, we obtain
\begin{equation}\label{eq:lbENq:4}
E^{\,\circ}_{N,q} \; \geq \; C_\cntc e^{h_\gb(\gd) N} Q_N(N\in\tau) 
\end{equation}
and therefore the proof will be complete once we show that $\liminf_{N} Q_N(N\in\tau))>0$. Applying the main theorem from \cite{Ney81}, there is $M$ a random variable with distribution Geom$(b)$, $b\in(0,1)$, and a sequence $Z_i=C'+\gz_i$, $i\geq1$ where $C'>0$ and $(\gz_i)_{i\in\N}$ are i.i.d. variables with distribution Exp$(a)$, $a>0$, such that $M$ and $(Z_i)_{i\in\N}$ are independent, and
\begin{equation}\label{eq:Ney81}
\Big| Q_N(n\in\tau)-\frac1{\mu_N}\Big| \;\leq\; C_\cntc\, \bP\Big(\sum_{i=1}^M Z_i > n\Big)\;,
\end{equation}
for all $n\in\N$ and $N\in\N$, where $\mu_N:=\bE_{Q_N}[\tau_1]\in(0,\infty)$.
The key feature here is that $C', C_{\arabic{cst}}>0$, $a>0$ and $b\in(0,1)$ can be taken uniformly in $N\in\N$. 
We do not write the details here as it suffices to check the proof in \cite{Ney81} ---more precisely one has to use that $Q_N(t)\leq C_{\cntc} Q_\infty(t)$ uniformly in $N,t\in\N$ (this follows from Lemma~\ref{lem:tdhhN:conv})
 in \cite[(2.9)]{Ney81}, and $Q_N(t)\geq C_{\cntc} Q_\infty(t)$ uniformly in $N\in\N$ and finitely many $t\in\N$ in \cite[(2.13)]{Ney81}.

In particular when $n=N$, the right hand side in \eqref{eq:Ney81} decays to 0 as $N\to\infty$, so $Q_N(N\in\tau)-\frac1{\mu_N}$ converges to $0$ as $N\to\infty$. Finally, we note that $\lim_{N\to \infty} \mu_N=\mu_\infty:=\bE_{Q_\infty}[\tau_1]\in(0,\infty)$ (by dominated convergence theorem). Thus,  $Q_N(N\in\tau)\geq \frac1{2\mu_\infty}$ for $N$ sufficiently large, and this concludes the proof of Proposition~\ref{prop:ENq:AC}.
\end{proof}

\subsection{Step 4 : area-penalized wetting model for $\gd\leq \tilde \gd_c(\gb)$}\label{sec:step4}

Estimates of $E^{\,\circ}_{N,q}$ in the phase $\cD\cC$ are more involved than in $\cA\cC$. Actually we only manage to find a sharp asymptotic of $E^{\,\circ}_{N,q}$ when $\gd=0$. To lighten upcoming notations, let us define for all $\gga>0$,
\begin{equation}
E_N(\gga) \;:=\; \Ebbzero \left[e^{-\,\gga \frac{A_{N}(I)}{N}} \pt\ind_{\{I\in B^{0,+}_{N}\}}\right]\,.
\end{equation}
Recall that $\sigma_\beta^2$ is the variance of $I_1$ under $\Pbb$.
 
For any $T>0$ and $y\geq0$, let $(\cM_{s,T}^y)_{s\in[0,T]}$ denote the Brownian meander on $[0,T]$ starting from $y$ ---henceforth we will denote its law with $\bbP$ to distinguish it from $\Pbb$. When $y>0$, it has the same law as the Brownian motion starting from $y$ and conditioned to remain positive on $[0,T]$. We will omit the superscript $y$ when $y=0$, and the subscript $T$ when $T=1$. 
Define for all $\gga>0$,
\begin{equation}
J(\gga)\;:=\; \lim_{T\to\infty} \frac1T \log \bbE\left[e^{-\gga \int_{0}^{T} \cM_{s,T} \pt\dd s }\right]\;=\; - 2^{-1/3} \pt |\mathrm{a}_1|\pt \gga^{2/3}\;,
\end{equation}
where $\mathrm{a}_1$ denotes the smallest zero (in absolute value) of the Airy function. More precisely one notices $\bbE\big[e^{-\gga \int_{0}^{T} \cM_{s,T}\pt \dd s }\big]=\bbE\big[e^{-\gga\pt T^{3/2} \int_{0}^{1} \cM_s \dd s }\big]$, and the latter expectation has been computed analytically in \cite{Takacs95} (see also \cite[(209)]{Janson07}), which leads to the second identity. We claim the following.

\begin{prop}\label{convinvblock}
One has
\begin{equation}
\frac1{N^{1/3}} \log E_N(\gga) \;\underset{N\to\infty}{\longrightarrow}\; J(\sigma_\beta\pt \gga)\;,
\end{equation}
locally uniformly in $\gga\in(0,\infty)$.
\end{prop}
\begin{proof}
First, we claim that it suffices to prove the pointwise convergence. Indeed, one notices that $\gga\mapsto J(\sigma_\beta\, \gga)$ is continuous, and $\gga\mapsto \frac1{N^{1/3}} \log E_N(\gga)$ is non-increasing for any $N\in\N$. It is well-known that those assumptions put together with the pointwise convergence imply the locally uniform convergence ---see e.g. \cite[Prop.~2.1]{Resnick07} for a proof.

Notice that $\frac1{\gs_\gb}I$ is a random walk on $\frac1{\gs_\gb}\bbZ$ with variance 1, and upcoming computations still hold when replacing $\gga$ with $\gs_\gb\pt\gga$. Hence we can assume without loss generality that $\gs_\gb=1$.\smallskip

\noindent {\it Upper bound.} Let $N\in\N$ and $T>0$, and let us denote $N_T:=\lfloor TN^{2/3}\rfloor$. Set also $a_N:=\lfloor N/N_T\rfloor$. We decompose a trajectory contributing to $E_N(\gga)$ into $a_N$ blocks of length $N_T$ and a remaining block of length at most $N_T$. We apply Markov property at times $jN_T$ for $j\in\{1,\dots,a_N\}$ and we bound the contribution of the very last block by $1$ to obtain 
\begin{equation}\label{eq:prfENqDC:UB}
E_N(\gga) \;\leq\; \left(\sup_{x\in\N} \Ebbx\left[e^{-\frac\gga N A_{N_T}(I)}\ind_{\big\{I\in B_{N_T}^+\big\}}\right]\right)^{a_N},
\end{equation}
where for all $n\in\N$, $B^{+}_n:= \big\{(X_k)_{k=1}^n \in\bbZ^{n}\,;\, X_k\geq 0 \;\forall\,1\leq k\leq n\big\}$.
\begin{lem}\label{lem:Dxm}
For $\ga>0$ and $m\in\N_0$,
\begin{equation*}
x\;\mapsto\; \Ebbx\left[e^{-\ga A_{m}(I)}\,\middle|\,I\in B_{m}^+\right]\;,
\end{equation*}
is non-increasing on $\N_0$. This also holds when conditioning over $B_{m}^{0,+}$ instead of $B_{m}^+$.
\end{lem}
We postpone the proof of this lemma to Section~\ref{proof:Dxm}. 
Choosing $\ga=\gga/N$ and applying $\frac1{N^{1/3}}\log$ to \eqref{eq:prfENqDC:UB} (and because $\Pbbx(B_{m}^+)\leq 1$), we obtain
\begin{align}
\nonumber \frac1{N^{1/3}} \log E_N(\gga) \;&\leq\; \frac{a_N}{N^{1/3}}\log \left(\sup_{x\in\N} \Ebbx\left[e^{-\frac\gga N A_{N_T}(I)}\,\middle|\,I\in B_{N_T}^+ \right]\right)\\\label{eq:prfENqDC:UB:0.25}
&\leq\; \Big(1+\frac1{N_T}\Big)\frac1{T} \log \Ebbzero\left[e^{-\frac\gga N A_{N_T}(I)}\,\middle|\,I\in B_{N_T}^+ \right],
\end{align}
for all $N,T$. Moreover the process $(\frac{1}{N^{1/3}}I_{\lfloor sN^{2/3}\rfloor})_{s\in [0,T]}$ conditioned to remain non-negative converges weakly as $N\to\infty$ to the Brownian meander $(\cM_{s,T})_{s\in[0,T]}$ in the set of cadlag functions on $[0,T]$ endowed with the uniform convergence topology (see \cite[Theorem 3.2]{Bol76}), and the swiped area is a continuous function of the trajectory. Thereby,
\begin{equation}\label{eq:prfENqDC:UB:0.5} \lim_{N\to\infty} \Ebbzero\left[e^{-\frac\gga N A_{N_T}(I)}\,\middle|\,I\in B_{N_T}^+ \right] \;=\; \bbE\!\left[\exp\left( -\gga\!\int_0^T \! \cM_{s,T} \pt\dd s \right)\right],
\end{equation}
for all $\gga>0$. 
Recollecting \eqref{eq:prfENqDC:UB:0.25}, we have for any fixed $T>0$,
\begin{equation*}
\limsup_{N\to\infty} \frac1{N^{1/3}} \log E_N(\gga) \;\leq\; \frac1T \log \left(\bbE\!\left[\exp\left( -\gga\!\int_0^T \! \cM_{s,T} \pt\dd s \right)\right]\right)\,.
\end{equation*}
Choosing  $T>0$ arbitrarily large, we conclude \[\limsup_{N\to\infty} \frac1{N^{1/3}} \log E_N(\gga) \leq J(\gga)\;.\]
\smallskip

\noindent {\it Lower bound.} 
It remains to prove that 
\begin{equation}\label{liminfE}
\liminf_{N\to\infty}  \frac1{N^{1/3}} \log E_N(\gga) \geq J(\gga).
\end{equation}
Let us  settle some notations, i.e., for $\kappa_1<\kappa_2\in \R$ and $N\in \N$ we define $\cO_{\kappa_1,\kappa_2,N}:= \big[\kappa_1\, N^{\frac{1}{3}},\, \kappa_2 N^{\frac{1}{3}}\big]$. 
For $\Delta>0$, $N,n\in \N$ and $x,y\in \N_0$  we set 
\begin{align}\label{eq:defGxy}
G^{\, x,y}_{N,n}&\;:=\;\Ebbx\left[e^{-\frac\gga N A_{n}(I)}\, \ind_{\big\{I\in B_{n}^{y,+}\big\}}
\right],\\
\nonumber\text{and} \qquad \tilde G^{\, x}_{N,\Delta,n}&\;:=\;\sum_{y\in \cO_{ \Delta/2,\Delta,N}} G^{\, x,y}_{N,n}\;.
\end{align}
Let also $N\in \N$ and $T>0$, and define $N_T:=\lfloor TN^{2/3}\rfloor$ and $a_N:=\lfloor(N-2N_T)/(N_T+2)\rfloor$ (where $a_N\geq0$ as soon as $N\geq 8T^3$). We decompose a trajectory in $E_N(\gga)$ into $a_N+2$ blocks: the first block has length $N_T$, the $a_N$ following blocks have length $N_T+2$, and the final block has length $\gr N_T$ for some $\gr\in[1,2+\frac1{N_T}]$. We restrict $E_N(\gga)$ to those trajectories located inside  $\cO_{ \Delta/2,\Delta,N}$ at times $N_T$ and $(j+1) N_T+2j$ for every $j\in \{1,\dots,a_N\}$. Then, applying   Markov property at those times
we obtain 
\begin{equation}\label{corgrai}
E_N(\gga)\geq \tilde G^{\,0}_{N,\Delta,N_T}\  \big[\hat G_{N,\, \Delta,\, N_T+2}\Big]^{a_N}\ \inf_{x\in \cO_{\Delta/2,\Delta,N}}
G^{\, x,0}_{N,\pt\gr N_T}
\end{equation}
 with 
 \begin{equation}\label{defCN}
 \hat G_{N,\Delta,n}:=\inf_{x\in \cO_{\Delta/2,\Delta,N}} \tilde G^{x}_{N,\Delta,n}.
 \end{equation}
 
 We will prove \eqref{liminfE} subject to claims \ref{convGends}, \ref{boundCbU}, \ref{compun} and \ref{convmeand}
 below. Before stating those claims we need some additional notations. For $\Delta>0$, $n\in \N$ and $N\in \N$ we set
 \begin{align}\label{controlCN}
U^{\,j}_{N,\Delta,n}&:=     \Ebbzero\left[e^{-\frac\gga N A_{n}(I)}\, \ind_{\big\{I\in \tilde B^+_{n,\Delta N^{1/3}}\big\}}\, \, \ind_{\big\{I_{n}\in \tilde \cO^{\, j}_{\Delta,N}\big\}}\right] , \quad j\in \{1,2,3\},
\end{align}
with
\begin{equation*}\label{deftio}
 \tilde \cO^1_{\Delta,N}:=\Big(\tfrac{-\Delta N^{1/3}}{2},\tfrac{-\Delta N^{1/3}}{4}\Big) \quad \text{and} \quad 
 \tilde \cO^2_{\Delta,N}:=\Big[\tfrac{-\Delta N^{1/3}}{4}, 0\Big) \quad \text{and} \quad 
   \tilde \cO^3_{\Delta,N}:=\Big[0,\tfrac{\Delta N^{1/3}}{4}\Big].
 \end{equation*}
and where we introduce for $\kappa\geq 0$ and $n\in \N$,
$$\tilde B^+_{n,\kappa}:=\big\{(X_i)_{i=1}^n \in \bbZ^{n}\,;\, X_i\geq -\tfrac{\kappa}{2}, \;\forall\,1\leq i< n\;,\, X_n>-\tfrac{\kappa}{2}\big\}.$$ 
Claim~\ref{convGends} handles the first and last factors in \eqref{corgrai}. Claim~\ref{boundCbU} combined with Claim~\ref{compun} will be useful  to get rid of the infimum in the definition of 
$ \hat G_{N,\Delta,n}$ in \eqref{defCN}
and replace it by $U^{\,1}_{N,\Delta,n}+U^{\,2}_{N,\Delta,n}$. On this latter quantity one may apply an FKG inequality and 
get rid of the constraint $\{I_n\in \cO_{ \Delta/2,\Delta,N}\}$ in $ \hat G_{N,\Delta,n}$. Claim \ref{convmeand} will be used at the end of the proof to retrieve $J(\gga)$.  

 \begin{claim}\label{convGends} For $\Delta>0$ and $T>0$, one has
\begin{equation}\label{eq:convGends}
\lim_{N\to \infty} \frac{1}{N^{1/3}} \log  \tilde G^{\,0}_{N,\Delta,\gr N_T} \;=\; \lim_{N\to \infty} \frac{1}{N^{1/3}} \log \Big(\inf_{x\in \cO_{\Delta/2,\Delta,N}}
G^{\, x,0}_{N,\gr N_T}\Big)\;=\;0\,,
\end{equation}
uniformly in $\gr\in[1,3]$.
 \end{claim}

 \begin{claim}\label{boundCbU}
For $\Delta>0$ and $N,n\in \N$,
 \begin{equation}\label{bbetzero} 
 \hat G_{N,\Delta,n}\geq \exp\big(-\tfrac{\gamma \Delta n}{N^{2/3}}\big)\,  \min \big\{U^2_{N,\Delta,n}, U^3_{N,\Delta,n}\big\}\,.
 \end{equation}


 \end{claim}
 
 \begin{claim}\label{compun}
 There exists a $C>0$ such that  for $\Delta>0$ and $N,n\in \N$,
 \begin{align}\label{comppre}
 U^{\,j}_{N,\Delta,n}&\leq C \,  \exp\big(\tfrac{\gamma \Delta n}{N^{2/3}} \big) \,  U^{\,j+1}_{N,\Delta,n+1} \;, \quad \text{for}\ j\in \{1,2\}\,.
 \end{align}
 \end{claim}

 \begin{claim}\label{convmeand} One has
 \begin{equation}\label{limverm}
 \limsup_{\Delta\to 0^+}\, \limsup_{T\to \infty} \frac{1}{T} \log  \bbE\left[e^{-\gga \int_{0}^{T} \cM_{s,T}^\Delta \  \dd s }\right]\geq 
 J(\gamma)\,.
 \end{equation}
 \end{claim}
 
We resume the proof of the lower bound. We observe that by constraining the last increment of the random walk $I$ to be null in $U^{\,2}_{N,\Delta,n+1}$ we get the inequality
\begin{align}\label{UvUp}
U^{\,2}_{N,\Delta,n}&\leq c_\beta \,  U^{\,2}_{N,\Delta, n+1}.
\end{align}
  Then, using Claim \ref{compun}  for $n=N_T$ and $j=1$ and using \eqref{UvUp} for $n=TN^{2/3}$,  we obtain that there exists a $C_1>0$ such that 
  for $T,\Delta>0$ and $N\in \N$
\begin{align}\label{ineginter}
 U^{\,1}_{N,\Delta,N_T}+ U^{\,2}_{N,\Delta,N_T}&\leq C_1\,  e^{\gamma \Delta T} \,  U^{\, 2}_{N,\Delta,N_T+1}\,.
\end{align}
Then, it suffices to apply, on the one hand, Claim \ref{compun}  for $j=2$ and $n=N_T+1$ to the r.h.s. in  \eqref{ineginter} and, on the other hand, \eqref{UvUp} for  $n=N_T+1$ to the r.h.s. in  \eqref{ineginter} 
to assert that there exists a $C_2>0$ such that for $T,\Delta>0$ and $N\in \N$,
\begin{equation}\label{vboundtilde}
V_{N_T}\;:=\;  U^{\,1}_{N,\Delta,N_T}+ U^{\,2}_{N,\Delta,N_T}\;\leq\; C_2\,  e^{2\gamma \Delta T}\, \min\{U^2_{N,\Delta,N_T+2}, U^3_{N,\Delta,N_T+2}\}\,.
\end{equation}
We note that $V_{N_T}$ may be rewritten under the form
\begin{align}
V_{N_T}\;&=\; \Ebbzero\left[e^{-\frac\gga N A_{N_T}(I)}\, \ind_{\big\{I\in \tilde B^+_{N_T,\, \Delta N^{1/3}}\big\}}\, \, \ind_{\big\{I_{N_T}\leq 0 \big\}}\right] \\
\nonumber &=\; \Pbbzero\left(I\in \tilde B^+_{N_T,\Delta N^{1/3}}\right)\  \Ebbzero\left[e^{-\frac\gga N A_{N_T}(I)}\, \ind_{\big\{I_{N_T}\leq 0 \big\}}\, \Big |\, I\in \tilde B^+_{N_T,\, \Delta N^{1/3}}\right] .
\end{align}
By using the FKG inequality described in Proposition \ref{prop:FKG} with $A:= \tilde B^+_{N_T,\, \Delta N^{1/3}}$
we obtain 
\begin{align} 
 \Ebbzero\Big[& e^{-\frac\gga N A_{N_T}(I)} \, \ind_{\big\{I_{N_T}\leq 0 \big\}} \,\Big|\, I\in \tilde B^+_{N_T,\, \Delta N^{1/3}}\Big] \geq \\
\nonumber &\Ebbzero\left[e^{-\frac\gga N A_{N_T}(I)} \,\Big|\, I\in \tilde B^+_{N_T,\, \Delta N^{1/3}}\right]\    
\Pbbzero\left[I_{N_T}\leq 0  \,\Big|\, I\in \tilde B^+_{N_T,\, \Delta N^{1/3}}\right].
\end{align}
By  Donsker's invariance principle the rescaled process $(\frac{1}{N^{1/3}}I_{\lfloor sN^{2/3}\rfloor})_{s\in [0,T]}$ converges in distribution 
towards  $(B_s)_{s\in[0,T]}$ a standard Brownian motion, in the set of cadlag functions on $[0,T]$ endowed with the uniform convergence topology. Therefore, we set $m_T:=\min\{B_s, s\in [0,T]\}$ and we obtain 
\begin{equation}\label{lininfv}
\liminf_{N\to \infty} V_{N_T}\geq \ \bbP\!\left[B_T<0, \ m_T> -\tfrac{\Delta}{2}\right]  \bbE\!\left[e^{ -\gga\!\int_0^T \! B_s ds} \, \big|\, m_T> -\tfrac{\Delta}{2}\right],
\end{equation}
Finally, using  \eqref{eq:convGends}, \eqref{bbetzero}, \eqref{vboundtilde} and \eqref{lininfv} we can deduce from \eqref{corgrai} that
\begin{align}\label{liminfbs}
\nonumber \liminf_{N\to \infty}\frac{1}{N^{1/3}} \log  E_N(\gga)\geq -\frac{\log C_3}{T} &-3\gga \Delta +\frac{1}{T}\log  \bbP\!\left[B_T<0, \ m_T> -\tfrac{\Delta}{2}\right]\\
& + \frac{1}{T}\log   \bbE\!\left[e^{ -\gga\!\int_0^T \! B_s ds} \, \big|\, m_T> -\tfrac{\Delta}{2}\right].
\end{align}
Recalling that $\Delta\leq 1$ and using \cite[Prop. 8.1]{KarShr} we can state that there exists a $c>0$ such that for $T$ large enough
$  \bbP\!\left[B_T<0, \ m_T> -\tfrac{\Delta}{2}\right]\geq c\,\Delta^2/T$ which, after taking $\limsup_{T\to \infty}$ in the r.h.s. 
in \eqref{liminfbs} allows us to write that for every $\Delta>0$
\begin{align}\label{liminfbs2}
\liminf_{N\to \infty}\frac{1}{N^{1/3}} \log  E_N(\gga)\geq&-3\gga \Delta 
+\limsup_{T\to \infty}  \frac{1}{T}\log   \bbE_{\Delta/2}\!\left[e^{ -\gga\!\int_0^T \! M_s ds} \right],
\end{align}
and we conclude by using Claim \ref{convmeand} and by taking $\limsup_{\Delta\to 0^+}$ in the r.h.s. in \eqref{liminfbs2}.
This proves \eqref{liminfE} and completes the proof of \eqref{convinvblock}.

%

\end{proof}

\begin{proof}[Proof of Claim \ref{convGends}]
Let $\Delta,T>0$. First we notice that we only need to prove a lower bound on $G^{\,0,x}_{N,n}$ uniform in $x\in \cO_{\Delta/2,\Delta,N}$ as $N\to\infty$. Indeed, a time-reversal argument yields that $G^{\,x,0}_{N,n}=G^{\,0,x}_{N,n}$ for all $x\in\N_0$, $n\in\N$ (recall that $\Pbb$ is symmetric, see \eqref{def:Pgb}), and
\begin{equation}\label{conGends:eq1}
\inf_{x\in \cO_{\Delta/2,\Delta,N}} G^{\,0,x}_{N,n} \;\leq\; \tilde G^{\,0}_{N,\Delta,n} \;\leq\; 1\;.
\end{equation}
Moreover for any $x\in \cO_{\Delta/2,\Delta,N}$ and $\gr\in[1,3]$ such that $\gr N_T\in\N$, we write 
\begin{equation}\label{conGends:eq1.5}
G^{\,x,0}_{N,\gr N_T} = \Ebbx\left[e^{-\frac\gga N A_{\gr N_T}(I)}\,\middle|\,I\in B_{\gr N_T}^{0,+}\right] \Pbbx\big(B_{\gr N_T}^{0,+}\big)\;.
\end{equation}
Lemma~\ref{lem:Dxm} claims that the first factor is non-increasing in $x$. Moreover, the second factor is polynomial in $N$ uniformly in $x\in \cO_{\Delta/2,\Delta,N}$ and $\gr\in[1,3]\cap \frac1{N_T}\N$, see \eqref{eq:asymp:PosBB} below. Thereby, we deduce from \eqref{conGends:eq1.5} that
\begin{equation}\label{conGends:eq2}
G^{\,x,0}_{N,\gr N_T} \;\geq\; \frac{C_\cntc}{N^{2/3}} \pt \bE_{\gb,\Delta N^{1/3}}\left[e^{-\frac\gga N A_{\gr N_T}(I)}\,\middle|\,I\in B_{\gr N_T}^{0,+}\right],
\end{equation}
for $N\in\N$, uniformly in $x\in \cO_{\Delta/2,\Delta,N}$ and $\gr\in[1,3]\cap \frac1{N_T}\N$ ---notice that we wrote $\Delta N^{1/3}$ instead of $\lceil\Delta N^{1/3}\rceil$, and we will omit all ceil functions henceforth to lighten notations. Recalling that $\frac12TN^{2/3}\leq \gr N_T\leq 3TN^{2/3}$ and $\gr\in[1,3]$, \eqref{conGends:eq2} implies
\begin{equation}\label{conGends:eq2bis}
G^{\,x,0}_{N,\gr N_T} \;\geq\; \frac{C_{\arabic{cst}}}{N^{2/3}} \bE_{\gb,\pt\Delta\sqrt{\frac2T}\sqrt{\gr N_T}}\left[e^{- \gga(\frac{3T}{\gr N_T})^{3/2} A_{\gr N_T}(I)}\,\middle|\,I\in B_{\gr N_T}^{0,+}\right].
\end{equation}
Furthermore, it is proven in \cite[Thm.~2.4]{CC13} that a properly rescaled random walk of length $n$ starting from $c \sqrt{n}$, $c\in\R_+$, conditioned to remain non-negative and end in $0$, converges in distribution as $n\to\infty$ to a Brownian bridge starting from $c$, conditioned to remain non-negative and to end in $0$. Moreover $\gr N_T\geq \tfrac12 TN^{2/3}$ for all $\gr\in[1,3]$, so the expectation in the r.h.s of \eqref{conGends:eq2bis} converges as $N\to\infty$ to some positive constant, uniformly in $\gr\in[1,3]$. Recollecting \eqref{conGends:eq1}, this concludes the proof of the claim.
\end{proof}

\begin{proof}[Proof of Claim \ref{boundCbU}]
Recall \eqref{defCN} and note that 
\begin{equation}\label{expraltiG}
\tilde G^{\, x}_{N,\Delta,n}=\Ebbx\left[e^{-\frac\gga N A_{n}(I)}\, \ind_{\big\{I\in \tilde B_{n,0}^{+}\big\}} \, \ind_{\big\{I_n\in \cO_{\Delta/2,\Delta,N}\big\}}
\right].
\end{equation}
Let us first consider the case  $x\in \cO_{\frac{\Delta}{2}, \frac{3\Delta}{4}, N}$.
We observe that if a trajectory $I:=(I_i)_{i=0}^{n}$ satisfies $ I_0=0$, $I\in \tilde B^+_{n,\Delta N^{1/3}}$ and  $I_n\in\cO_{0,\Delta/4,N}$
then $x+I:=(x+I_i)_{i=0}^n$ satisfies  $x+I\in \tilde B_{n,0}^{+}$ and 
$x+I_n\in  \cO_{\frac{\Delta}{2},\Delta,N}$.
As a consequence 
\begin{align}\label{expraltiG2}
\nonumber \tilde G^{\, x}_{N,\Delta,n}&\geq\Ebbzero\left[e^{-\frac\gga N A_{n}(x+I)}\, \ind_{\big\{I\in \tilde B^+_{n,\Delta N^{1/3}}\}}   \ind_{\{I_n\in\cO_{0,\Delta/4,N}\}}
\right]\\
&\geq  \exp\big(-\tfrac{\gamma \Delta n}{N^{2/3}}\big)\, U^3_{N,\Delta,n}\;.
\end{align}
The case  $x\in \cO_{\frac{3\Delta}{4},\, \Delta, N}$ is taken care of similarly.  The only difference is that we consider 
$I:=(I_i)_{i=0}^{n}$ satisfies $ I_0=0$, $I\in \tilde B^+_{n,\Delta N^{1/3}}$ and  $I_n\in\cO_{-\frac{\Delta}{4},0,N}$ such that 
$x+I\in \tilde B_{n,0}^{+}$ and $x+I_n\in  \cO_{\frac{\Delta}{2},\Delta,N}$. Then, 
\begin{align}\label{expraltiG3}
\nonumber \tilde G^{\, x}_{N,\Delta,n}&\geq\Ebbzero\left[e^{-\frac\gga N A_{n}(x+I)}\, \ind_{\big\{I\in \tilde B^+_{n,\Delta N^{1/3}}\}}   \ind_{\{I_n\in\cO_{-\frac{\Delta}{4},0,N}\}}
\right]\\
&\geq  \exp\big(-\tfrac{\gamma \Delta n}{N^{2/3}}\big)\, U^2_{N,\Delta,n}\;,
\end{align}
and this completes the proof of the claim.
%
\end{proof}

\begin{proof}[Proof of Claim \ref{compun}]
We will focuss on proving the Claim for $j=2$. The case $j=1$ can be taken care of similarly. 
We decompose  $U^{\,2}_{N,\Delta,n}$ depending on the time  $\tau$ at which the trajectory is above the $x$-axis for the last time before time $n$, that is $\tau:=\; \max\{i\leq n\colon\, I_i\geq 0\}$. This gives
\begin{equation}\label{Uwithtau}
U^{\,2}_{N,\Delta,n}=\sum_{k=0}^{n-1} U^{\,2,k}_{N,\Delta,n}:=
\sum_{k=0}^{n-1}     \Ebbzero\left[e^{-\frac\gga N A_{n}(I)}\, \ind_{\big\{I\in \tilde B^+_{n,\Delta N^{1/3}}\big\}}\, \ind_{\{\tau=k\}} \, \ind_{\big\{I_{n}\in \tilde \cO^{\, 2}_{\Delta,N}\big\}}\right].
\end{equation}
For $k\in \{0,\dots,n-1\}$ we partition the trajectories contributing $U^{\,2,k}_{N,\Delta,n}$ depending on the values $z$ and $y$ taken by 
$I_\tau$ and $I_{\tau+1}$, respectively. This gives 
\begin{align}\label{defaltU3}
\nonumber U^{\,2,k}_{N,\Delta,n}:= \sum_{z\geq 0}\sum_{y=1}^{\Delta N^{1/3}/2} &    \Ebbzero\left[e^{-\frac\gga N A_{k}(I)}\, \ind_{\big\{I\in \tilde B^+_{k,\Delta N^{1/3}}\big\}}\, \ind_{\{I_k=z\}}\right]\,  \Pb_{\gb}(I_1=-z-y)\\
&  \Eb_{\gb,-y}\left[e^{-\frac\gga N A_{n-k-1}(I)}\, \ind_{\big\{-\frac{\Delta N^{1/3}}{2} \leq I_i<0,\  i\leq n-k-1\big\}}\,  \ind_{\big\{I_{n-k-1}\in \tilde \cO^{\, 2}_{\Delta,N}\big\}}\right].
\end{align}
We observe that $ \Pb_{\gb}(I_1=-z-y)=c_\beta \,  \Pb_{\gb}(I_1=-z)\,  \Pb_{\gb}(I_1=y)$.
Since the increments of $I$ have a symmetric law, we can rewrite the last expectation of the r.h.s. in \eqref{defaltU3} as
\begin{align}\label{defaltU4}
\nonumber  \Eb_{\gb,y}&\left[e^{-\frac\gga N A_{n-k-1}(-I)}\, \ind_{\big\{-\frac{\Delta N^{1/3}}{2} \leq -I_i<0,\  i\leq n-k-1\big\}}\,  \ind_{\big\{-I_{n-k-1}\in \tilde \cO^{\, 2}_{\Delta,N}\big\}}\right]\\
&\leq \exp\left(\frac{\gamma\Delta n}{N^{2/3}}\right) \Eb_{\gb,y}\left[e^{-\frac\gga N A_{n-k-1}(I)}\, \ind_{\big\{0< I_i\leq \frac{\Delta N^{1/3}}{2},\  i\leq n-k-1\big\}}\,  \ind_{\big\{I_{n-k-1}\in \tilde \cO^{\, 3}_{\Delta,N}\setminus\{0\}\big\}}\right].
 \end{align}
where we have used that   any trajectory  
$I=(I_i)_{i=0}^{n-k-1}$ that contributes the expectation of the l.h.s. in \eqref{defaltU4} satisfies 
$A_{n-k-1}(I)=-A_{n-k-1}(-I)\leq \tfrac12 \Delta N^{1/3} (n-k-1).$
Thus, combining \eqref{defaltU3} with \eqref{defaltU4} we obtain
\begin{align}\label{defaltU5}
\nonumber  U^{\,2,k}_{N,\Delta,n}\leq c_\beta \, &\exp\left(\frac{\gamma\Delta n}{N^{2/3}}\right) \,  \sum_{z\geq 0}\sum_{y=1}^{\Delta N^{1/3}/2}     \Ebbzero\left[e^{-\frac\gga N A_{k}(I)}\, \ind_{\big\{I\in \tilde B^+_{k,\Delta N^{1/3}}\big\}}\, \ind_{\{I_k=z\}}\right]\, \Pb_{\gb}(I_1=-z)\\  
\nonumber & \quad \,   \Pb_{\gb}(I_1=y)\, 
 \Eb_{\gb,y}\left[e^{-\frac\gga N A_{n-k-1}(I)}\, \ind_{\big\{0< I_i\leq \frac{\Delta N^{1/3}}{2},\  i\leq n-k-1\big\}}\,  \ind_{\big\{I_{n-k-1}\in \tilde \cO^{\, 3}_{\Delta,N}\setminus\{0\}\big\}}\right]\\
 \leq c_\beta \, & \exp\left(\frac{\gamma\Delta n}{N^{2/3}}\right) 
  \Ebbzero\left[e^{-\frac\gga N A_{n+1}(I)}\, \ind_{\big\{I\in \tilde B^+_{n+1,\Delta N^{1/3}}\big\}}\, \ind_{\{\tilde \tau=k+1\}}
  \,   \ind_{\big\{I_{n+1}\in \tilde \cO^{\, 3}_{\Delta,N}\big\}}\right],
\end{align}
where $\tilde \tau:=\max\{0\leq j\leq n+1\colon I_j=0\}$ is the last time before time $n+1$ at which $I$ touches the $x$-axis.
Therefore,
\begin{align*}
\sum_{k=0}^{n-1} U^{\,2,k}_{N,\Delta,n}&\leq c_\beta \, \exp\left(\frac{\gamma\Delta n}{N^{2/3}}\right) 
 \Ebbzero\left[e^{-\frac\gga N A_{n+1}(I)}\, \ind_{\big\{I\in \tilde B^+_{n+1,\Delta N^{1/3}}\big\}}\,
  \,   \ind_{\big\{I_{n+1}\in \tilde \cO^{\, 3}_{\Delta,N}\big\}}\right]\\
& \leq   c_\beta \, \exp\left(\frac{\gamma\Delta n}{N^{2/3}}\right) \, U^{\,3}_{N,\Delta,n+1}\,,
\end{align*}
and this completes the proof of the claim.
 \end{proof}

\begin{proof}[Proof of Claim \ref{convmeand}.]
For conciseness we set, for $\Delta, T>0$ 
\begin{equation}\label{defRdelta}
R_\Delta(T):= \bbE\left[e^{-\gga \int_{0}^{T} \cM_{s,T}^\Delta\,  \dd s }\right].
\end{equation}
We assume by contradiction that there exists an $\gep>0$ and a $\Delta_0\in (0,1]$ such that 
for every $\Delta\in (0,\Delta_0]$
\begin{equation}\label{contrad}
\limsup_{T\to \infty} \frac{1}{T} \log R_\Delta(T)\leq J(\gamma)-\gep\,.
\end{equation}
For $\Delta, T>0$, we set $\tau_{\Delta}:=\inf\{s\geq 0\colon\, \cM_{s,T}=\Delta\}$ ($\tau_\Delta$ depends on $T$ as well but we omit it for conciseness). We state a small ball inequality that will be proven at the end of the present proof, i.e., there exists $c_1, c_2>0$ such that for $T>0$,
$\eta\in (0,1]$ and $\Delta>0$,
\begin{align}\label{taudelta}
\bbP\left(\tau_{\Delta}\geq \eta T \right)&=\bbP\Big[\max_{s\in [0,\eta T]} \cM_{s,T}\leq \Delta\Big]\\
\nonumber &=\bbP\Big[\max_{s\in [0,\eta]} \cM_{s} \leq \Delta/\sqrt{T} \Big] \leq \frac{c_1}{\sqrt{\eta}}e^{-\frac{c_2}{\Delta^2} \eta T},
\end{align}
where we have used a standard scaling property of Brownian meander to write the second equality in \eqref{taudelta}.


At this stage, we choose $\eta>0$ and $\Delta \in (0,\Delta_0]$ such that 
\begin{equation}\label{condeta}
(J(\gamma)-\tfrac{\gep}{4})\, (1-\eta)\leq J(\gamma)-\tfrac{\gep}{8} \quad \text{and} \quad -\tfrac{c_2 \eta}{ \Delta^2} \leq J(\gamma)-\tfrac{\gep}{2}.
\end{equation}
Applying Markov property at time $\tau_\Delta$ we obtain 
\begin{align}\label{decomtau}
\bbE\left[e^{-\gga \int_{0}^{T} \cM_{s,T} \dd s }\right]&\leq \bbP\left(\tau_{\Delta} \geq \eta T \right)+\bbE\left[\ind_{\{\tau_\Delta<\eta T\}}\, 
R_\Delta(T-\tau_\Delta)\right]
\end{align} 
Since  $ \Delta\in (0,\Delta_0]$
we apply \eqref{contrad} and there exists a $T_0>0$ such that $R_{\Delta}(T)\leq e^{(J(\gamma)-\gep/4)T}$
for $T\geq T_0$.
It remains to apply \eqref{decomtau} with $T\geq T_0/(1-\eta)$ in combination with \eqref{taudelta} and \eqref{condeta} to obtain
\begin{align}\label{decomtau2}
\nonumber \bbE\left[e^{-\gga \int_{0}^{T} \cM_s \dd s }\right]&\leq c_1 e^{(J(\gamma)-\frac{\gep}{2})T}+
\bbE\left[\ind_{\{\tau_\Delta<\eta T\}}\, R_{\Delta}(T-\tau_\Delta)\right]\\
\nonumber &\leq c_1 e^{(J(\gamma)-\frac{\gep}{2})T}+
\bbE\left[\ind_{\{\tau_\Delta<\eta T\}}\, e^{(J(\gamma)-\gep/4) (T-\tau_{\Delta})}\right]\\
& \leq c_1 e^{(J(\gamma)-\frac{\gep}{2})T}+e^{(J(\gamma)-\gep/8) T}.
\end{align} 
Taking $\frac{1}{T}\log$ on both sides in \eqref{decomtau2} and letting $T\to \infty$ we obtain the contradiction
$J(\gamma)\leq J(\gamma)-\gep/8$. This completes the proof of the claim.

\medskip

 Let us quickly sketch the proof of the inequality in \eqref{taudelta}. We use that $(\cM_{s})_{s\in [0,1]}$ is the limit in distribution of 
$(B_s)_{s\in [0,1]}$ conditioned on $\{m_1>-u\}$ as $u\to 0^{+}$ (see \cite[Section 2]{Dur77}). Therefore, for $\Delta>0$, $u\in(0,\Delta)$ and 
$\eta\in (0,1)$, 
\begin{equation}\label{smallball}
\bbP\Big[\max_{s\in [0,\eta]} \cM_{s} \leq \Delta \Big]
=\lim_{u\to 0^+} \bbP\Big[\max_{s\in [0,\eta ]} B_{s} \leq \Delta\, \big |\,  m_1>-u\Big].
\end{equation}
By applying Markov property at time $\eta/2$, we can bound from above the probability in the r.h.s. in \eqref{smallball}  by
\begin{equation}\label{smallball2}
\frac{1}{\bbP(m_1>-u)} \, \bbP(B_{s}\in (-u,\Delta ], \forall s\in [0,\eta/2]) 
\sup_{x\in (-u, \Delta]} \bbP_x(B_{s}\in (-u, \Delta],  \forall s\in [0,1-\eta/2 ]),
\end{equation}
and since a Brownian motion $(B_s)_{s\geq 0}$ of law $\bbP_x$ has the same law as $(x+B_s)_{s\geq 0}$ with $(B_s)_{s\geq 0}$ of law $\bbP_0$, \eqref{smallball2} is also smaller than 
\begin{equation}\label{smallball3}
\frac{\bbP(m_{\eta/2}>-u)}{\bbP(m_1>-u)} \, 
\bbP\Big[\max_{s\in[0,\eta/2]}| B_{s}|\leq  2\Delta\Big]
\end{equation}

At this stage, the inequality in the r.h.s. in \eqref{taudelta} is obtained by combining 
\eqref{smallball3} with the following two results:
\begin{equation}\label{reflectprinc}
\bbP(m_{t}>-u)=\sqrt{\frac{2}{\pi}} \int_0^{u/\sqrt{t}} e^{-\frac{x^2}{2}}\, dx, \quad \text{for}\  t,u>0
\end{equation}
and there exist $c_1, c_2>0$ such that for every $\kappa>0$
\begin{equation}\label{smallballbro}
\bbP(\max_{ s\in [0,1 ]} |B_{s}| \leq \kappa )\leq c_1 e^{-c_2/\kappa^2}.
\end{equation}
Note that \eqref{reflectprinc} is obtained by observing that, since $B$ is symetric, $\bbP(m_{t}>-u)=\bbP(M_{t}<u)$, where 
$M_{t}$ is the maximum of $B$ on $[0,t]$ and the law of $M_t$ is well known (see e.g. \cite[Proposition 8.1]{KarShr}). Proving
\eqref{smallballbro} can be done by estimating the probability that $|B|$ is smaller than $\kappa$ at times $\{j/k, \, 1\leq j\leq k\}$
with $k=\lfloor1/\kappa^2\rfloor$ and by applying Markov property at those times.

\end{proof}

\subsection{Step 5 : horizontal extension inside the collapsed phase}
We now prove that in the collapsed phase, the typical horizontal extension of the polymer is of order $\sqrt{L}$ for $L$ large. For any interval $I\subseteq\R$, define
\begin{equation}
\widetilde{Z}_{L,\gb,\gd}^{\,\circ,+}(I) \;=\; \frac{1}{c_\gb e^{\gb L}} Z_{L,\gb,\gd}^{\,\circ,+}(I) \;:=\; \sumtwo{1\leq N\leq L/2,}{N\in I\pt\cap\pt\N} (\gGa_\gb)^{2N} D^{\,\circ}_{N,q_N^L}\;,
\end{equation}
where $q^L_N$ and $D^{\,\circ}_{N,q^L_N}$ are defined in Proposition~\ref{prop:PEZ}.

\begin{lemma}\label{lem:exthor}
Let $(\beta,\delta)\in (\cC_{\mathrm{bead}})^o$, there exists $(a_1,a_2)\in (0,\infty)^2$ such that 
\begin{equation}\label{resthext}
\lim_{L\to \infty} \frac{Z_{L,\beta,\delta}^{o,+} ([a_1,a_2] \sqrt{L})}{Z_{L,\beta,\delta}^{o,+}}=1
\end{equation}

\end{lemma}
\begin{proof}
Recall \eqref{PEZ} and \eqref{def:DNL}. By relaxing all constraints on $S,I$ but $\{I\in B_N^{0,+}\}$ in \eqref{def:DNL}, one has the obvious bound $D^{\,\circ}_{N,q}\leq Z_{\wet,N}^{\gb,\gd}$ (recall \eqref{def:Zwetting}). Moreover Proposition~\ref{prop:wetpoint}~(ii) implies that there exists some $K>0$ such that $Z_{\wet,N}^{\gb,\gd}\leq K e^{h_\gb(\gd)N}$ for $N\geq1$. Thereby,
\begin{equation*}
\widetilde Z_{L,\beta,\delta}^{o,+} ([a_2\sqrt{L},L])=\sum_{N=a_2 \sqrt{L}}^{L/2} (\Gamma_\beta)^{2N} D^{\,\circ}_{N,q_{N}^L}
\leq  \sum_{N=a_2 \sqrt{L}}^{L/2} (\gGa_\gb)^{2N} \pt K \pt e^{h_\gb(\gd)N}\;.
\end{equation*}
In $(\cC_{\mathrm{bead}})^o$, one has $2\log \gGa_\gb +h_\gb(\gd)<0$ (recall \eqref{eq:cc:bead}), so we conclude that there exist $c_1,c_2>0$ such that
\begin{equation}\label{fbz}
\widetilde Z_{L,\beta,\delta}^{o,+} ([a_2\sqrt{L},L])\leq c_1 
e^{-c_2 a_2 \sqrt{L}}\;,
\end{equation}
for all $L\geq1$ and $a_2\in\frac{\N}{\sqrt{L}}$.

Regarding trajectories with an horizontal extension smaller than $a_1 \sqrt{L}$, we use $\Gamma_\beta\leq 1$ and (\ref{PEZ}--\ref{def:DNL}) to write
\begin{align}\label{fbz2}
 \widetilde Z_{L,\beta,\delta}^{o,+} ([1,a_1\sqrt{L}])\leq e^{\gd a_1 \sqrt{L}} \sum_{N\leq a_1 \sqrt{L}} \Pbb(A_{N+1}(S)\geq \tfrac{L}{2}-N).
 \end{align}
The inequality $A_{N+1}(S)\geq \frac{L}{2}-N$ implies that $\max\{S_i, i\leq N\}\geq \frac{L/2-N}{N+1}$. Therefore, for $L$ large enough   we can assert that for every  $a_1\leq 1$  and  
 $N\leq a_1 \sqrt{L}$
\begin{equation}\label{lda}
 \Pbb(A_{N+1}(S)\geq \tfrac{L}{2}-N)\leq  \Pbb\big(\sum_{i\leq a_1\sqrt{L}} |S_i-S_{i-1}|\geq \tfrac{\sqrt{L}}{4 a_1}\big).
\end{equation}
Since $|S_1-S_0|$ (of law $\Pbb$) has finite small exponential moments, Chernov's bound guarantees us that there exists a function 
$\rho:(0,\infty)\mapsto \R^+$ satisfying $\lim_{a_1\to 0^+} \rho(a_1)=\infty$ such that 
the r.h.s. in \eqref{lda}
is bounded above by $e^{-\rho(a_1)\sqrt{L}}$. Therefore, for $L$ large enough and $a_1\leq 1$,
\begin{align}\label{fbz22}
 \widetilde Z_{L,\beta,\delta}^{o,+} ([1,a_1\sqrt{L}])\leq \sqrt{L} e^{-(\rho(a_1)-\gd a_1)\sqrt{L}}.
 \end{align}
 At this stage it remains to display a lower bound on $Z_{L,\beta,\delta}^{o,+}$. To that aim we only consider the term indexed by 
 $N=\sqrt{L}$ in the sum in \eqref{PEZ}. Applying Proposition \eqref{prop:estimD} with $q=\frac{1}{2}-\frac{1}{\sqrt{L}}$ and using that $g$ is $\cC^1$ we deduce that there exists a $C>0$ such that 
\begin{align}\label{upperb}
\widetilde Z_{L,\beta,\delta}^{o,+} &\geq \frac{C}{L} \, e^{-g(\frac{1}{2},0) \sqrt{L}} \, E^{\,\circ}_{\sqrt{L},\, \frac{1}{2}-\frac{1}{\sqrt{L}}}.
\end{align}

It remains to bound  $E^{\,\circ}_{\sqrt{L},\,  \frac12-\frac{1}{\sqrt{L}}}$ from below by choosing the trajectory $(I_i)_{i=0}^{\sqrt{L}}$ that sticks to zero, i.e., $E^{\,\circ}_{\sqrt{L},\,\frac{1}{2}-\frac{1}{\sqrt{L}}}\geq (e^\delta/c_\beta)^{\sqrt{L}}$. We set  $\kappa:=\delta-g(1/2,0)-\log c_\beta$ and we obtain that for $L$ large enough,
\begin{align}\label{upperb2}
\widetilde Z_{L,\beta,\delta}^{o,+} &\geq \frac{C}{L} e^{\kappa \sqrt{L}}.
\end{align}
Combining \eqref{fbz}, \eqref{fbz22} and \eqref{upperb2}, it suffices to choose $a_1$ (resp. $a_2$) small enough
(resp. large enough) for \eqref{resthext} to hold.
\end{proof}

\subsection{Step 6 : proof of Theorem~\ref{th:partfunbead}}
We finally have all required estimates to prove Theorem~\ref{th:partfunbead}. Assume $\gb>\gb_c$ and $\gd<\gd^{\pt\circ}_c(\gb)$. Applying Proposition~\ref{prop:PEZ} and Lemma~\ref{lem:exthor}, we can restrict the sum in \eqref{PEZ} to $N\in[a_1,a_1]\sqrt{L}$. In particular this implies that $q^L_N=\frac{L-2N}{2N^2}\in [\tfrac{1}{2a_2^2}-\tfrac{1}{a_1\sqrt{L}},\frac{1}{2a_1^2}]$, and $\tfrac{1}{2a_2^2}-\tfrac{1}{a_1\sqrt{L}}>0$ for $L$ sufficiently large. Thereby Proposition~\ref{prop:estimD} yields
\begin{equation}\label{eq:prfthm2:gen}
Z^{\,\circ,+}_{L,\gb,\gd} \;\asymp\; \frac{e^{\gb L}}{L} \sum_{N=a_1\sqrt L}^{a_2\sqrt L} \big(\gGa_\gb\big)^{2N} e^{-N g\big(\tfrac{L}{2N^2},0\big)} E^{\,\circ}_{N,q_N^L}\;,
\end{equation}
uniformly in $L\in\N$ sufficiently large, and where we also used that $g$ is $\cC^1$, so $g(q^L_N,0)=g(\frac{L}{2N^2},0)+O(1/N)$ uniformly in $N\in[a_1,a_1]\sqrt{L}$.

\subsubsection*{Case $\tilde\gd_c(\gb)<\gd<\gd^{\pt\circ}_c(\gb)$} We apply Proposition~\ref{prop:ENq:AC} to \eqref{eq:prfthm2:gen}, to write
\begin{equation}\label{eq:prfthm2}
Z^{\,\circ,+}_{L,\gb,\gd} \;\asymp\; \frac{e^{\gb L}}{L} \sum_{N=a_1\sqrt L}^{a_2\sqrt L} e^{\sqrt{L} \pt \phi(\frac{N}{\sqrt{L}})}\;,
\end{equation}
where we define for all $a>0$,
\begin{equation}\label{eq:defphi}
\phi(a)\;:=\; a \bigg(2\log\gGa_\gb + h_\gb(\gd) - g\Big(\frac{1}{2a^2},0\Big) \bigg)\;.
\end{equation}
Notice that $\phi$ is $\cC^2$ and negative on $(0,\infty)$ (recall that $g$ is non-negative and $2\log\gGa_\gb +h_\gb(\gd)<0$ in $(\cC_{\mathrm{bead}})^o$). We claim that it is strictly concave, and that it reaches its maximum at some $\tilde a\in(0,\infty)$. Indeed we have
\begin{align*}
\phi'(a)\;&=\; 2\log\gGa_\gb + h_\gb(\gd) - g\Big(\frac1{2a^2},0\Big) +\frac1{a^2}\partial_qg\Big(\frac1{2a^2},0\Big)\;,\\
\phi''(a)\;&=\; -\frac1{a^3} \partial_qg\Big(\frac1{2a^2},0\Big) -\frac1{a^5}\partial_q^2g\Big(\frac1{2a^2},0\Big)\;.
\end{align*}
Recall that $\partial_qg\big(\frac2{a^2},0\big)=\tilde h_0^{(2a^{-2},0)}>0$ (see \eqref{eq:gradg} and Section~\ref{proof:DNLSI}), and $g$ is convex hence $\partial_q^2g\big(\frac2{a^2},0\big)\geq0$; therefore $\phi''(a)<0$ for all $a\in(0,\infty)$. Besides, $\phi'(a)$ can be written
\begin{equation}\label{eq:prfthm2:phi'}
\phi'(a)\;=\; 2\log\gGa_\gb + h_\gb(\gd) + \frac1{2a^2} \tilde h_0^{2a^{-2},0} + \frac12 \cL_\gL\Big(\tilde \bh^{2a^{-2},0}\Big)\;.
\end{equation}
Recall that $2\log\gGa_\gb + h_\gb(\gd)<0$ whenever $(\gb,\gd)\in(\cC_{\mathrm{bead}})^{o}$. Moreover $\tilde\bh$ is continuous and $\tilde\bh^{0,0}=(0,0)$, so $\phi'(a)<0$ for $a$ sufficiently large. As $a\searrow0$, one has $\tilde h_0^{2a^{-2},0}\to\frac{\gb}{2}>0$ and $\cL_\gL\big(\tilde \bh^{2a^{-2},0}\big)\geq0$; hence $\phi'(a)>0$ for $a$ sufficiently small.

This proves that $\sup_{(0,\infty)}\phi=\phi(\tilde a)$ for some $\tilde a = \tilde a(\gb,\gd)\in(0,\infty)$. Provided that $a_1$ (resp. $a_2$) is sufficiently small (resp. large), we can assume $\tilde a\in(a_1,a_2)$. Because $f$ is strictly concave and $\cC^2$ on $[a_1,a_2]$, there are constants $c,c'>0$ such that
\begin{equation}\label{eq:phiconv}
-c(a-\tilde a)^2 \;\leq\; \phi(a)-\phi(\tilde a) \;\leq\; -c'(a-\tilde a)^2\;,
\end{equation}
for all $a\in[a_1,a_2]$. Thus we can rewrite \eqref{eq:prfthm2} as
\begin{equation}\label{eq:prfthm2:1}\begin{aligned}
&C_\cntc \frac{e^{\gb L}}{L} e^{\sqrt{L}\pt\phi(\tilde a)} \sum_{N=a_1\sqrt L}^{a_2\sqrt L} e^{-c\sqrt{L} \big(\frac{N}{\sqrt{L}}-\tilde a\big)^2}  \\
&\quad\;\leq\; Z^{\,\circ,+}_{L,\gb,\gd} \;\leq\; C_\cntc \frac{e^{\gb L}}{L} e^{\sqrt{L}\pt\phi(\tilde a)} \sum_{N=a_1\sqrt L}^{a_2\sqrt L} e^{-c'\sqrt{L} \big(\frac{N}{\sqrt{L}}-\tilde a\big)^2}\;.
\end{aligned}\end{equation}
Finally, we observe that for any $c>0$ and $b_1<0<b_2$, one has
\begin{equation}\label{eq:riemannsum}
\sumtwo{b_1\sqrt L\leq N\leq b_2\sqrt L,\vspace{1pt}}{N\in\bbZ} e^{-c \frac{N^2}{\sqrt{L}}} \;\asymp\; L^{1/4}\;.
\end{equation}
Indeed, let $R>0$, and define $I_{R,L}=[-R\pt L^{1/4},R\pt L^{1/4}]$. We decompose the above sum into $S^{L,\gb,R}_1+S^{L,\gb,R}_2$, where
\begin{align*}
S^{L,\gb,R}_1\;:= \sum_{N\in I_{R,L}\cap\pt\bbZ} e^{-c \frac{N^2}{\sqrt{L}}}\;,
 \qquad \quad \text{and} \qquad S^{L,\gb,R}_2\;:= \sumtwo{N\in[b_1\sqrt{L},b_2\sqrt{L}]\pt\setminus\pt I_{R,L},}{N\in\bbZ} e^{-c \frac{N^2}{\sqrt{L}}}\;.
\end{align*}
Let us first handle $S^{L,\gb,R}_1$. With a Riemann sum approximation, we have
\begin{equation}\label{eq:prfthm2:lem}
L^{-1/4} \pt S^{L,\gb,R}_1 = L^{-1/4}\sum_{N=-RL^{1/4}}^{RL^{1/4}} e^{-c \big(\frac{N}{L^{1/4}}\big)^2} \;\underset{L\to\infty}\longrightarrow \; \int_{-R}^R e^{-cx^2}\dd x\;,
\end{equation}
therefore $S^{L,\gb,R}_1\sim L^{1/4} \int_{-R}^R e^{-cx^2}\dd x$ as $L\to\infty$. Regarding $S^{L,\gb,R}_2$, we have
\begin{equation*}
\sumtwo{N\in[b_1\sqrt L,b_2\sqrt L]\pt\setminus\pt I_{R,L},}{N\in\bbZ} e^{-c \frac{N^2}{\sqrt{L}}} \;\leq\; 2\!\!\! \sum_{N> R\pt L^{1/4}} e^{-c\big(\frac{N}{L^{1/4}}\big)^2}\;,
\end{equation*}
and a Riemann sum approximation yields that
\begin{equation}\label{eq:prfthm2:lem:2}
\lim_{L\to\infty} \; L^{-1/4}\sum_{N>R\pt L^{1/4}} e^{-c\big(\frac{N}{L^{1/4}}\big)^2} \;=\;\int_R^\infty e^{-cx^2} \dd x\;.
\end{equation}
Hence there is $C_R>0$ such that $S^{L,\gb,R}_2 \leq C_R L^{1/4}$ for all $L\geq1$. Together with \eqref{eq:prfthm2:lem}, this concludes the proof of \eqref{eq:riemannsum}. Plugging \eqref{eq:riemannsum} into \eqref{eq:prfthm2:1}, this proves Theorem~\ref{th:partfunbead} in $(\cA\cC)^o$.

\subsubsection*{Case $\gd=0$} We now prove the Theorem when $\gd=0$ ---then the case $0<\gd\leq\tilde\gd_c(\gb)$ will be a straightforward consequence of the other two. Let $\eps>0$. Recollecting \eqref{eq:prfthm2:gen} and Proposition~\ref{convinvblock}, and noticing that $q\mapsto J(\gs_\gb\pt\partial_qg(q,0))$ is $\cC^1$ on $(0,\infty)$; we can write the bounds 
\begin{align}
\nonumber & C_\cntc \frac{e^{\gb L}}{L} \sum_{N=a_1\sqrt L}^{a_2\sqrt L} \exp\Big(\sqrt L \pt\phi\big(\tfrac{N}{\sqrt{L}}\big)+L^{1/6} \big(\psi\big(\tfrac N{\sqrt{L}}\big)-\eps\big)\Big) \\
& \label{eq:prfthm2:2} \quad \leq \; Z^{\,\circ,+}_{L,\gb,0} \;\leq\; C_\cntc \frac{e^{\gb L}}{L} \sum_{N=a_1\sqrt L}^{a_2\sqrt L} \exp\Big(\sqrt L \pt\phi\big(\tfrac{N}{\sqrt{L}}\big)+L^{1/6} \big(\psi\big(\tfrac N{\sqrt{L}}\big)+\eps\big)\Big) \;,
\end{align}
for $L$ sufficiently large, where $\phi$ is defined in \eqref{eq:defphi}, and we define for all $a\in(0,\infty)$,
\begin{equation}\label{eq:defpsi}
\psi(a)\;:=\; a^{1/3} \pt J\big(\gs_\gb\pt\partial_qg(\tfrac{1}{2a^2},0)\big)\;<\;0\;.
\end{equation}
Notice that we can drop the constant and polynomial factors by slightly adjusting $\eps$. Let us fix some $R>0$ and define
\[
I_{R,L}^1\;:=\; \tilde a \sqrt{L} + [-RL^{1/3}, RL^{1/3}] \qquad \text{and} \quad I_{R,L}^2\;:=\; [a_1 \sqrt L, a_2 \sqrt L] \setminus I_{R,L}^1\;,
\]
so \eqref{eq:prfthm2:2} yields
\begin{equation}\label{eq:prfthm2:2.5}e^{\gb L} Q^1_{L,-\eps} \leq Z^{\,\circ,+}_{L,\gb,0} \leq e^{\gb L} (Q^1_{L,\eps} + Q^2_{L,\eps})\;,\end{equation}
where we define for any $i\in\{1,2\}$ and $\gh\in\{-\eps,\eps\}$, 
\[ Q_{L,\gh}^i \;:=\; \sum_{N\in I_{R,L}^i} \exp\Big(\sqrt L \pt\phi\big(\tfrac{N}{\sqrt{L}}\big)+L^{1/6} \big(\psi\big(\tfrac N{\sqrt{L}}\big)+\gh\big)\Big)\;.
\]
Regarding the lower bound, recall that $\psi$ is $\cC^1$ on $(0,\infty)$, so there is some $C_\cntc>0$ such that
\begin{equation}\label{eq:prfthm2:3} |\psi(\tfrac N{\sqrt L})-\psi(\tilde a)| \; \leq \; C_{\arabic{cst}} R L^{-1/6}\;,
\end{equation} uniformly in $L$ sufficiently large and $N\in I_{R,L}^1$. Hence for $L$ large,
\begin{equation}\label{eq:prfthm2:4}\begin{aligned}
Q^1_{L,-\eps} \;&\geq\; \exp\big(\sqrt{L} \pt\phi(\tilde a) + L^{1/6}(\psi(\tilde a)-\eps) - C_{\arabic{cst}}R\big) \\
&\qquad\times \sum_{N=\tilde a\sqrt L -L^{1/3}}^{\tilde a\sqrt L +L^{1/3}} \exp\Big(\sqrt L \big(\phi(\tfrac{N}{\sqrt{L}})-\phi(\tilde a)\big)\Big)\;,
\end{aligned}\end{equation}
and recalling \eqref{eq:phiconv} to \eqref{eq:prfthm2:lem:2}, the latter sum is of order $C_\cntc L^{1/4} \geq 1$ when $L\to\infty$, which yields the expected lower bound (by slightly adjusting $\eps$).

Similarly to the lower bound just above, we deduce from \eqref{eq:prfthm2:3} and from \eqref{eq:phiconv} to \eqref{eq:prfthm2:lem:2} that
\begin{equation}\label{eq:prfthm2:5} Q_{L,\eps}^1 \;\leq\; \exp\big( \sqrt{L} \pt\phi(\tilde a) + L^{1/6}(\psi(\tilde a)+2\eps) \big)\;,
\end{equation}
for $L$ sufficiently large. Regarding $Q_{L,\eps}^2$, \eqref{eq:phiconv} yields that for all $N\in I_{R,L}^2$,
\[
\sqrt{L}\pt\big(\phi(\tfrac N{\sqrt{L}})-\phi(\tilde a)\big) \;\leq\; -c'R^2L^{1/6}\;,
\]
Moreover $\psi(N/\sqrt L) +\eps \leq \tfrac12\sup_{[a_1,a_2]} \psi<0$ for all $N\in[a_1\sqrt L, a_2 \sqrt L]$, $L\in\N$ provided that $\eps$ is sufficiently small. Thereby,
\[
Q_{L,\eps}^2 \;\leq\; (a_2-a_1)\sqrt{L} \times \exp\big(\sqrt L\pt\phi(\tilde a)-c'R^2L^{1/6}\big) \;.
\]
Finally, noticing that $Q_{L,\eps}^1\geq Q_{L,-\eps}^1$ and recalling \eqref{eq:prfthm2:4}, we have $Q_{L,\eps}^2/Q_{L,\eps}^1 = o(1)$, provided that $R$ satisfies $-c'R^2<\inf_{[a_1,a_2]}\psi-\eps$. Recollecting \eqref{eq:prfthm2:2.5} this completes the proof of the upper bound, and concludes the case $\gd=0$.

\subsubsection*{Case $0<\gd\leq \tilde \gd_c(\gb)$} When $\gd\in(0,\tilde \gd(\gb)]$, we do not provide sharp estimates but we can still claim the following: first, $\gd\mapsto Z^{\,\circ,+}_{L,\gb,\gd}$ is clearly non-decreasing, so $Z^{\,\circ,+}_{L,\gb,\gd}\geq Z^{\,\circ,+}_{L,\gb,0}$, and the lower bound from the case $\gd=0$ also holds for any $\gd>0$. Secondly, the recall that $E^{\,\circ}_{N,q} \leq Z_{\wet, N}^{\gb,\gd}$ for all $q>0$, $N\in\N$ (see \eqref{def:Enq}), so Proposition~\ref{prop:wetpoint}~(ii) implies that there exists $C_\cntc>0$ (which depends on $\gb,\gd$) such that $E^{\,\circ}_{N,q} \leq C_{\arabic{cst}} $ for $N\in\N$. Thus \eqref{eq:prfthm2:gen} yields
\begin{equation*}
Z^{\,\circ,+}_{L,\gb,\gd} \;\leq\; \frac{e^{\gb L}}{L} \pt C_{\arabic{cst}} \sum_{N=a_1\sqrt L}^{a_2\sqrt L} e^{\sqrt{L} \pt \phi(\frac{N}{\sqrt{L}})} \;,
\end{equation*}
where $\phi$ is defined in \eqref{eq:defphi}, with $h_\gb(\gd)=0$ for all $\gd\in(0,\tilde \gd(\gb)]$. Then the behavior of the sum above is the same as in the case $\gd\in(\tilde \gd(\gb), \gd^{\pt\circ}_c(\gb))$, which yields the same upper bound. This fully concludes the proof of Theorem~\ref{th:partfunbead}, subject to Lemmas~\ref{lem:FKG} and \ref{lem:Dxm}, and Proposition~\ref{prop:DNLSI}.

\subsection{Proof of Corollary~\ref{corol:partfunbead}}\label{proof:traj}
In this section we prove Corollary~\ref{corol:partfunbead}, which follows directly from Theorem~\ref{th:partfunbead}. 
Let $\gb>\gb_c$. We first focus on the case $\gd\in(\tilde \delta_c(\beta), \delta^{\circ}_c(\beta))$. Recall the expression of $\phi'(a)=\phi'_{(\gb,\gd)}(a)$ in \eqref{eq:prfthm2:phi'} and the definition of $\tilde a(\gb,\gd)$, then notice that $(\gd,a)\mapsto\phi'_{(\gb,\gd)}(a)$ is $\cC^1$ on $(\tilde \delta_c(\beta), \delta^{\circ}_c(\beta))\times(0,+\infty)$ and $\phi''_{(\gb,\gd)}(\tilde a(\gb,\gd))<0$. It follows from the implicit function theorem that $\gd\mapsto \tilde a(\gb,\gd)$ is $\cC^1$ on $(\tilde \delta_c(\beta), \delta^{\circ}_c(\beta))$. Hence $\gd\mapsto\Phi(\gb,\gd)$ is $\cC^1$ on $(\tilde \delta_c(\beta), \delta^{\circ}_c(\beta))$ too.

Let $\gd\in(\tilde \delta_c(\beta), \delta^{\circ}_c(\beta))$ and $\eps>0$, and let us denote for $\ell\in\gO_L^{\,\circ,+}$,
\[ Q(\ell)\;:=\; \sum_{k=1}^{N_\ell}\ind_{\{\sum_{i=1}^k\ell_i=0\}} , \]
for conciseness. Let $t>0$ be such that $\gd+t\in(\tilde \delta_c(\beta), \delta^{\circ}_c(\beta))$, and write with Chernov's bound
\begin{align*}
\bP_{L,\beta,\delta}^{\,\circ,+} \Big( Q(\ell) \geq \big(\partial_{\gd}\Phi(\gb,\gd)+\eps\big) \sqrt{L} \Big)  \;&\leq\; \bE_{L,\beta,\delta}^{\,\circ,+}\big[e^{t \pt Q(\ell)}\big] e^{-t \pt (\partial_{\gd}\Phi(\gb,\gd)+\eps)\sqrt L}  \\
&\leq\; \frac{Z_{L,\beta,\delta+t}^{\,\circ,+}}{Z_{L,\beta,\delta}^{\,\circ,+}} e^{-t \pt (\partial_{\gd}\Phi(\gb,\gd)+\eps)\sqrt L} \:.
\end{align*}
Applying Theorem~\ref{th:partfunbead} $(i)$, we therefore get that there exists some $C_\cntc>0$ such that for $L\geq1$,
\begin{equation*}
\bP_{L,\beta,\delta}^{\,\circ,+} \Big( Q(\ell) \geq \big(\partial_{\gd}\Phi(\gb,\gd)+\eps\big) \sqrt{L} \Big) \:\leq\: C_{\arabic{cst}} e^{ (\Phi(\gb,\gd+t)-\Phi(\gb,\gd)  -t \pt\partial_{\gd}\Phi(\gb,\gd) -t \pt\eps)\sqrt L}\,.
\end{equation*}
Finally, $\Phi(\gb,\gd+t)-\Phi(\gb,\gd)  -t \pt\partial_{\gd}\Phi(\gb,\gd)=o(t)$ as $t\to0$ (since $\delta\mapsto \Phi(\gb,\gd)$ is $\cC^1$), so we can fix $t>0$ such that the r.h.s. above goes to 0 as $L\to\infty$. Similarly, we can write for $s>0$ such that $\gd-s\in(\tilde \delta_c(\beta), \delta^{\circ}_c(\beta))$,
\begin{align*}
\bP_{L,\beta,\delta}^{\,\circ,+} \Big( Q(\ell) \leq \big(\partial_{\gd}\Phi(\gb,\gd)-\eps\big) \sqrt{L} \Big) \;&\leq\; \frac{Z_{L,\beta,\delta-s}^{\,\circ,+}}{Z_{L,\beta,\delta}^{\,\circ,+}} e^{s \pt (\partial_{\gd}\Phi(\gb,\gd)-\eps)\sqrt L}  \\
&\leq\; C_\cntc e^{(\Phi(\gb,\gd-s)-\Phi(\gb,\gd)  +s \pt\partial_{\gd}\Phi(\gb,\gd) -s \pt\eps)\sqrt L} 
\;,
\end{align*}
for some $C_{\arabic{cst}}>0$, where we also applied Chernov's bound and Theorem~\ref{th:partfunbead} $(i)$. We can then choose $s>0$ such that this term also goes to $0$ as $L\to\infty$. This concludes the proof of Corollary~\ref{corol:partfunbead}-(i).

Regarding the case $\gd\in[0,\tilde \delta_c(\beta))$, let $K>0$, $\eps>0$ and $t>0$ be such that $\gd+t<\tilde \delta_c(\beta)$. Using once again Chernov's bound and Theorem~\ref{th:partfunbead} $(ii)$, there exist $C_\cntc>0$ and $L_0\in\N$ such that for $L\geq L_0$,
\begin{align*}
\bP_{L,\beta,\delta}^{\,\circ,+} \Big( Q(\ell) \geq K L^{1/6} \Big)  \;&\leq\; \bE_{L,\beta,\delta}^{\,\circ,+}\big[e^{t \pt Q(\ell)}\big] e^{-t\pt K L^{1/6}}  \\
&\leq\; \frac{Z_{L,\beta,\delta+t}^{\,\circ,+}}{Z_{L,\beta,\delta}^{\,\circ,+}} e^{-t\pt K L^{1/6}} \leq\; C_{\arabic{cst}} L^{-3/4} e^{-(\Psi(\gb) +\eps -t\pt K) L^{1/6} }
\:.
\end{align*}
This goes to $0$ as $L\to\infty$ provided that $K$ has been fixed sufficiently large, and concludes the proof of Corollary~\ref{corol:partfunbead}-(ii).

\section{Proof of Theorem \ref{Phase-diag} and Proposition~\ref{Phase-diag:bead}}\label{sec:prth11}
\subsection{Proof of Theorem \ref{Phase-diag}}
The proof of Theorem~\ref{Phase-diag} also relies on the random walk representation introduced in \cite{NGP13} and adapted in the proof of Theorem~\ref{th:partfunbead} (see Section~\ref{sec:prth22}), but it is much less involved. We devide the proof into 4 steps. First, we prove that the free energy is not changed when we additionally constrain the polymer to end on the horizontal axis (that is $\sum_{i=1}^N\ell_i=0$) ---in particular the \emph{constrained} partition function $Z_{L,\gb,\gd}^{+,c}$ is super-multiplicative, which implies the well-posedness of $f$. Next, we adapt to the present model the random-walk representation of IPDSAW. As in Section~\ref{sec:prth22}, we derive a probabilistic representation of the partition function by rewriting, for every $N\leq L$, the contribution to the partition function of those trajectories made of $N$ stretches, in terms of two auxiliary (coupled) random walks $S$ and $I$. Then we compute the generating function of the partition function $Z_{L,\gb,\gd}^{+,c}$. With our random walk representation, we can rewrite it as the partition function of a wetting model for two independent random walks, with the in-between area constraint of $S$ and $I$ becoming an in-between area penalization in the generating function. This allows us to characterize $\tilde f$ in terms of the free energy of this ``coupled, in-between area penalized'' wetting model. Finally we place ourselves on the critical curve between $\cC$ and $\cE$ where the area penalization vanishes, and we apply well-known results on the wetting model to derive the equation of the curve.

\subsubsection{Step 1: constraining the partition function.}
Let us define the \emph{constrained partition function}:
\begin{equation}\label{def:Zc}
Z_{L,\gb,\gd}^{+,c} \;:=\; \sum_{N=1}^L \sum_{\substack{{\ell\in \cL^+_{N,L}}\\{\ell_0=\ell_{N+1}=0}}} e^{H(\ell)} \ind_{\left\{\sum_{i=1}^N \ell_i = 0\right\}}\;.
\end{equation}
Notice that this partition function is super-multiplicative: indeed for any $L_1,L_2\geq0$, we bound $Z_{L_1+L_2,\gb,\gd}^{+,c}$ from below by constraining it to touch the axis between the $(L_1-1)$-th and $L_1$-th monomers; then we notice that the contribution to $H(\ell)$ from self-touching between the two parts of the polymer (before and after the segment $(L_1-1,L_1)$) are non-negative; and we separate the trajectory in two at the $L_1$-th monomer (recall that all our trajectories end with a horizontal step) to finally write
\begin{equation}
Z_{L_1+L_2,\gb,\gd}^{+,c} \,\geq\, \pt Z_{L_1,\gb,\gd}^{+,c} \pt Z_{L_2,\gb,\gd}^{+,c}\,.
\end{equation}
Hence $\frac1L \log Z_{L,\gb,\gd}^{+,c}$ converges as $L\to\infty$, by Fekete's lemma.

We now claim that $Z_{L,\gb,\gd}^{+}$ and $Z_{L,\gb,\gd}^{+,c}$ are comparable. More precisely for $L\in\N$,
\begin{equation}\label{eq:Zc}
\frac1{L^2}\big(Z_{L,\gb,\gd}^{+}\big)^2\;\leq\;Z_{2L,\gb,\gd}^{+,c}\;\leq\;Z_{2L,\gb,\gd}^{+}\;.
\end{equation}
The upperbound is straightforward. For the lower bound, we constrain $Z_{2L,\gb,\gd}^{+,c}$ to make a horizontal step between the $(L-1)$-th and $L$-th monomers, and sum over its possible heights. Writing $z_{L-1}=(x_{L-1},y_{L-1})$ and $z_L=(x_L,y_L)$ the coordinates of those monomers, we have
\begin{equation*}
Z_{2L+2,\gb,\gd}^{+,c} \;\geq\; \sum_{y=0}^{L-1} Z_{2L,\gb,\gd}^+(y_{L-1}=y_{L} =y, x_L=1+x_{L-1})\;.
\end{equation*}
We then separate the trajectory at the $L_1$-th monomer. The first half gives the term $Z_{L,\gb,\gd}^+(y_{L} =y)$ (recall that the last step is always horizontal). As for the second half, we shift the last horizontal step to just before $z_L$ (which costs at most a factor $e^{\gd}$), then we reverse the trajectory horizontally, to obtain that it is bounded from below by $Z_{L,\gb,\gd}^+(y_{L} =y)$. Therefore,
\begin{equation*}\begin{aligned}
Z_{2L,\gb,\gd}^{+,c}\;&\geq\; \sum_{y=0}^{L-1} \Big(Z_{L,\gb,\gd}^+(y_{L}=y)\Big)^2\;\geq\; \Big(\sup_{0\leq y\leq L-1} Z_{L,\gb,\gd}^+(y_{L}=y)\Big)^2\;,
\end{aligned}\end{equation*}
and we conclude the proof of \eqref{eq:Zc} by writing $Z_{L,\gb,\gd}^+ \leq L\pt \sup_{0\leq y\leq L-1} Z_{L,\gb,\gd}^+(y_L=y)$. In particular, this implies that
\begin{equation}\label{eq:fZZc}
f(\gb,\gd)=\lim_{L\to\infty} \frac1L \log Z_{L,\gb,\gd}^{+}=\lim_{L\to\infty} \frac1L \log Z_{L,\gb,\gd}^{+,c}\;,
\end{equation} is well defined, and that we can replace $Z_{L,\gb,\gd}^{+}$ with $Z_{L,\gb,\gd}^{+,c}$ to prove Theorem~\ref{Phase-diag}.

\subsubsection{Step 2: a random-walk representation.}
We provide a probabilistic description of the constrained partition function $Z_{L,\gb,\gd}^{+,c}$ as in Section~\ref{sec:prth22} ---which also applies to the free counterpart $Z_{L,\gb,\gd}^{+}$.
Recalling \eqref{def:Zc}, \eqref{def:Ham} and \eqref{def:twg}, we can rewrite the constrained partition function similarly to \eqref{def:zxy+:bis} to obtain
\begin{equation}\label{def:z+:bis}
Z_{L,\gb,\gd}^{+,c}= e^{\gb L}\sum_{N=1}^{L} e^{-\gb N}\sumtwo{\ell\in\cL_{N,L}^{+}}{\ell_0=\ell_{N+1}=0} e^{-\frac{\gb}{2} \sum_{i=0}^{N} |\ell_i + \ell_{i+1}|}\,e^{\gd\sum_{k=1}^{N}\ind_{\{\sum_{i=1}^{k}\ell_i=0\}}} \ind_{\big\{\sum_{i=1}^N\ell_i=0\big\}}.
\end{equation}
Recall the definition of $\bP_\gb$ in \eqref{def:Pgb}: similarly to \eqref{def:zxy+:ter}, we consider two independent one-dimensional random walks $S:=(S_i)_{i\geq 0}$ and $I:=(I_i)_{i\geq 0}$ starting from 0 and such that $(S_{i+1}-S_i)_{i\geq 0}$ and $(I_{i+1}-I_i)_{i\geq 0}$ are i.i.d. sequences of random variables of law $\Pb_\gb$. We notice that for every $\ell\in\cL_{N,L}^{+}$ (with $\ell_0=\ell_{N+1}=0$) the first factor in the second sum in \eqref{def:z+:bis} satisfies
\begin{equation}\label{def:z+:ter} 
e^{-\frac{\gb}{2} \sum_{i=0}^{N} |\ell_i + \ell_{i+1}|}=c_\beta^{N+1} \pt \bP_\gb\left(S_k=\sum_{i=0}^{2k-1}\ell_i, \forall k\leq N_S\right) \pt \bP_\gb\left( I_k=\sum_{i=0}^{2k}\ell_i, \forall k\leq N_I\right)\;,
\end{equation}
with $N_S:=\lfloor N/2\rfloor +1$ and $N_I:=\lceil N/2\rceil$. Hence we define $S$ and $I$ as in Section~\ref{rwrep}. 
Similarly to \eqref{def:G_N}, define
\begin{equation}\label{def:G_N:bis}
G_N(S,I)\;:=\; \sum_{k=1}^{N_I} | I_k - S_k | + \sum_{k=1}^{N_S} |S_k-I_{k-1}| \;,
\end{equation}
and our definition of $S$, $I$ implies $\sum_{i=1}^N|\ell_i|=G_N(S,I)$. Finally the constraint $\sum_{i=1}^N\ell_i=0$ is equivalent to $S_{N_S}=I_{N_I}=0$ (recall $\ell_{N+1}=0$).

Defining $\tilde Z_{L,\gb,\gd}^{+,c} := \frac1{c_\gb} e^{-\gb L} Z_{L,\gb,\gd}^{+,c}$, and plugging \eqref{def:z+:ter} into \eqref{def:z+:bis}, we obtain
\begin{equation}\label{def:z+:quad} 
\tilde Z_{L,\gb,\gd}^{+,c}= \sum_{N=1}^{L} (\gGa_\gb)^{N} \Ebb\bigg[e^{\gd \sum_{k=1}^{N_S}\ind_{\{S_k=0\}}}e^{\gd \sum_{k=1}^{N_I}\ind_{\{I_k=0\}}}\ind_{\{G_N(S,I)=L-N\}}\ind_{\{S\in B_{N_S}^{0,+}\}}\ind_{\{I\in B_{N_I}^{0,+}\}}\bigg]\,,
\end{equation}
where we recall $\gGa_\gb=c_\gb e^{-\gb}$, $G_N(S,I)$ is defined in \eqref{def:G_N:bis}, and $B_n^{0,+}$ is the set of non-negative trajectories of length $n$ ending on the horizontal axis (see \eqref{def:By+}).

\subsubsection{Step 3: the generating function of $\tilde Z_{L, \gb,\gd}^{+,c}$}
Recalling \eqref{eq:fZZc}, the excess free energy satisfies
\begin{equation}\label{def:ftilde}
\tilde f (\gb,\gd)\;=\; \sup\big\{\gga\geq0\,;\,\sum_{L\geq1} \tilde Z^{+,c}_{L,\gb,\gd}e^{-\gga L} = +\infty\big\}\;,
\end{equation}
and $\tilde f (\gb,\gd)=0$ if this set is empty. Let us compute the generating function of the tilted partition function. Recalling \eqref{def:z+:quad} and inverting the sums over $L$ and $N$, we obtain
\begin{equation}\label{eq:generatingZ}
\sum_{L\geq1} \tilde Z_{L,\gb,\gd}^{+,c} e^{-\gga L}
= \sum_{N\geq1} \big(\gGa_\gb e^{-\gga}\big)^N Q_{N}^{\gb,\gd,\gga},
\end{equation}
where we define
\begin{equation}\label{def:QNgbgdgga}
Q_{N}^{\gb,\gd,\gga}\,:=\,\Ebbzero\bigg[e^{-\gga G_N(S,I)} \,e^{\gd \sum_{k=1}^{N_S}\ind_{\{S_k=0\}}}\,e^{\gd \sum_{k=1}^{N_I}\ind_{\{I_k=0\}}} \ind_{\{S\in B_{N_S}^{0,+},\pt I\in B_{N_I}^{0,+}\}}\bigg]\,,
\end{equation}
which is the partition function of a coupled wetting model, with an additional ``in-between area'' penalization. Note that $(Q_{N}^{\gb,\gd,\gga})_{N\geq1}$ is super-multiplicative, therefore we can define the free energy of this model:
\begin{equation}\label{def:htilde}
 h_\gb (\gd,\gga) \;:=\; \lim_{N\to\infty} \frac{1}{N} \log Q_{N}^{\gb,\gd,\gga},
\end{equation}
Notice that $h_\gb(\gd,\gga) \leq \gd$ (in particular $h_\gb$ is finite), $h_\gb$ is non-increasing in $\gga$ and non-decreasing in $\gd$, and it is continuous.

Recollecting \eqref{def:ftilde}, we conclude that $\tilde f(\gb,\gd)$ is the only positive solution (if it exists) in $\gga$ of:
\begin{equation}\label{eq:fhcharact}
\log \gGa_\gb \, - \, \gga \, + \, h_\gb (\gd,\gga) \;=\; 0\,,
\end{equation}
and $\tilde f(\gb,\gd)=0$ otherwise (this equation has at most one solution because $\gga\mapsto -\gga+h_\gb (\gd,\gga)$ is decreasing).\smallskip

\subsubsection{Step 4: characterizing the critical curve}
Let $(\gb,\gd)\in \overline \cC \cap \overline \cE$ be a point of the critical curve: then $\tilde f(\gb,\gd)=0$ (because $\tilde f$ is continuous) and $\log \gGa_\gb + h_\gb(\gd,0)=0$ (because \eqref{eq:fhcharact} is continuous in $\gga\searrow0$). In particular there is no area constraint in $Q_N^{\gb,\gd,0}$, and because $S$ and $I$ are independent we can uncouple them and write
\begin{equation}\label{eq:Zwetting}
Q_N^{\gb,\gd,0} \;=\; Z_{\wet, N_S}^{\gb, \gd}\times Z_{\wet, N_I}^{\gb, \gd}\;,
\end{equation}
where $Z_{\wet, N}^{\gb, \gd}$, $N\geq1$ is the partition function of a wetting model (see \eqref{def:Zwetting}). By applying $\frac1{N}\log$ to~\eqref{eq:Zwetting}, we notice that $h_\gb(\gd,0)$ matches exactly the free energy of the wetting model: $h_\gb(\gd,0)=h_\gb(\gd)$ (see \eqref{wetmod}).

The asymptotical behavior of $Z_{\wet, N}^{\gb, \gd}$ is already well-known ---see Proposition~\ref{prop:wetpoint}. Thus we can finally conclude the proof of Theorem~\ref{Phase-diag}. Recall that $(\gb,\gd)\in \overline \cC \cap \overline \cE$ implies $\log \gGa_\gb + h_\gb (\gd)=0$, where $\gGa_\gb$ is decreasing in $\gb$, and $\gb_c$ is the only solution of $\gGa_{\gb}=1$. Therefore,

-- if $\gb<\gb_c$, then $\log \gGa_\gb + h_\gb (\gd)\geq \log \gGa_\gb >0$ so $(\gb,\gd)\notin \overline \cC \cap \overline \cE$,

-- if $\gb=\gb_c$, then $(\gb,\gd)\in \overline \cC \cap \overline \cE$ if and only if $h_\gb (\gd)=0$, that is $\gd\leq\tilde \gd_{c}(\gb_c)$,

-- if $\gb>\gb_c$, then there is a unique solution in $\gd\geq\tilde \gd_{c}(\gb)$ to $h_\gb (\gd)=-\log \gGa_\gb$, which we note $\gd_c(\gb)$.

Moreover, recall that the l.h.s. of \eqref{eq:fhcharact} is decreasing and continuous in $\gga$, so $\tilde f(\gb,\gd)=0$ if and only if $\log \gGa_\gb + h_\gb (\gd)\leq0$, that is $\gb\geq\gb_c$ and $\gd\leq \gd_c(\gb)$; 
which fully characterizes $\cC$. Finally,  the analytic expression of $\gd_c(\gb)$ follows directly from \eqref{eq:cc} and \eqref{eq:explicit:h} by solving a quadratic equation; details are left to the reader.

\subsection{Proof of Proposition \ref{Phase-diag:bead}}
This is very similar to Theorem~\ref{Phase-diag}. The single-bead partition $Z^{\,\circ,+}_{L,\gb,\gd}$ function is already constrained to return to 0, hence it is super-multiplicative (and $f^{\pt\circ}$ is well-posed) and we do not need to replicate Step 1. The random walk representation is already laid out in Proposition~\ref{PEZ} ---notice that the main difference with \eqref{def:z+:quad} is the constraint $\{S\succ I\}$ (recall \eqref{def:succN}), in particular the random walk $S$ cannot touch the wall. Thereby the generating function of $\tilde Z^{\,\circ,+}_{L,\gb,\gd}$ can be written (similarly to \eqref{eq:generatingZ}) as
\begin{equation}
\sum_{L\geq1} \tilde Z_{L,\gb,\gd}^{\,\circ, +} e^{-\gga L}
= \sum_{N\geq1} \big(\gGa_\gb e^{-\gga}\big)^{2N} R_{N}^{\gb,\gd,\gga},
\end{equation}
where we define
\begin{equation}\label{def:RNgbgdgga}
R_{N}^{\gb,\gd,\gga}\,:=\,\Ebbzero\bigg[e^{-2\gga (A_{N+1}(S)-A_N(I))} \,e^{\gd \sum_{k=1}^{N}\ind_{\{I_k=0\}}} \ind_{\{S\succ I\}} \ind_{\{S\in B_{N+1}^{0,+},\pt I\in B_{N}^{0,+}\}}\bigg]\,.
\end{equation}
$(R_{N}^{\gb,\gd,\gga})_{N\geq1}$ is super-multiplicative, thereby we define its free energy
\[\gk_\gb(\gd,\gga)\;:=\;\lim_{N\to\infty} \frac1N \log R_N^{\gb,\gd,\gga}\;,\]
and we note that $\tilde f^{\pt\circ}(\gb,\gd)$ is the only positive solution (if it exists) in $\gga$ of
\[2 \log\gGa_\gb + 2\gga + \gk_\gb(\gd,\gga) \;=\;0\;,\]
and $\tilde f^{\pt\circ}(\gb,\gd)=0$ otherwise. 

Let $(\gb,\gd)\in\partial \pt \cC_{\mathrm{bead}}$ be on the critical curve (which implies $\tilde f^{\pt\circ}(\gb,\gd)=0$ and $2 \log\gGa_\gb +\gk_\gb(\gd,0) =0$). We notice $R_N^{\gb,\gd,\gga}\leq Z_{\wet,N}^{\gb,\gd}$, so $\gk_\gb(\gd,\gga)\leq h_\gb(\gd)$, and we will now compute a lower bound on $R_N^{\gb,\gd}:=R_N^{\gb,\gd,0}$ to deduce $\gk_\gb(\gd,0)= h_\gb(\gd)$.

Let $\ga>0$ and recall \eqref{def:RNgbgdgga} with $\gga=0$. Constraining $I_k$ to remain below $\lfloor\ga \sqrt{N}\rfloor$ for all $1\leq k <N$, and constraining $S$ to $S_1=S_N=\lfloor\ga\sqrt{N}\rfloor+1$, and $S_k\geq \lfloor\ga\sqrt{N}\rfloor+1$ for all $2\leq k <N$, we obtain the lower bound
\begin{align}
R_N^{\gb,\gd} &\;\geq\; \Ebbzero\left[e^{\gd \sum_{k=1}^{N}\ind_{\{I_k=0\}}} \ind_{\{I_k\leq \lfloor\ga\sqrt{N}\rfloor,\,\forall 1\leq k<N\}} \ind_{\{I\in B_N^{0,+}\}}\right]\notag\\
&\qquad\times \Pbbzero\left(S_1=S_N=\lfloor\ga\sqrt{N}\rfloor+1,\, S_{N+1}=0,\, S_k\geq \lfloor\ga\sqrt{N}\rfloor+1,\, \forall\, 2\leq k<N\right)\notag\\
&\;\geq\;\Ebbzero\left[e^{\gd \sum_{k=1}^{N}\ind_{\{I_k=0\}}} \ind_{\{I_k\leq \lfloor\ga\sqrt{N}\rfloor,\,\forall 1\leq k<N\}} \ind_{\{I\in B_N^{0,+}\}}\right]\left(\frac{e^{-\frac\gb2 \lfloor\ga\sqrt{N}\rfloor+1}}{c_\gb}\right)^2\Pbbzero\big(B_{N-1}^{0,+}\big),\notag
\end{align}
where we can estimate $\Pbbzero(B_{N-1}^{0,+})\asymp N^{-3/2}$ with \eqref{eq:asymp:PosBB} below. Moreover we notice that
\[(x_k)_{k=1}^N\mapsto e^{\gd \sum_{k=1}^N \ind_{\{x_k=0\}}}\,,\qquad \text{and}\quad (x_k)_{k=1}^N\mapsto \ind_{\{x_k\leq \lfloor\ga\sqrt{N}\rfloor,\,\forall 1\leq k<N\}}\;,\]
are both bounded, non-increasing functions on $B_N^{0,+}$, hence we can apply an {\rm FKG} inequality (see Proposition~\ref{prop:FKG}) to claim
\begin{equation}\label{eq:RNZwet}\begin{aligned}
&\Ebbzero\left[e^{\gd \sum_{k=1}^{N}\ind_{\{I_k=0\}}} \ind_{\{I_k\leq \lfloor\ga\sqrt{N}\rfloor,\,\forall 1\leq k<N\}} \ind_{\{I\in B_N^{0,+}\}}\right] \\
&\quad\geq\; Z_{\wet,N}^{\gb,\gd} \times \Pbbzero\big(I_k\leq \lfloor\ga\sqrt{N}\rfloor,\,\forall 1\leq k<N\,\big|\, I\in B_N^{0,+}\big)\;,
\end{aligned}\end{equation}
where we recall that $Z_{\wet, N}^{\gb, \gd}$, $N\geq1$ is the partition function of a wetting model (see \eqref{def:Zwetting}), and the second factor in the r.h.s. is bounded from below by some positive constant provided that $\ga$ is large (see \eqref{eq:invarexcursion} below).

All this implies $\gk_\gb(\gd,0)\geq h_\gb(\gd)$, so $\gk_\gb(\gd,0)= h_\gb(\gd)$. Therefore we conclude the proof of Proposition~\ref{Phase-diag:bead} similarly to Theorem~\ref{Phase-diag}.

\section{Proofs of technical results}\label{sec:tech}
In this section we prove Lemmas~\ref{lem:FKG} and \ref{lem:Dxm}, then the much more involved Proposition~\ref{prop:DNLSI}.

 \subsection{Proof of Lemma~\ref{lem:FKG}}\label{proof:FKG} Before proving Lemma~\ref{lem:FKG}, we provide a useful estimate on non-negative random walk bridges proven in \cite{CC13}. Recall the definition of $B^{y,+}_n$ in \eqref{def:By+}, and that $\Pbbx$ is the law of a random walk with increments distributed as $\Pbb$ (see \eqref{def:Pgb}) starting from $x\geq0$. One has
\begin{equation}\label{eq:asymp:PosBB}
\Pbbzero(B^{x,+}_{n}) \;=\; \Pbbx(B^{0,+}_{n}) \;\asymp\; \frac{\max(x,1)}{n^{3/2}}\;,
\end{equation}
uniformly in $n\geq1$ and $0\leq x\leq C_\cntc\sqrt{n}$ for any $C_{\arabic{cst}}>0$. The first identity is obtained by reversing the walk in time (notice that $\Pbb$ is symmetric), and the asymptotic behavior is derived in \cite[(4.3)]{CC13}. 

It is also proven in \cite[Corol.~2.5]{CC13} that a properly rescaled centered random walk with finite variance conditioned to remain non-negative and to finish in 0 converges (in distribution) to a Brownian excursion. The maximum being a continuous function on $\cC([0,1],\R)$, we thereby deduce that for any $\gh>0$, there exist $C_A>0$ and $N_0\in\N$ such that
\begin{equation}\label{eq:invarexcursion}
\Pbbzero\Big(\max_{1\leq k \leq N}(X_k)>C_A \sqrt N \,\Big|\, X\in B^{0,+}_N\Big)\;\leq\; \gh\;,
\end{equation}
for all $N\geq N_0$.

\begin{proof}[Proof of Lemma~\ref{lem:FKG}]
Notice that $\max_{1\leq k \leq N}(X_k)\leq C_A\sqrt N$ implies $A_N(X)\leq C_A N^{3/2}$. Thus we have the lower bound,
\begin{equation}\begin{aligned}\label{eq:lemprf:FKG:-1}
&\Pbbzero\big(\forall\, 1\leq k < N, X_k\leq \tilde f_I(k)\,; A_N(X)\leq C_A N^{3/2} \,\big|\, X\in B^{0,+}_N\big) \\
& \; \geq\; 1 - \Pbbzero\Big(\exists\, 1\leq k < N, X_k> \tilde f_I(k)\,; \max_{1\leq k \leq N}(X_k)\leq C_A \sqrt N \,\Big|\, X\in B^{0,+}_N\Big) \\
&\qquad - \Pbbzero\Big(\max_{1\leq k \leq N}(X_k)>C_A\sqrt N \,\Big|\, X\in B^{0,+}_N\Big)\;.
\end{aligned}
\end{equation}
Recalling \eqref{eq:invarexcursion}, the last term in \eqref{eq:lemprf:FKG:-1} is not larger than some $\gh>0$ arbitrarily small. Regarding the other term, we write
\begin{equation}\begin{aligned}\label{eq:lemprf:FKG:0}
&\Pbbzero\Big(\exists\, 1\leq k < N\,, X_k> \tilde f_I(k)\,; \max_{1\leq k \leq N}(X_k)\leq C_A\sqrt N \,\Big|\, X\in B^{0,+}_N\Big) \\
& \quad \leq\; 2\sum_{k=1}^{\lceil N/2 \rceil} \Pbbzero\Big(\tilde f_I(k)<X_k\leq C_A\sqrt N \,\Big|\, X\in B^{0,+}_N\Big) \;,
\end{aligned}
\end{equation}
where we reversed in time all terms with index larger than $\lceil N/2 \rceil$. For any $k\leq\lceil N/2\rceil$, we partition the $k$-th term over possible values of $X_k$, then we apply Markov property at time $k$ to write:
\begin{align}
\Pbbzero\Big(\tilde f_I(k)<X_k\leq C_A\sqrt N \,\Big|\, X\in B^{0,+}_N\Big)
\label{eq:lemprf:FKG:1}
\;&=\; \sum_{z=\tilde f_I(k)+1}^{C_A \sqrt N} \Pbbzero\big(B^{z,+}_k\big) \frac{\Pbbz\big(B^{0,+}_{N-k}\big)}{ \Pbbzero\big(B^{0,+}_N\big)}\,.
\end{align}
Recalling \eqref{eq:asymp:PosBB}, we can bound $\Pbbz\big(B^{0,+}_{N-k}\big)\pt/\pt \Pbbzero\big(B^{0,+}_N\big)$ from above by $C_\cntc \times z$ with $C_\arabic{cst}>0$ a constant uniform in $z\leq C_A \sqrt{N}$ (recall that $N-k\geq \frac N2$). Hence, \eqref{eq:lemprf:FKG:1} becomes
\begin{align*}
\Pbbzero\Big(\tilde f_I(k)<X_k\leq C_A\sqrt N \,\Big|\, X\in B^{0,+}_N\Big) 
\;&\leq\;
C_{\arabic{cst}} \pt \Ebbzero\left[ X_k \ind_{\{\tilde f_I(k)< X_k\leq C_A\sqrt N\}} \ind_{\{X_i\geq0, \forall i\leq k\}}\right]
\,.
\end{align*}
Recall the definition of $\tilde f_I$ in \eqref{def:tildefSI}. 
Noticing that $\Ebbzero[e^{\gl X_k}] = e^{k\cL(\gl)}$ and $\Ebbzero[X_k e^{\gl X_k}] = \partial_\gl \Ebbzero[e^{\gl X_k}]= k\cL'(\gl) e^{k\cL(\gl)}$ for any $0<\gl<\gb/2$, and choosing $\gl$ sufficiently small so that $\cL(\gl)<\gl \gga$ (recall $\gga>0$ and $\cL(\gl)\sim\gl^2\gs_\gb^2/2$ as $\gl\searrow0$), we obtain with Markov's inequality,
\begin{align*}
\Pbbzero\Big(\tilde f_I(k)<X_k\leq C_A\sqrt N \,\Big|\, X\in B^{0,+}_N\Big) \;&\leq\; C_{\arabic{cst}} \pt \Ebbzero\left[ X_k e^{\gl(X_k-\gga k-K_{I})}\right]\\
&\leq\; C_\cntc \pt k \pt e^{-\gl K_I} e^{-k(\gl \gga-\cL(\gl))}\;.
\end{align*}
The r.h.s. being summable in $k$, there exists $C_\cntc>0$ such that,
\begin{equation}\label{eq:lemprf:FKG:4}
\sum_{k=1}^{\lceil N/2\rceil} \Pbbzero\Big(\tilde f_I(k)<X_k\leq C_A\sqrt N \,\Big|\, X\in B^{0,+}_N\Big) \;\leq\;
C_\arabic{cst} \pt e^{-\gl K_I} \;,
\end{equation}
uniformly in $N\geq N_0$. Assuming $K_I$ is sufficiently large, this term is smaller than any fixed $\gh>0$. 
Recollecting \eqref{eq:lemprf:FKG:-1}, \eqref{eq:invarexcursion}, \eqref{eq:lemprf:FKG:0} and \eqref{eq:lemprf:FKG:4} and fixing $\gh=\eps/3$, this concludes the proof.
\end{proof}

\subsection{Proof of Lemma~\ref{lem:Dxm}}\label{proof:Dxm}
Let us define for any $x,m\in\N_0$ some $\ga>0$,
\begin{equation}\label{def:Dxm}
D_{x,m} \;:=\; \Ebbx\left[e^{-\ga A_{m}(I)}\,\middle|\,I\in B_{m}^+\right].
\end{equation}
Let us prove that $x\mapsto D_{x,m}$ is non-increasing for any $\ga>0$, $m\in\N_0$ by induction on $m$. Notice that the proof also holds when conditioning by $B_{m}^{0,+}$ instead of $B_{m}^+$ without further changes (this is required in the proof of Claim~\ref{convGends}).

\begin{proof}[Proof of Lemma~\ref{lem:Dxm}]
When $m=0$, $D_{x,0}=e^{-\ga x}$ for all $x\geq0$, which is non-increasing.
Let $m\in\N$ and assume $x\mapsto D_{x,m-1}$ is non-increasing. For any $x\geq0$ and $\ga>0$, one has by Markov's property,
\begin{align}
\nonumber \Ebbx\Big[e^{-\ga A_{m}(I)}\,\Big|\,B_{m}^+\Big] \;&=\; \sum_{y\geq0} \Ebbx\Big[e^{-\ga x}e^{-\ga \sum_{k=1}^{m}I_k}\pt\ind_{\{I_1=y\}}\,\Big|\,B_{m}^+\Big]\\\label{eq:prfENqDC:UB:1}
&=\; e^{-\ga x} \sum_{y\geq0} \Ebby\Big[e^{-\ga A_{m-1}(I)}\,\Big|\,B_{m-1}^+\Big] R_x(y)\;,
\end{align}
where
\begin{equation}\label{eq:prfENqDC:UB:1.5}
R_x(y)\;:=\; \frac{\Pbbx(I_1=y)\Pbby(B_{m-1}^+)}{\Pbbx(B_{m}^+)}\,=\, \Pbbx(I_1=y\,|\, B_{m}^+)\;.\end{equation}
For any $x,y\geq0$, let $\overline R_x(y):= \sum_{t\geq y} R_x(t)$, so \eqref{eq:prfENqDC:UB:1} becomes
\begin{align}
\nonumber & e^{\ga x}\,\Ebbx\Big[e^{-\ga A_{m}(I)}\,\Big|\,B_{m}^+\Big] 
\;=\; \sum_{y\geq0} \Ebby\Big[e^{-\ga A_{m-1}(I)}\,\Big|\,B_{m-1}^+\Big] \Big(\overline R_x(y)-\overline R_x(y+1)\Big)\\
\nonumber &\quad=\; \overline R_x(0) \Ebbzero\Big[e^{-\ga A_{m-1}(I)}\,\Big|\,B_{m-1}^+\Big] \\\label{eq:prfENqDC:UB:2}
&\qquad + \sum_{y\geq 1} \overline R_x(y)\bigg(\Ebby\Big[e^{-\ga A_{m-1}(I)}\,\Big|\,B_{m-1}^+\Big]-\bE_{\gb,y-1}\Big[e^{-\ga A_{m-1}(I)}\,\Big|\,B_{m-1}^+\Big]\bigg)\;.
\end{align}
Recall that we assumed that $y\mapsto D_{y,m-1}$ is non-increasing, so \[\Ebby\Big[e^{-\ga A_{m-1}(I)}\,\Big|\,B_{m-1}^+\Big] - \bE_{\gb,y-1}\Big[e^{-\ga A_{m-1}(I)}\,\Big|\,B_{m-1}^+\Big] \;\leq 0\;,\] for all $y\geq1$.

Moreover we claim that $\overline R_{x+1}(y)\geq \overline R_x(y)$ for all $x,y,m\geq0$. To prove this, it suffices to notice that for all $x,m\in\N_0$, and for any trajectory $(s_k)_{k=1}^m\in\bbZ^m$, 
\[
\bP_{\gb,x+1}\big((I_k)_{k=1}^m=(s_k)_{k=1}^m\big) \;=\; e^{\pm \tfrac\gb2}\Pbbx\big((I_k)_{k=1}^m=(s_k)_{k=1}^m\big),
\]
where the sign $\pm$ is $+$ if $s_1\leq x$ and $-$ otherwise; then to distinguish the cases $y\geq x+1$ (where the proof is instantaneous), and $y\leq x+1$ (where it is obtained by induction on $0\leq y\leq x+1$ for a fixed $x\in\N_0$).

Noticing also that $\overline R_x(0) = \overline R_{x+1}(0)=1$ for all $x\geq0$, we obtain the lower bound
\begin{align*}
& e^{\ga x}\pt\Ebbx\Big[e^{-\ga A_{m}(I)}\,\Big|\,B_{m}^+\Big]
\;\geq\; \overline R_{x+1}(0) \Ebbzero\Big[e^{-\ga A_{m-1}(I)}\,\Big|\,B_{m-1}^+\Big] \\
&\qquad\qquad + \sum_{y\geq 1} \overline R_{x+1}(y)\bigg(\Ebby\Big[e^{-\ga A_{m-1}(I)}\,\Big|\,B_{m-1}^+\Big]-\bE_{\gb,y-1}\Big[e^{-\ga A_{m-1}(I)}\,\Big|\,B_{m-1}^+\Big]\bigg)\;,
\end{align*}
and the identity \eqref{eq:prfENqDC:UB:2} finally yields
\[ e^{\ga x}\pt\Ebbx\Big[e^{-\ga A_{m}(I)}\,\Big|\,B_{m}^+\Big] \;\geq\;  e^{\ga (x+1)}\pt\bE_{\gb,x+1}\Big[e^{-\ga A_{m}(I)}\,\Big|\,B_{m}^+\Big]\;,\]
which concludes the induction.
\end{proof}

\subsection{Proof of Proposition~\ref{prop:DNLSI}}\label{proof:DNLSI}
This proposition is an improvement of \cite[Prop. 2.5]{CNP16}, where a polynomial lower bound is displayed for the probability that an $n$-step walk remains positive,  comes back to $0$ at time $n$ and enclose an area $qn^2$. To improve this result, we recall some tools introduced in \cite{CNP16} and we keep in mind that all the upcoming claims are proven in \cite{CNP16}. Recall that $\cL(h)$ is the logarithmic moment generating function of the increments of the random walk $X$, see \eqref{def:cL}. To lighten upcoming formulae, let us define for any $n\in\N$,
\begin{equation}\label{def:gLn}
\gL_n:=\big(A_n(X)/n,X_n\big),
\end{equation}
(notice that the area is normalized by $n$ in this definition), as well as the parallelograms
\begin{equation}\label{def:cDgbn}
\cD_{\gb,n} \;:=\; \Big\{(h_0,h_1)\in\R^2\,;\;|h_1|<\gb/2,\; |(1-1/N)h_0+h_1|<\gb/2\Big\}.
\end{equation}
For any $\bh\in\cD_{\gb,n}$, we define the tilted law:
\begin{equation}\label{def:Pnh}
\frac{\dd \Pb_{n,\bh}}{\dd \Pbbzero} (X) = e^{\bh\cdot\gL_n-\cL_{\gL_n}(\bh)}\;,\quad \cL_{\gL_n}(\bh):= \log\Ebbzero[e^{\bh\cdot\gL_n}],
\end{equation}
where $\bh\cdot\gL_n:=h_0A_n/n+h_1X_n$ denotes the scalar product. Note that
\begin{equation}\label{eq:tilth:dvlpmt}
\bh\cdot\gL_n=\sum_{k=1}^n \big((1-k/n)h_0 + h_1\big)(X_k-X_{k-1}),
\end{equation}
so this change of measure is equivalent to tilt all increments of $X$ independently, with an intensity depending on $k$. For any $q, p\in\R$ and $n\in\N$, let $\bh_n^{q,p}=(h_{n,0}^{q,p},h_{n,1}^{q,p})$ be the unique solution in $\bh\in\cD_{\gb,n}$ of the equation:
\begin{equation}\label{def:hnq}
\bE_{n,\bh}\bigg[\frac1n\gL_n\bigg] = \nabla \bigg[\frac1n\cL_{\gL_n}\bigg](\bh) = (q,p)\;,
\end{equation}
where it is proven in \cite[Lem.~5.4]{CNP16} that $\nabla [\frac1n\cL_{\gL_n}]$ is a $\cC^1$ diffeomorphism from $\cD_{\gb,n}$ to $\R^2$. Notice that
\begin{equation}\label{eq:Pnhnq}
\Pbbzero(\gL_n=(nq,np))= \Pb_{n,\bh_n^{q,p}}(\gL_n=(nq,np)) e^{-n\bh_n^{q,p}\cdot\,(q,p)+\cL_{\gL_n}(\bh_n^{q,p})}.
\end{equation}
Moreover \cite[Prop. 6.1]{CNP16} gives a uniform local central limit theorem for $\Pb_{n,\bh_n^{q,p}}$: for any $q_1<q_2$, $p_1<p_2$, $t_1<t_2$ and $s_1<s_2$, there are constants $C_\cntc,C'_{\arabic{cst}}>0$ and $n_0\in\N$ such that for any $q\in[q_1,q_2]$, $p\in[p_1,p_2]$, $t\in[t_1,t_2]$, $s\in[s_1,s_2]$ and $n\geq n_0$,
\begin{equation}\label{prop:Pnhnq}
\frac{C_{\arabic{cst}}}{n^2} \leq \Pb_{n,\bh_n^{q,p}}\big(\gL_n=(qn+t\sqrt{n},pn+s\sqrt{n})\big) \leq \frac{C'_{\arabic{cst}}}{n^2}\;.
\end{equation}
Notice that the authors of \cite{CNP16} only state it for $q\in[q_1,q_2]$ and $p=0$, but the local limit theorem and all the uniformity arguments also hold for $p\in[p_1,p_2]$. 

The asymptotic of $\bh_n^{q,p}$ for large $n$ can be described sharply. For any $\bh\in\cD_\gb$, Recall that we defined $\cL_\gL(\bh) := \int_0^1 \cL(h_0x+h_1)\dd x$ for all $\bh\in\cD_\gb$ (see \eqref{def:cLgL}), and let $\tilde\bh^{q,p}=(\tilde h_0^{q,p},\tilde h_1^{q,p})$ be the unique solution in $\bh\in\cD_\gb$ of the equation $\nabla\cL_\gL(\bh)=(q,p)$, where
\begin{equation}\begin{aligned}\label{def:tildebh}
\nabla\cL_\gL(\bh) \;&=\; (\partial_{h_0} \cL_\gL, \partial_{h_1} \cL_\gL)(\bh)\\
&=\; \Big(\int_0^1 x\cL'(xh_0+h_1)\dd x, \int_0^1 \cL'(x h_0+h_1)\dd x\Big)\;,
\end{aligned}\end{equation}
is a $\cC^1$ diffeomorphism from $\cD_\gb$ to $\R^2$ (see \cite[Lem.~5.3]{CNP16}). 
If $h_0\neq 0$, an integration by parts yields
\begin{equation}\label{def:tildebh:bis}
\nabla\cL_\gL(\bh) \;=\; \frac1{h_0} \big(\cL(h_0+h_1)-\cL_\gL(\bh),\,\cL(h_0+h_1)-\cL(h_1)\big)\;,
\end{equation}
and if $h_0=0$, $\nabla\cL_\gL(\bh)=(\cL'(h_1)/2,\,\cL'(h_1))$.

For any $q_1<q_2$ and $p_1<p_2\in\R$, \cite[Prop. 2.3]{CNP16} gives constants $C_\cntc, C'_{\arabic{cst}}>0$ and $n_0$, such that for any $q\in[q_1,q_2]$, $p\in[p_1,p_2]$ and $n\geq n_0$, one has
\begin{equation}\label{prop:cLgLn}
\left|\left[\frac{1}{n}\cL_{\gL_n}(\bh_n^{q,p})-\bh_{n}^{q,p}\cdot(q,p)\right]-\left[\cL_\gL(\tilde\bh^{q,p})-\tilde\bh^{q,p}\cdot(q,p)\right]\right|\leq \frac{C_{\arabic{cst}}}{n},
\end{equation}
and
\begin{equation}\label{prop:tildehhnq}
\|\bh_n^{(q,p)}-\tilde \bh^{q,p}\| \leq \frac{C'_{\arabic{cst}}}{n}.
\end{equation}
Notice that \cite[Prop. 2.3]{CNP16} is only formulated for $p=0$ and $q\in[q_1,q_2]$ with $q_1>0$, but all uniformity arguments also hold for $p\in[p_1,p_2]$ and any $[q_1,q_2]\subset\R$. These estimates, combined with \eqref{eq:Pnhnq} and \eqref{prop:Pnhnq}, give the precise asymptotic of $\Pbb(\gL_n=(nq,np))$ (up to constant factors) uniformly in $q\in[q_1,q_2]$, $p\in[p_1,p_2]$.

Before we start the proof, let us point out that $\tilde h_0^{q,p}>0$ for all $q,p\in\R$ with $p\leq2q-\eps$. Indeed, one notices that $\tilde h_0^{q,p}=0$ and \eqref{def:tildebh} imply $p=2q$. Because $\tilde \bh$ is continuous, it has constant sign on each sets $\{p<2q\}$ and $\{p>2q\}$, and it is already stated in \cite[Rem.~5.5]{CNP16} that $\tilde h_0^{q,0}>0$ for all $q>0$.



Finally we also define the i.i.d. uniform tilt of the increments of $X$ for any $|\gd|<\gb/2$:
\begin{equation}\label{def:Pgd}
\frac{\dd \tilde\Pb_\gd}{\dd \Pbbzero} (X_1) = e^{\gd X_1 - \cL(\gd)},
\end{equation}
where we defined $\cL$ in \eqref{def:cL}.

\begin{proof}[Proof of Proposition~\ref{prop:DNLSI}] Here is the strategy of the proof: recollecting \eqref{eq:Pnhnq}, \eqref{prop:Pnhnq} and \eqref{prop:cLgLn}, we already have
\begin{equation}
\Pbbzero(\gL_N=(qN,pN))\;\asymp\; \frac1{N^2} e^{-N\big(\tilde\bh^{q,p}\cdot(q,p)-\cL_\gL(\tilde\bh^{q,p})\big)}\;,
\end{equation}
(recall \eqref{defcomp} for the definition of $\asymp$). Hence the proof will be complete when we show that
\[\Pbbzero(X_k\geq \tilde f(k),\, \forall\, 0<k<N\,|\,\gL_N=(qN,pN))\]
is bounded from below by some positive constant uniform in $q\in[q_1,q_2]$, $p\in[p_1,p_2]$ and $N\in\N$.

For that purpose we decompose a trajectory $X$ with $X_0=0$, $X_N=pN$ and area $A_N(X)=qN^2$ into three parts: both ends of length $u_N:=N^{1/3}$, and the middle part of length $N':=N-2\uN$. To estimate the probability to go under the fixed curve $\tilde f(k)$, we study the middle part under the (non-uniform) tilt $\Pb_{N',\bh}$, and we handle both ends with uniform tilts $\tilde\Pb_{\gd_1}$, $\tilde\Pb_{\gd_2}$. Then we take advantage of the fact that uniformly tilted walks remain positive (and even above certain affine curves) with positive probability, and we handle the middle part with forementioned estimates.


For $i\in\{1,2\}$, define
\begin{equation}\label{def:QuN}\begin{aligned}
Q_{i,\uN}\;&:=\;\Big(\big[a_i\uN/2\,,\,b_i\uN/2 + 3K + K_S\big]\cap\Big(\frac1\uN\bbZ\Big)\Big) \\
&\qquad \times \Big(\big[a_i(\uN-1)+K+K_S\,,\,b_i(\uN-1)+3K+K_S\big]\cap\bbZ\Big)\;,
\end{aligned}\end{equation}
where $a_i<b_i$, $i\in\{1,2\}$ and $K$ are convenient positive constants which will be fixed below. We constrain the first bit of the walk (until index $\uN$) to grow by $X_\uN=x_1$ and have an area $A_\uN=A_1$, with $(A_1/\uN, x_1)\in Q_{1,\uN}$. Similarly, we reverse the third bit in time (from $N$ to $N-\uN$), and we constrain it to grow by $X_{N-\uN}-pN=x_2$ and have an area $(A_N-A_{N-\uN})-pN\uN =A_2$, with $(A_2/\uN, x_2)\in Q_{2,\uN}$.

Therefore we have the decomposition
\begin{equation}\label{eq:dcpDNLSI:1}\begin{aligned}
&\Pbbzero(X_k\geq \tilde f(k), \forall\, 0<k<N\,;\gL_N=(qN,pN))\\
&\geq\sumtwo{(A_1/\uN,x_1)\in Q_{1,\uN},}{(A_2/\uN,x_2)\in Q_{2,\uN}} \Pbbzero(X_k\geq \tilde f(k), \forall\, 0<k<\uN\,;\gL_\uN=(A_1/\uN,x_1))\\
&\qquad \times \Pbbzero(X_k\geq \tilde f(k+\uN)-x_1, \forall\, 0<k<N'\,;\gL_{N'}=(q'N',p'N'))\\
&\qquad \times \Pbbzero(X_k\geq \tilde f(N-k)-pN, \forall\, 0<k<\uN\,;\gL_\uN=(A_2/\uN,x_2)),
\end{aligned}\end{equation}
where $N':=N-2\uN$, and we define $(q',p')$ such that $q'N'^2= qN^2-A_1-A_2-x_1N'-pN\uN$ and $p'N'=pN+x_2-x_1$. Notice that with our choice of $A_1, A_2=O(\uN^2)$, $x_1, x_2=O(\uN)$, with $\uN=N^{1/3}$, and because $q\in[q_1,q_2]$, $p\in[p_1,p_2]$, then for $N$ sufficiently large $q'$ and $p'$ are contained in some compact sets $[q_1',q_2']$ and $[p_1',p_2']$. 
Moreover we have the following estimates:
\begin{equation}\label{eq:prfDNLSI:qq'pp'}\begin{aligned}
q - q' &= (p'-4q')\frac{\uN}N +\frac{x_1}N + o(1/N)\;,\\
p - p' &= -2p'\frac{\uN}N + \frac{x_1-x_2}N\;.
\end{aligned}\end{equation}

As mentioned above, we tilt uniformly the first and last factors in the r.h.s. of \eqref{eq:dcpDNLSI:1} with respective parameters $\gd_1$, $\gd_2$ (which will be explicitly fixed later),
\begin{equation}\label{eq:prfDNLSI:ends}\begin{aligned}
&\Pbbzero(X_k\geq \tilde f(k), \forall\, 0<k<\uN\,;\gL_\uN=(A_1/\uN,x_1)) \\
&\;= e^{-\gd_1 x_1 + \uN \cL(\gd_1)} \tilde\Pb_{\gd_1}(X_k\geq \tilde f(k), \forall\, 0<k<\uN\,;\gL_\uN=(A_1/\uN,x_1))\;\;;\\
&\Pbbzero(X_k\geq \tilde f(N-k)-pN, \forall\, 0<k<\uN\,;\gL_\uN=(A_2/\uN,x_2)) \\
&\;= e^{-\gd_2 x_2 + \uN \cL(\gd_2)} \tilde\Pb_{\gd_2}(X_k\geq \tilde f(N-k)-pN, \forall\, 0<k<\uN\,;\gL_\uN=(A_2/\uN,x_2)).
\end{aligned}\end{equation}
The second factor in the r.h.s. of \eqref{eq:dcpDNLSI:1} is bounded from below by the following lemma.

\begin{lemma}\label{lem:prfDNLSI}
Fix some $\eps>0$ and $q'_1<q'_2$, $p'_1<p'_2$ in $\R$. Then there are constants $c,C>0$ and $N_0\in\N$ such that for any $N'\geq N_0$, $q'\in[q'_1,q'_2]$ and $p'\in[p'_1,p'_2]$ satisfying $p'\leq 2q'-\eps$, one has
\begin{equation}\label{eq:lem:prfDNLSI}\begin{aligned}
&\Pbbzero(X_k\geq \tilde f(k+\uN)-x_1, \forall\, 0<k<N'\,;\gL_{N'}=(q'N',p'N')) \\
&\quad \geq e^{-N' \bh_{N'}^{q',p'}\cdot\,(q',p') + \cL_{\gL_{N'}}(\bh_{N'}^{q',p'})} \big(\Pb_{N',\bh_{N'}^{q',p'}}(\gL_{N'}=(q'N',p'N'))-C\, e^{-c \,\uN}\big)\;.
\end{aligned}\end{equation}
\end{lemma}

This lemma will be proven afterwards. Writing \eqref{eq:dcpDNLSI:1} with \eqref{eq:prfDNLSI:ends} and \eqref{eq:lem:prfDNLSI}, we have
\begin{equation}\begin{aligned}
&\Pbbzero(X_k\geq \tilde f(k), \forall\, 0<k<N\,;\gL_N=(qN,pN))\\
&\geq\sumtwo{(A_1/\uN,x_1)\in Q_{1,\uN},}{(A_2/\uN,x_2)\in Q_{2,\uN}} e^{-\gd_1 x_1 + \uN \cL(\gd_1)} e^{-\gd_2 x_2 + \uN \cL(\gd_2)} e^{-N' \bh_{N'}^{q',p'}\cdot\,(q',p') + \cL_{\gL_{N'}}(\bh_{N'}^{q',p'})} \\
&\qquad \times \Big(\Pb_{N',\bh_{N'}^{q',p'}}(\gL_{N'}=(q'N',p'N'))-C\, e^{-c \,\uN}\Big)\\
&\qquad \times \tilde\Pb_{\gd_1}(X_k\geq \tilde f(k), \forall\, 0<k<\uN\,;\gL_\uN=(A_1/\uN,x_1))\\
&\qquad \times \tilde\Pb_{\gd_2}(X_k\geq \tilde f(N-k)-pN, \forall\, 0<k<\uN\,;\gL_\uN=(A_2/\uN,x_2))\;.
\end{aligned}\end{equation}
Then we divide by $\Pbbzero(\gL_N=(qN,pN))$, which we express with the tilt $\Pb_{N,\bh_{N}^{q',p'}}$ (rather than $\Pb_{N,\bh_{N}^{q,p}}$), so we have
\begin{equation}\label{eq:dcpDNLSI:2}\begin{aligned}
&\Pbbzero(X_k\geq \tilde f(k), \forall\, 0<k<N\,|\,\gL_N=(qN,pN))\\
&\geq\sumtwo{(A_1/\uN,x_1)\in Q_{1,\uN},}{(A_2/\uN,x_2)\in Q_{2,\uN}} e^{g(x_1,a_1,x_2,a_2)} \left(\frac{\Pb_{N',\bh_{N'}^{q',p'}}(\gL_{N'}=(q'N',p'N'))-C\, e^{-c \,\uN}}{\Pb_{N,\bh_{N}^{q',p'}}(\gL_N=(qN,pN))}\right) \\
&\qquad \times \tilde\Pb_{\gd_1}(X_k\geq \tilde f(k), \forall\, 0<k<\uN\,;\gL_\uN=(A_1/\uN,x_1))\\
&\qquad \times \tilde\Pb_{\gd_2}(X_k\geq \tilde f(N-k)-pN, \forall\, 0<k<\uN\,;\gL_\uN=(A_2/\uN,x_2))\;,
\end{aligned}\end{equation}
where we define
\begin{equation}\label{def:gDNLSI}\begin{aligned}
&g(x_1,a_1,x_2,a_2) := -\gd_1 x_1 + \uN \cL(\gd_1) -\gd_2 x_2 + \uN \cL(\gd_2)-N' \bh_{N'}^{q',p'}\cdot(q',p')\\
&\qquad \qquad + \cL_{\gL_{N'}}(\bh_{N'}^{q',p'}) +N \bh_{N}^{q',p'}\cdot(q,p) - \cL_{\gL_{N}}(\bh_{N}^{q',p'})\;.
\end{aligned}\end{equation}

Let us handle the factor in parenthesis in \eqref{eq:dcpDNLSI:2} first. We can apply \eqref{prop:Pnhnq} to both probabilities (notice that $q-q'$, $p-p'$ are of order $\uN/N\ll N^{-1/2}$, recall \eqref{eq:prfDNLSI:qq'pp'}) to bound it from below by $\big(C_\cntc \frac{N^2}{N'^2}-C_\cntc N^2e^{-c \uN}\big)\geq C_\cntc>0$ for $N$ sufficiently large, uniformly in $q,p$.

Now let us focus on $g(x_1,a_1,x_2,a_2)$, where we fix $\gd_1=\tilde h^{q',p'}_0+\tilde h^{q',p'}_1$ and $\gd_2=-\tilde h^{q',p'}_1$. Notice that they match respectively the value of the exponential tilt applied on the first and last increments in the second piece of trajectory and in the limit $N\to\infty$ (recall \eqref{eq:tilth:dvlpmt}, and the sign of $\gd_2$ is inverted because we reversed in time the third bit of the trajectory). We introduce in \eqref{def:gDNLSI} a term $N \bh_{N}^{q',p'}\cdot(q',p')$ and we apply \eqref{prop:cLgLn} for $(q',p',N')$ and $(q',p',N)$ to bound from below
\begin{equation}\begin{aligned}
&g(x_1,a_1,x_2,a_2) \geq -\gd_1 x_1 + \uN \cL(\gd_1) -\gd_2 x_2 + \uN \cL(\gd_2)+N\bh^{q',p'}_N \cdot (q-q',p-p') \\
&\qquad \qquad-N' \Big(\tilde \bh^{q',p'}\cdot(q',p') - \cL_{\gL}(\tilde\bh^{q',p'})\Big) +N \Big(\tilde \bh^{q',p'}\cdot(q',p') - \cL_{\gL}(\tilde\bh^{q',p'})\Big) - C_\cntc \; ,
\end{aligned}\end{equation}
for some uniform constant $C_{\arabic{cst}}<\infty$. Recall $N'=N-2\uN$, and apply \eqref{prop:tildehhnq} to the term $N\bh^{q',p'}_N \cdot (q-q',p-p')$ to obtain
\begin{equation}\begin{aligned}
&g(x_1,a_1,x_2,a_2) \geq -\gd_1 x_1 + \uN \cL(\gd_1) -\gd_2 x_2 + \uN \cL(\gd_2)+N\tilde \bh^{q',p'} \cdot (q-q',p-p')\\
&\qquad \qquad+2\uN \Big(\tilde \bh^{q',p'}\cdot(q',p') - \cL_{\gL}(\tilde\bh^{q',p'})\Big)  - C_\cntc\; .
\end{aligned}\end{equation}
Recalling that $\gd_1=\tilde h^{q',p'}_0+\tilde h^{q',p'}_1$ and $\gd_2=-\tilde h^{q',p'}_1$, and also \eqref{eq:prfDNLSI:qq'pp'}, we have
\begin{equation}\begin{aligned}
&g(x_1,a_1,x_2,a_2) \geq -(\tilde h^{q',p'}_0+\tilde h^{q',p'}_1) x_1 + h^{q',p'}_1x_2\\
&\qquad \qquad+\uN \Big(2\tilde \bh^{q',p'}\cdot(q',p') - 2\cL_{\gL}(\tilde\bh^{q',p'})+\cL(\tilde h^{q',p'}_0+\tilde h^{q',p'}_1) + \cL(-\tilde h^{q',p'}_1)\Big) \\
&\qquad \qquad + \tilde h^{q',p'}_0 \big((p'-4q')\uN +x_1\big) + \tilde h^{q',p'}_1 \big(-2p'\uN +x_1-x_2\big)- C_\cntc\;.
\end{aligned}\end{equation} So,
\begin{equation}\begin{aligned}
g(x_1,a_1,x_2,a_2) \;&\geq\; \uN \Big(-2\tilde h^{q',p'}_0 q' + \tilde h^{q',p'}_0 p' - 2\cL_{\gL}(\tilde\bh^{q',p'})\\
&\qquad +\cL(\tilde h^{q',p'}_0+\tilde h^{q',p'}_1) + \cL(-\tilde h^{q',p'}_1)\Big)-C_{\arabic{cst}}\;,
\end{aligned}\end{equation}
and we conclude by recalling the definition of $\tilde \bh^{q',p'}$ and \eqref{def:tildebh:bis}: the first term is zero, and $g(x_1,a_1,x_2,a_2)$ is bounded from below uniformly by some constant.

Therefore, we can bound \eqref{eq:dcpDNLSI:2} from below by
\begin{equation}\label{eq:dcpDNLSI:3}\begin{aligned}
&\Pbbzero(X_k\geq \tilde f(k), \forall\, 0<k<N\,|\,\gL_N=(qN,pN))\\
&\qquad \geq C_\cntc \times \tilde\Pb_{\gd_1}\big(X_k\geq \tilde f(k), \forall\, 0<k<\uN\,;\gL_\uN\in Q_{1,\uN}\big) \\
&\qquad\quad \times \tilde\Pb_{\gd_2}\big(X_k\geq \tilde f(N-k)-pN, \forall\, 0<k<\uN\,;\gL_\uN\in Q_{2,\uN}\big)\;,
\end{aligned}\end{equation}
with $\gd_1=\tilde h^{q',p'}_0+\tilde h^{q',p'}_1$, $\gd_2=-\tilde h^{q',p'}_1$, and $C_{\arabic{cst}}$ is some uniform positive constant. The proof will be over once we show that those factors are uniformly bounded away from 0. We will focus on the first factor, because we notice in the second one $\tilde f(N-k)-pN = Nf(\frac kN) -k\pt p + K_S$, which matches $\tilde f(k)$ with an inverted slope $-p$ (recall we reversed in time that bit); hence it is handled by the exact same arguments.

Recall that $\nabla \cL_\gL$ is a $\cC^1$ diffeomorphism, in particular $\tilde\bh$ is Lipschitz on compact sets, and $(q-q')$, $(p-p')$ are of order $\uN/N=N^{-2/3}$ (recall \eqref{eq:prfDNLSI:qq'pp'}), so there exists some constant $C_\cntc>0$ such that for all $q\in[q_1,q_2]$, $p\in[p_1,p_2]$ and $N\in\N$,
\begin{equation}\label{eq:prfDNLSI:tilt}
\big|\,\gd_1 - \tilde h_0^{p,q}- \tilde h_1^{p,q}\,\big| \;\leq\; \frac{C_{\arabic{cst}}\pt\uN}{N}\;.
\end{equation}
In particular for $N$ large enough, $\gd_1$ is contained in some compact set $[\ul \gd, \ol\gd]\subset(-\gb/2,\gb/2)$ uniformly in $q,p$. 
As claimed earlier we have $\tilde h_0^{q,p}>0$ under our assumptions ---it is even bounded away from 0--- so we also have $\tilde h_0^{q',p'}>0$ for $N$ large (recall \eqref{eq:prfDNLSI:qq'pp'} again and $\tilde \bh$ is locally Lipschitz). Combined with \eqref{def:tildebh:bis} and with the strict convexity of $\cL$, this implies
\begin{equation}\label{eq:ineq:pcL'}
\cL'(\tilde h_1^{q',p'})\:<\:p'\:<\:\cL'(\tilde h_0^{q',p'} + \tilde h_1^{q',p'})\;.
\end{equation}
In particular there is some constant $C_\cntc>0$ such that $\cL'(\gd_1)-p'\geq C_{\arabic{cst}}$ and $p'+\cL'(\gd_2)\geq C_{\arabic{cst}}$ uniformly in $q\in[q_1,q_2]$, $p\in[p_1,p_2]$ and $N\in\N$ (recall that those functions are continuous, and $\cL'(\tilde h^{q',p'}_1)=-\cL'(-\tilde h^{q',p'}_1)$).

Fix some $a,C>0$. For any $\gd$ such that $|\gd|<\gb/2$, and any $0<t<\gb/2+\gd$, $k\in\N$, Markov's inequality implies that
\begin{equation}
\tilde \Pb_\gd(X_k\leq ak-C) \;\leq\; \tilde \bE_\gd\big[e^{t(ak-C)-tX_k}\big] \;=\; e^{-Ct}\, e^{k(\tilde\cL_\gd(-t)+at)}\;,
\end{equation}
where $\tilde\cL_\gd(-t):=\log\tilde\bE_\gd[e^{-tX_1}]=\cL(\gd-t)-\cL(\gd)$ for any $t\in(-\gb/2+\gd,\gb/2+\gd)$. Choosing $a=\frac34\cL'(\gd_1)+\frac14p'$ and writing a Taylor expansion of $\cL$, there is some (uniform) $c>0$ such that for $t$ small,
\begin{equation}
\tilde\cL_{\gd_1}(-t)+at \;\leq\; \frac t4 \big(p'-\cL'(\gd_1)\big) +ct^2\;.
\end{equation}
Recall that $p'-\cL'(\gd_1)<0$, and it is even bounded away from 0 (uniformly). So for $t$ sufficiently small, we have $\tilde\cL_{\gd_1}(-t)+at \;\leq\; \frac t8 \big(p'-\cL'(\gd_1)\big)$. 
Hence there exists a $t_0>0$ such that,
\begin{equation}\label{eq:prfDNLSI:tilt1}\begin{aligned}
\tilde \Pb_{\gd_1} \big(\exists\, k\geq1, X_k\leq ak-C\big) \;&\leq\; \sum_{k\geq1} \tilde \Pb_{\gd_1}(X_k\leq ak-C) \\
&\leq\; e^{-Ct_0} \frac{e^{t_0(p'-\cL'(\gd_1))/8}}{1-e^{t_0(p'-\cL'(\gd_1))/8}} \;\leq\; C_\cntc\, e^{-Ct_0} \; ,
\end{aligned}\end{equation}
where we bounded the fraction by some constant uniform in $(q',p')$.

Similarly, we have
\begin{equation}
\tilde \Pb_\gd(X_k\geq bk+C) \;\leq\; \tilde \bE_\gd\big[e^{wX_k-w(bk+C)}\big] \;=\; e^{-Cw}e^{k(\tilde\cL_\gd(w)-bw)}\;,
\end{equation}
for some $b,C>0$ fixed, and any $k\in\N$, $\gd$ such that $|\gd|<\gb/2$, and $0<w<\gb/2-\gd$. Because $\tilde\cL_\gd(w)\sim  \cL'(\gd)w$ as $w\to0$, we fix some $b>\sup\{\cL'(\gd),\,\gd\in[\ul \gd, \ol \gd] \}$, and for some $w_0>0$ sufficiently small we have
\begin{equation}\label{eq:prfDNLSI:tilt2}
\tilde \Pb_{\gd_1} \big(\exists\, k\geq1, X_k\geq bk+C\big) \;\leq\; e^{-Cw_0} \frac{e^{\tilde\cL_{\gd_1}(w_0)-bw_0}}{1-e^{\tilde\cL_{\gd_1}(w_0)-bw_0}}
\;\leq\; C_\cntc e^{-C w_0} \;,
\end{equation}
where we bounded the fraction by some uniform constant (notice that we could pick $w_0>0$ such that $\tilde\cL_\gd(w_0)-bw_0<0$ for all $\gd\in[\ul \gd, \ol \gd]$).

Finally, we estimate the first factor in \eqref{eq:dcpDNLSI:3} by constraining the trajectory $X$ to have a first step $X_1=2K+K_S$, then to remain between the two curves $a(k-1)+K+K_S$ and $b(k-1)+3K+K_S$, $k\geq2$. Moreover for $\gga>0$ sufficiently small, we have
\begin{equation}\label{eq:prfDNLSI:apgga}
a-(p+\gga) \;=\; \frac34(\cL'(\gd_1)-p') -\gga + (p'-p) \;>\;0\;,
\end{equation}
(recall that $p'-p=O(N^{-2/3})$ and $\cL'(\gd_1)-p'$ is uniformly bounded away from 0), and this $\gga$ can be chosen independently of $K_S$. This means that our constraint implies that $X_k$ remains above $\tilde f(k)=\gga \pt k + k\pt p + K_S$ for all $1\leq k \leq \uN$. It also implies $\gL_\uN\in Q_{1,\uN}$ by setting $a_1=a=\frac34\cL'(\gd_1)+\frac14p'$ and $b_1=b$ in the definition of $Q_{1,\uN}$. Recollecting \eqref{eq:prfDNLSI:tilt1} and \eqref{eq:prfDNLSI:tilt2}, we finally obtain
\begin{align}\label{eq:prfDNLSI:tilt3}
\nonumber &\tilde\Pb_{\gd_1}(X_k\geq \tilde f(k), \forall\, 0<k<\uN\,;\gL_\uN\in Q_{1,\uN} ) \\
\nonumber &\quad \geq \tilde\Pb_{\gd_1}\big(X_1=2K+K_S\,;\, a(k-1)+K+K_S\leq X_k \leq b(k-1)+3K+K_S, \forall\, k\geq 2\big)\\  
\nonumber &\quad = \tilde\Pb_{\gd_1}\big(X_1=2K+K_S\big) \times \tilde\Pb_{\gd_1}\big(ak-K\leq X_k \leq bk+K, \forall\, k\geq 1\big)\\
&\quad \geq \tilde\Pb_{\gd_1}\big(X_1=2K+K_S\big) \times \big(1- C_\cntc e^{-K\min(t_0,w_0)}\big)
 \;\geq\; C_\cntc\;>\;0\;.
\end{align}
where we choose $K>0$ such that the second factor in \eqref{eq:prfDNLSI:tilt3} is greater than $\frac12$, then bound $\tilde \Pb_{\gd_1}(X_1=2K+K_S)$ uniformly over $\gd_1\in[\ul \gd, \ol\gd]$. We can reproduce the same proof to handle the other factor in \eqref{eq:dcpDNLSI:3}, by setting $a_2=\frac34\cL'(\gd_2)-\frac14p'$ and a convenient $b_2>a_2$ in the definition of $Q_{2,\uN}$. Therefore \eqref{eq:dcpDNLSI:3} is bounded from below by some uniform positive constant, and this concludes the proof of the proposition (subject to Lemma~\ref{lem:prfDNLSI}).
\end{proof} 

\smallskip 
Let us now prove Lemma~\ref{lem:prfDNLSI}. This result relies on estimates that are  similar to those displayed at the end of the proof of Proposition~\ref{prop:DNLSI}, where we used Markov inequalities for $\tilde\Pb_\gd$. First we prove an upper bound on the moment generating function of the difference between the random walk with law $\bP_{N',\bh_{N'}^{q,p}}$ and the straight line with slope $p$.
\begin{lemma}\label{lem:prfDNLSI:markov} There exists $\gl_0>0$ such that for any $\gl\in(0,\gl_0)$, there are $c>0$ and $N_0\in\N$ such that
\begin{equation}
\bE_{N,\bh_{N}^{q,p}}\Big[e^{-\gl (X_k-pk)}\Big]\;\leq\; e^{-ck} \;,
\end{equation}
uniformly in $0\leq k\leq N$, $N\geq N_0$, $q\in[q_1,q_2]$, $p\in[p_1,p_2]$ satisfying $p\leq 2q-\eps$.
\end{lemma}
\begin{proof}
Under the law $\bP_{N,\bh_{N}^{q,p}}$, the increments of the random walk $X$ are still independent (but no more identically distributed), so we define $h_N^i:=(1-i/N)h_{N,0}^{q,p}+h_{N,1}^{q,p}$, the tilt parameter for the increment $X_i-X_{i-1}$, $i\geq1$. Then we have
\begin{equation}\label{eq:lemprfDNLSImarkov:1}\begin{aligned}
\log \bE_{N,\bh_{N}^{q,p}}\big[e^{-\gl (X_k-pk)}\big] \;&=\; \gl\pt p\pt k +\sum_{i=1}^k \big(\cL(-\gl+h_N^i)-\cL(h_N^i)\big)\\
&=\; -\gl \sum_{i=1}^k \bigg(\frac{\cL(h_N^i) - \cL(-\gl+h_N^i)}{\gl} -p\bigg)\\
&\leq\; -\gl \sum_{i=1}^k \Big(\cL'(-\gl+h_N^i) -p\Big)\;,
\end{aligned}\end{equation}
where we used the convexity of $\cL$, and the r.h.s. is well-defined for all $|\gl|<\gl_0$, for some $\gl_0>0$ uniform in $(q,p)$ (recall that $\bh_{N}^{q,p}$ is contained in a compact subset of $\cD_{\gb,n}$). Notice that $h_N^i$ is affine, decreasing in $i$ (recall that $p\leq 2q-\eps$ implies $\tilde h_0^{q,p}>0$), and recall that $\cL'$ is increasing. Therefore it suffices to prove the claim for $k=N$, that is to prove
\begin{equation}\label{eq:lemprfDNLSImarkov:2}
\frac 1N \sum_{i=1}^N \cL'(-\gl+h_N^i) \,-\, p\;\geq\;c\,>\,0\;,
\end{equation}
and we conclude by noticing that if $(u_n)_{n\geq 1}$ is a decreasing sequence, then $(\frac 1n \sum_{i=1}^nu_i)_{n\geq 1}$ is decreasing too, hence the claim holds for $k\leq N$.

Notice that the first term in the l.h.s. of \eqref{eq:lemprfDNLSImarkov:2} is a Riemann sum, and we have
\begin{equation}\label{eq:lemprfDNLSImarkov:3}
\frac 1N \sum_{i=1}^N \cL'(-\gl+h_N^i) \underset{N\to\infty}{\longrightarrow} \int_0^1 \cL'\big(x\tilde h_0^{q,p} + \tilde h_1^{q,p} - \gl\big)\dd x\;.
\end{equation}
We also claim that the convergence holds uniformly in $q\in[q_1,q_2]$, $p\in[p_1,p_2]$ ---this follows from a Riemann-sum approximation of the r.h.s., from \eqref{prop:tildehhnq} and from the fact that $\cL'$ is locally Lipschitz. Moreover we have $p=\int_0^1 \cL'(x\tilde h_0^{q,p} + \tilde h_1^{q,p})\dd x$ (recall \eqref{def:tildebh}), so this concludes the proof of \eqref{eq:lemprfDNLSImarkov:2} for $N$ sufficiently large and some $c>0$ uniform in $q\in[q_1,q_2]$, $p\in[p_1,p_2]$.
\end{proof}

\begin{proof}[Proof of Lemma~\ref{lem:prfDNLSI}]
Let us bound from above
\[
\Pbbzero\big(X_k < \tilde f(k+\uN)-x_1\, ;\, \gL_{N'}=(q'N',p'N')\big)\;,
\]
for any $0< k<N'$, where we recall $x_1\in[a_1(\uN-1) + K + K_S, b_1(\uN-1) + 3K + K_S]$, $\uN=N^{1/3}$, $N'=N-2\uN$ and $a_1$ satisfies \eqref{eq:prfDNLSI:apgga} with $a=a_1$. We have
\begin{equation}\begin{aligned}
&\Pbbzero\big(X_k < \tilde f(k+\uN)-x_1\, ;\, \gL_{N'}=(q'N',p'N')\big)\\ &\quad\leq\; \Pbbzero\big(X_k < -\gz\uN+p\pt k + \gga (k\wedge(N'-k)) -K + a_1\, ;\, \gL_{N'}=(q'N',p'N')\big)\;,
\end{aligned}\end{equation}
where we define $\gz := a_1-p-\gga$ (which is positive and bounded away from 0 uniformly in $p$). Recalling \eqref{def:Pnh} and applying Markov's inequality, we have
\begin{align}
\nonumber&\Pbbzero\big(X_k < \tilde f(k+\uN)-x_1\, ;\, \gL_{N'}=(q'N',p'N')\big) \\
\nonumber&\quad \leq\; e^{-N' \bh_{N'}^{q',p'}\cdot\,(q',p') + \cL_{\gL_{N'}}(\bh_{N'}^{q',p'})} \\
\nonumber&\qquad \times \Pb_{N', \bh_{N'}^{q',p'}}\big(X_k < -\gz\uN+p\pt k + \gga (k\wedge(N'-k)) -K + a_1\, ;\, \gL_{N'}=(q'N',p'N')\big)\\
\nonumber&\quad \leq\; e^{-N' \bh_{N'}^{q',p'}\cdot\,(q',p') + \cL_{\gL_{N'}}(\bh_{N'}^{q',p'})} \\
&\qquad \times e^{-\gl\gz\uN + \gl\gga \,(k \wedge (N'-k))+\gl(a_1-K)}\, \bE_{N', \bh_{N'}^{q',p'}}\big[e^{- \gl(X_k-p'k)} \big]e^{\gl(p-p')k}\;,
\end{align}
for some $\gl\in(0,\gl_0)$. Applying Lemma~\ref{lem:prfDNLSI:markov} and the inequality $k\wedge (N'-k)\leq k$, we obtain
\begin{equation}\begin{aligned}
&\Pbbzero\big(X_k < \tilde f(k+\uN)-x_1\, ;\, \gL_{N'}=(q'N',p'N')\big) \\
&\quad \leq \; e^{-N' \bh_{N'}^{q',p'}\cdot\,(q',p') + \cL_{\gL_{N'}}(\bh_{N'}^{q',p'})} e^{-\gl\gz\uN + \gl\gga k+\gl(a_1-K)} e^{-c\pt k} e^{\gl(p-p')k}\;.
\end{aligned}\end{equation}
Choosing $\gga<c/\gl$ (independently of $K_S$), and recalling that $p-p'$ goes uniformly to 0 (see \eqref{eq:prfDNLSI:qq'pp'}), there is $N_0\in\N$ such that $\gl\gga-c+\gl(p-p')\leq-C_\cntc<0$ uniformly in $(q,p)$ and $N\geq N_0$. Therefore we have
\begin{align}
\nonumber&\Pbbzero\big(X_k\geq \tilde f(k+\uN)-x_1, \forall\, 0<k<N'\,;\gL_{N'}=(q'N',p'N')\big) \\
\nonumber&\quad \geq\; \Pbbzero\big(\gL_{N'}=(q'N',p'N')) - \sum_{k=1}^{N'} \Pbbzero(X_k< \tilde f(k+\uN)-x_1\,;\, \gL_{N'}=(q'N',p'N')\big)\\
\nonumber&\quad \geq\; e^{-N' \bh_{N'}^{q',p'}\cdot\,(q',p') + \cL_{\gL_{N'}}(\bh_{N'}^{q',p'})} \\
& \qquad \times \Big(\Pb_{N',\bh_{N'}^{q',p'}}(\gL_{N'}=(q'N',p'N'))- e^{-\gl\gz \uN}e^{\gl(a_1-K)}\sum_{k=1}^{N'} e^{-C_{\arabic{cst}}k}
\Big)\;,
\end{align}
and $\sum_{k\geq1} e^{-C_{\arabic{cst}}k}<\infty$, which concludes the proof.
\end{proof} 

\appendix
\section{The wetting model}\label{app:wetting}
In this appendix we provide well-known results and estimates on the wetting model which are proven in \cite{Giac07}. Recall that we defined $\Pbb$ in \eqref{def:Pgb}, and for all $N\geq1$, let
\begin{equation}\label{def:Zwetting}
Z_{\wet,N}^{\gb,\gd} \;:=\; \Ebbzero\left[e^{\gd\sum_{k=1}^N\ind_{\{X_k=0\}}}\ind_{\{X\in B^{0,+}_N\}}\right]\;,
\end{equation}
be the (constrained) partition function of a wetting model associated with a random walk of law $\Pbb$. Let $h_\gb(\gd)$ be its free energy ---recall \eqref{wetmod}, and notice that it is well-defined because $(Z_{\wet,N}^{\gb,\gd})_{N\geq1}$ is super-multiplicative. Note that $h_\gb(\gd)$ is non-decreasing in $\gd$, that $h_\gb(\gd)\leq\gd$, and recalling \eqref{eq:asymp:PosBB}, we have $h_\gb(0)=0$; hence $h_\gb(\gd)\geq0$ for all $\gd\geq0$, and we define $\tilde\gd_c(\gb):=\inf\{\gd\geq0,h_\gb(\gd)>0\}\in[0,\infty]$.

Let $\tau:=\inf\{t\geq1; X_t\leq0\}$ and let $K_\gb(t):=\Pbbzero(\tau=t,X_t=0)$, $t\geq1$. 
Notice that conditionally to $\tau$, $-X_\tau$ follows a geometric distribution on $\N_0$ (in particular $\tau$ and $X_\tau$ are independent). Therefore one has $\sum_{t\geq1} K_\gb(t)=\Pbbzero(X_\tau=0)=1-e^{-\gb/2}$, and $K_\gb(t) = \Pbbzero(\tau=t)(1-e^{-\gb/2})$.
It is a  straightforward application of \cite[(4.5)]{CC13}
that there exists $c>0$ depending on $\beta$ only such that 
\begin{equation}\label{equivexc}
K_\beta(t)\,\sim\, \frac{c}{t^{3/2}}\qquad\text{as}\quad t\to \infty\;.
\end{equation}  

\begin{prop}[{\cite[Thm.~2.2 \& Eq. 2.2]{Giac07}}]\label{prop:wetpoint}~

(i) For $\gb>0$, the free energy of this wetting model $h_\gb(\gd)$ is the only solution in $\gz\geq0$ of the equation
\begin{equation}\label{eq:prop:wetpoint}
\sum_{t\geq1} K_\gb(t)e^{-\gz t} \;=\; e^{-\gd}\;,
\end{equation}
if it exists (that is if $\gd>\tilde \gd_c(\gb)=-\log (1-e^{-\gb/2})$), and $0$ otherwise. An explicit expression of $h_\gb(\gd)$ is given by
\begin{equation}\label{eq:explicit:h}
h_\gb(\gd) \;=\;  \log\bigg( \frac{(e^\gd-1)(1-e^{-\gb/2})^2}{1-e^{-\gd}-e^{-\gb}} \bigg) \;,\qquad \text{for}\quad \gd > -\log (1-e^{-\gb/2})\;,
\end{equation}
and $0$ otherwise.

(ii) The asymptotics of the partition function are given by:\\
($a$) if $\gd>\tilde \gd_c(\gb)$, then as $N\to\infty$, \[Z_{\wet, N}^{\gb,\gd} \,\sim\, C_\wet e^{h_\gb(\gd) N}\qquad \text{where}\quad C_\wet\,:=\,\bigg(e^{\gd} \sum_{t\geq1} t\pt K_\gb(t) e^{-h_\gb(\gd)t}\bigg)^{-1}\, ;\]\\
($b$) if $\gd < \tilde \gd_c(\gb)$, then as $N\to\infty$
\[ Z_{\wet, N}^{\gb,\gd} \,\sim\, \frac{c\, e^{-\gd}}{\big(e^{-\gd}-1+e^{-\gb/2}\big)^2} N^{-3/2} \,;\]
($c$) if $\gd = \tilde \gd_c(\gb)$, then as $N\to\infty$,
\[ Z_{\wet, N}^{\gb,\pt \tilde \gd_c(\gb)} \,\sim\, \frac{1-e^{-\gb/2}}{2\pt\pi \pt c } \pt N^{-1/2} \,.\]
\end{prop}

Recall \eqref{equivexc} for the definition of $c$. The first claim follows from \cite[(2.2)]{Giac07} and above observations regarding $K_\gb$. The explicit formula for $h_\gb(\gd)$ is obtained by computing
\[
\sum_{t\geq1} K_\gb(t)e^{-\gz t} =(1-e^{-\beta/2})\,  \Ebbzero[e^{-\gz\tau}] =1-e^a\, e^{-\beta/2}\,,
\]
for some $a>0$, thanks to a stopping time argument and with the observation that $(e^{ -a  X_n - \gz n})_{n\geq 0}$ is a martingale when $a>0$ is the unique solution of $\Ebbzero[e^{-a X_1-\gz}]=1$; details are left to the reader.
 The asymptotics of $Z_{\wet, N}^{\gb,\gd}$ as $N\to\infty$ are given in \cite[Thm.~2.2]{Giac07}, where we recall \eqref{equivexc}.

\section{An {\rm FKG} inequality}\label{app:FKG}
We provide here a (conditional) {\rm FKG} inequality for random walks of length $N\in\N$ with increment distributed as $\Pbb$, similarly to more general, already well-known {\rm FKG} inequalities as developed in \cite{P74}. Define a partial order on $\bbZ^{N}$: for all $u,v\in\bbZ^N$, $u\preceq v$ if and only if $u_k\leq v_k$ for all $1\leq k\leq N$. A function $f:\bbZ^N\to\R$ is \emph{non-decreasing} (resp. \emph{non-increasing}) if for all $u,v\in\bbZ^N$, $u\preceq v$ implies $f(u)\leq f(v)$ (resp. $f(u)\geq f(v)$). Define also for every $u,v\in\bbZ^{N}$,
\begin{equation}\begin{aligned}
u\wedge v &:= (\min(u_1,v_1),\ldots, \min(u_N,v_N))\,,\\
u\vee v &:= (\max(u_1,v_1),\ldots, \max(u_N,v_N))\,.
\end{aligned}\end{equation}
\begin{prop}\label{prop:FKG}
Let $A\subseteq \bbZ^N$ be such that \begin{itemize}
\item $\Pbbzero(A)>0$,
\item for any $u,v\in A$, one has $u\wedge v\in A$ and $u\vee v \in A$.
\end{itemize}
Let $f,g : A\to \R$ be two non-increasing (or two non-decreasing), non-negative and bounded functions. 
Then,
\begin{equation}
\Ebbzero\big[f(X)\pt g(X)\pt\big|\pt X\in A\big] \;\geq\; \Ebbzero\big[f(X)\pt\big|\pt X\in A\big] \pt \Ebbzero\big[g(X)\pt\big|\pt X\in A\big]\;.
\end{equation}
\end{prop}
Note that this claim holds in particular for $A=B_N^{0,+}$ or $A=\bbZ^N$. 
\begin{proof}
We prove the proposition for $f$, $g$ both non-increasing. Let $u,v\in A$. For every subset $B\subseteq A$, let us define $\mu_1(B):=\Pbbzero(B\pt|\pt A)$, and 
\begin{equation*}
\mu_2(B):= \frac{\Ebbzero\big[g(X)\ind_{B}(X) \,\big|\,X\in A\big]}{\Ebbzero\big[g(X)\,\big|\,X\in A\big]}\;,
\end{equation*}
where we assume without loss of generality that $\Ebbzero\big[g(X)\,\big|\,X\in A\big]>0$ ---otherwise $g\equiv 0$ on $A$ and the proposition is trivial. Notice that $\mu_1$ and $\mu_2$ are probability measures on $A$. To apply an {\rm FKG} inequality from \cite{P74}, we must prove that for any $u,v\in A$,
\begin{equation}\label{eq:FKG:1}
\mu_1(u\vee v) \, \mu_2(u\wedge v) \geq \mu_1(u) \, \mu_2(v)\;.
\end{equation}
To that end we multiply the left hand side in \eqref{eq:FKG:1} by $\Ebbzero[g(X)|A]\, \times \Pbbzero(A)^2$, to obtain
\begin{equation}\label{eq:FKG:1.5}
\begin{aligned}
&\Ebbzero\big[g(X)\big|A\big]\, \Pbbzero(A)^2 \mu_1(u\vee v) \, \mu_2(u\wedge v) \\
& \quad = g(u\wedge v) \Pbbzero(X=u\vee v) \Pbbzero(X=u\wedge v) \\
& \quad = g(u\wedge v)\pt c_\gb^{-2N} \prod_{k=1}^N e^{-\frac{\gb}{2}|(u_k\vee v_k) - (u_{k-1}\vee v_{k-1})| }  \prod_{k=1}^N e^{-\frac{\gb}{2}|(u_k\wedge v_k) - (u_{k-1}\wedge v_{k-1})| }\:,
\end{aligned}\end{equation}
where $u_0=v_0:=0$. Moreover we claim that for any $a,b,c,d\in\R$,
\begin{equation}\label{calcwe}
|a\vee b - c\vee d| + |a\wedge b - c\wedge d| \leq |a-c| + |b-d|\;.
\end{equation}
To prove this, observe that $|a-c|=a+c-2(a\wedge c)$ for any $a,c\in\R$, hence \eqref{calcwe} is equivalent to
\[a\wedge c + b\wedge d \;\leq\; (a\vee b)\wedge (c\vee d) + (a\wedge b)\wedge (c\wedge d)\;.\]
Notice that $(a\wedge b)\wedge (c\wedge d) =\min\{a,b,c,d\}$ is either equal to $a\wedge c$ or $b\wedge d$. Assume $a\wedge b\wedge c\wedge d=a\wedge c$ without loss of generality. Then one obviously has $b\wedge d \leq (a\vee b)\wedge (c\vee d)$, which concludes the proof of \eqref{calcwe}.

Using \eqref{calcwe} in \eqref{eq:FKG:1.5} and recalling that $g$ is non-increasing, we conclude
\begin{equation*}\begin{aligned}
&\Ebbzero\big[g(X)\big|A\big]\, \Pbbzero(A)^2 \mu_1(u\vee v) \, \mu_2(u\wedge v) \\
& \quad \geq g(v)\, c_\gb^{-2N} \prod_{k=1}^N e^{-\frac{\gb}{2}|u_k - u_{k-1}| }  \prod_{k=1}^N e^{-\frac{\gb}{2}|v_k -v_{k-1}| }\\
& \quad \geq \Ebbzero\big[g(X)\big|A\big]\, \Pbbzero\big(A\big)^2 \mu_1(u) \, \mu_2(v)\;,
\end{aligned}\end{equation*}
which eventually proves \eqref{eq:FKG:1}.

Since $f$ is also a non-increasing function on $A$, we can use a generalized {\rm FKG} inequality claimed in \cite[Theorem~3]{P74} to obtain 
\begin{equation}
\int_{A} f \pt \dd \mu_1 \; \leq \int_{A} f \pt \dd \mu_2\;,
\end{equation}
which concludes the proof.
\end{proof}

\bibliographystyle{plain}
\bibliography{biblio.bib}

\end{document}